\documentclass[11pt]{amsart} 
\usepackage[english]{babel}
\usepackage{graphicx,epsfig}
\usepackage{parskip}
\usepackage{bbm}
\usepackage{listings}
\usepackage{amsthm}
\usepackage{comment}
\usepackage{enumerate}
\usepackage{esvect}
\usepackage{amsmath,amssymb,latexsym, amsfonts, amscd, amsthm, xy}
\input{xy}
\xyoption{all}

\usepackage{mathrsfs}
\usepackage{appendix}

\usepackage{mathtools}

\usepackage[margin=1in]{geometry}

\usepackage[dvipsnames]{xcolor}

\makeindex 

\usepackage{hyperref}
\hypersetup{pdftoolbar=true, pdftitle={GKform}, pdffitwindow=true, colorlinks=true, citecolor=blue, filecolor=black, linkcolor=purple, urlcolor=red, hypertexnames=false}

\theoremstyle{plain}

\newtheorem{theorem}{Theorem}[section]
\newtheorem{conjecture}[theorem]{Conjecture}
\newtheorem{proposition}[theorem]{Proposition}
\newtheorem{corollary}[theorem]{Corollary}

\newtheorem*{assumption*}{Assumption}
\newtheorem{lemma}[theorem]{Lemma}

\theoremstyle{definition}
\newtheorem{definition}[theorem]{Definition}
\newtheorem{remark}[theorem]{Remark}

\newtheorem*{goal*}{Goal}
\newtheorem*{problem*}{Comment}

\newtheorem{notation}{Notation}[section]

\def\lra{{\longrightarrow \, }}
\def\SL{{\rm SL}}
\def\GL{{\rm GL}}

\def\X{\mathbf{X}}
\def\A{\mathbb{A}}

\DeclareMathOperator{\Gr}{Gr}

\DeclareMathOperator{\cW}{\mathcal{W}}

\DeclareMathOperator{\Vol}{Vol}

\DeclareMathOperator{\Fil}{Fil}

\DeclareMathOperator{\sym}{Sym}

\DeclareMathOperator{\dR}{dR} \DeclareMathOperator{\pr}{pr}
\DeclareMathOperator{\id}{id}

\DeclareMathOperator{\Pet}{Pet}

 \DeclareMathOperator{\cl}{cl}
\DeclareMathOperator{\Cl}{Cl}

\DeclareMathOperator{\et}{et} 
 
\DeclareMathOperator{\Hom}{Hom}

\DeclareMathOperator{\End}{End} \DeclareMathOperator{\CH}{CH}
\DeclareMathOperator{\AJ}{AJ} \DeclareMathOperator{\Aut}{Aut}

\DeclareMathOperator{\spec}{Spec}

\DeclareMathOperator{\ord}{ord}

\DeclareMathOperator{\Disc}{Disc}

\DeclareMathOperator{\cE}{\mathcal{E}}
\DeclareMathOperator{\tcE}{\tilde{\mathcal{E}}}
\DeclareMathOperator{\cV}{\mathcal{V}}
\DeclareMathOperator{\bL}{\mathbb{L}}
\DeclareMathOperator{\cL}{\mathcal{L}}
\DeclareMathOperator{\tW}{\tilde{W}}
\DeclareMathOperator{\tX}{\tilde{X}}

\DeclareMathOperator{\an}{an}
\DeclareMathOperator{\GS}{BB}

\DeclareMathOperator{\new}{new}

\def\fa{\mathfrak{a}}

\def\fc{\mathfrak{c}}

\def\fp{\mathfrak{p}}

\def\fm{\mathfrak{m}}
\def\fn{\mathfrak{n}}

\def\fd{\mathfrak{d}}
\def\fq{\mathfrak{q}}

\def\d{\mathrm{d}}

\def\Z{\mathbb{Z}}

\def\F{\mathbb{F}}

\def\Q{\mathbb{Q}}
\def\W{\mathbf{W}}

\def\C{\mathbb{C}}

\def\R{\mathbb{R}}

\def\bdf{\begin{defn}}
\def\edf{\end{defn}}

\def\cL{\mathcal{L}}
\def\cH{\mathcal{H}}

\def\cO{\mathcal{O}}

\def\cA{\mathcal{A}}
\def\cV{\mathcal{V}}

\def\cW{\mathcal{W}}
\def\cE{\mathcal{E}}

\def\Gal{{\rm Gal}}

\def\corr{{\rm Corr}}

\def\cB{{\mathcal B}}
\def\cC{{\mathcal C}}

\def\d1{d^{(1)}}

\def\d{\mathbf{d}}

\def\Y{\mathbf{Y}}
\def\bZ{\mathbf{Z}}

\def\cA{\mathcal{A}}
\def\fraka{\mathfrak{a}}

\def\SO{{\bf SO}}

\def\oh{\mathcal{O}}

\usepackage{tikz}
\usetikzlibrary{positioning} 
\usepackage{tikz-cd}
\usetikzlibrary{arrows,graphs,decorations.pathmorphing,decorations.markings}

\tikzset{
commutative diagrams/.cd,
arrow style=tikz,
diagrams={>=latex}}

\makeatletter
\let\@wraptoccontribs\wraptoccontribs
\makeatother

\begin{document} 

 \title[Derivatives and heights]{Derivatives of Rankin--Selberg $L$-functions and heights of generalized Heegner cycles}
\author{David T.-B. G. Lilienfeldt}
\address{Mathematical Institute, Leiden University, The Netherlands} 
\email{d.t.b.g.lilienfeldt@math.leidenuniv.nl}
\author{Ari Shnidman}
\address{Einstein Institute of Mathematics, Hebrew University of Jerusalem, Israel}
\email{ariel.shnidman@mail.huji.ac.il}
\date{\today}

\begin{abstract}
Let $f$ be a newform of weight $2k$ and let $\chi$ be an unramified imaginary quadratic Hecke character of infinity type $(2t, 0)$, for some integer $0 < t \leq k-1$.  We show that the central derivative of the Rankin--Selberg $L$-function $L(f,\chi,s)$ is, up to an explicit positive constant, equal  to the Beilinson--Bloch height of a generalized Heegner cycle. This generalizes the Gross--Zagier formula (the case $k = 1$) and Zhang's higher weight formula (the case $t=0$).
\end{abstract}

\maketitle

\tableofcontents

\section{Introduction}

\subsection{Main result}

Let $0<t\leq k-1$ be integers, and let $K$ be an imaginary quadratic field of odd discriminant. We consider the Rankin--Selberg $L$-function $L(f, \chi, s)$ associated to a newform $f\in S_{2k}(\Gamma_0(N))$ and an unramified Hecke character $\chi$ of $K$ of infinity type $(2t,0)$. 
This $L$-series admits analytic continuation and a functional equation with global root number $-\epsilon_K(N)$, where $\epsilon_K$ is the quadratic Dirichlet character associated to $K$. We further impose the Heegner hypothesis:
\begin{equation}\label{HH}
p \mid N \implies p \text{ splits in } K, 
\end{equation}
so that $-\epsilon_K(N) = -1$ and $L(f,\chi,s)$ vanishes to odd order at its center $s=k+t$. 

Our main result (which is also Theorem \ref{thm: main result}) is the following formula for the central derivative in terms of the Beilinson--Bloch height of a certain algebraic cycle $\Delta'_{\chi,f}$:

\begin{theorem}\label{thm:intro}
    With the notations and assumptions as above, we have 
    \[
    L'(f,\chi,k+t)=\frac{4^{k+t}\pi^{2k} (f,f)_{\mathrm{Pet}}}{(k-t-1)!(k+t-1)!  h u^2 \vert D\vert^{k-t-1/2}}  \langle \Delta'_{\chi, f}, \Delta'_{\chi, f}\rangle^{\GS},
    \]
where $h=\# \mathrm{Cl}_K$, $D = \Disc(K)$, $2u=\# \oh_K^\times$, and $(f,f)_{\Pet}$ is the Petersson inner product.
\end{theorem}

The class $\Delta'_{\chi,f}$ is constructed from an algebraic cycle $\Delta_\chi \in \CH^{k+t}(X)_\C,$ of middle arithmetic  codimension on the generalized Kuga--Sato variety $X = W_{2k-2} \times A^{2t}$ defined over the Hilbert class field $H$ of $K$. Here, $W_{2k-2} \lra X(N)$ is the Kuga--Sato variety realizing the motive of $f$ \cite{scholl} and $A$ is an elliptic curve over $H$ with CM by $\oh_K$. The cycles $\Delta_\chi$ are weighted sums of  the {\it generalized Heegner cycles} defined by Bertolini--Darmon--Prasanna \cite{bdp1}, and will be described in more detail in Section \ref{subsec:gen Heeg cycles}. Roughly speaking, they are built out of graphs of isogenies living in CM fibers of $X \lra X(N)$, isogenies which ``mix'' the Kuga--Sato components with the constant components of the fiber.  Let $\CH^i(X)_0 \subset \CH^i(X)$ denote the subgroup of homologically trivial algebraic cycles. The conjectural Beilinson--Bloch height pairing \cite{beilinsonheight, blochheight}
\[
\langle \;\; , \;\; \rangle^{\GS} \colon \CH^{k+t}(X)_0 \times \CH^{k+t}(X)_0 \lra \R
\]
is well-defined on a certain subgroup $G \subset \CH^{k+t}(X)_0$, and $\Delta_\chi \in G \otimes \C$. 
Let $\Delta_{\chi}'$ be the image of $\Delta_\chi$ in the quotient of $\mathrm{Span}\{ T_m\Delta_\chi \colon (m,N)=1 \}$ by the kernel of the height pairing (extended to a $\C$-valued hermitian pairing). The cycle $\Delta'_{\chi, f}$ in Theorem \ref{thm:intro} is the $f$-isotypic component of $\Delta_{\chi}'$ with respect to the Hecke action. 

Theorem \ref{thm:intro} generalizes the formula of Zhang \cite{zhang} (the case $t = 0$), which is itself a generalization of the Gross--Zagier formula  in the case $k = 1$ (and hence $t = 0$) \cite{GZ}. The second author proved a $p$-adic version of Theorem \ref{thm:intro} in \cite{pGZshnidman}, which was subsequently generalized by Disegni in \cite{disegni}. 

As in \cite{GZ} and \cite{zhang}, we deduce the following result about the analytic ranks 
\[r_{\an}(f \otimes \chi) := \ord_{s=k+t} L(f,\chi,s)\] 
of the complex $L$-functions $L(f,\chi,s)$.

\begin{corollary}
For $\sigma \in \Gal_\Q$, we have 
\[r_{\an}(f \otimes \chi) = 1 \iff r_{\an}(f^\sigma \otimes \chi^\sigma) = 1.\] 
\end{corollary}
If the height pairing on generalized Heegner cycles is positive-definite, as is expected, then we also deduce that $L'(f,\chi,k+t)\geq 0$, as is predicted by the Riemann hypothesis for $L(f,\chi,s)$.

\subsection{The Beilinson--Bloch and Bloch--Kato conjectures}
Let $Y$ be a smooth and proper variety over a number field $F$. Taken together, the Beilinson--Bloch and Bloch--Kato conjectures \cite{bloch, blochkato} predict that for each integer $i \geq 0$ and prime $\ell$, the higher Abel--Jacobi map \begin{equation}\label{eq:AJ}
    \AJ_{Y}^{i} \colon \CH^{i}(Y)_0 \otimes \Q_\ell \longrightarrow H^1_f(\Gal_F, H_{\et}^{2i-1}(\overline{Y},\Q_\ell(i)))
    \end{equation}
to the Bloch--Kato Selmer group is an isomorphism of $\Q_\ell$-vector spaces, and moreover these spaces have dimension $\ord_{s = i} L(H^{2i-1}(Y),s)$,
where $L(H^{2i-1}(Y),s)$ is the $L$-function attached to the associated system of $\ell$-adic $\Gal_F$-representations. 

In our setting, $L(f,\chi,s)$ is a factor of $L(H^{2k+2t-1}(X),s)$, and so Theorem \ref{thm:intro} yields the following implication predicted by Beilinson--Bloch--Kato:

\begin{corollary}
 With the notations and assumptions as above, we have 
\[
r_{\an}(f \otimes \chi)=1 \implies \dim_{\Q} \CH^{k+t}(X)_{0} \geq 1.
\]
\end{corollary}

The Beilinson--Bloch--Kato conjecture is compatible with algebraic correspondences: if $M  \subset h(X)$ is a rational Chow submotive of $X$, determined by an idempotent correspondence $\epsilon$, then we expect that $\ord_{s = i}L(H^{2i-1}(M),s) = \dim_\Q \CH^i(M)_0 = \dim_\Q \epsilon\CH^i(X)_0$.  Combining Theorem \ref{thm:intro} with work of Castella--Hsieh \cite{CastellaHsieh} and Elias \cite{elias} on Euler systems of generalized Heegner cycles, we obtain the following conditional result for the motive $h(f) \otimes h(\chi) \subset h(X)$:

\begin{corollary}\label{cor: combine with Yara}
    Suppose the map $\AJ_X^{k+t}$ $(\ref{eq:AJ})$ is injective. Then 
    \[
    r_{\an}(f \otimes \chi) =  1 \implies \dim_{K(\chi)} \epsilon_{\chi,f}\CH^{k+t}(X)_{0} = 1.
    \]
Here, $\epsilon_{\chi,f}$ is the idempotent defining the rational Chow motive $h(f) \otimes h(\chi)$ inside $h(X)$.\footnote{Scholl constructs $h(f) \subset h(W_{2k-2})$ in the category of Chow motives modulo homological equivalence \cite{scholl}, but injectivity of (\ref{eq:AJ}) implies that $h(f)$ has the desired properties even in the category of motives modulo rational equivalence.} 
\end{corollary}

Castella--Hsieh \cite{CastellaHsieh} proved an analogous result in analytic rank $0$, under some technical hypotheses. See \cite[Rem. 1.2.2]{Castella24} for further applications of our results to the Tamagawa number conjecture.

As the Bloch--Kato Selmer group of a rational Chow motive $M$ depends only on its $\ell$-adic realizations $M_\ell$, we also deduce some applications to an a priori larger class of motives which should have useful applications. 
Let $V_{f,\chi,\ell}$ be the $4$-dimensional $\ell$-adic representations of $\Gal_\Q$ attached to $f \otimes \chi$. 
\begin{corollary}\label{cor: general motive}
    Let $M$ be a rational Chow motive over $\Q$ such that $M_\ell \simeq V_{f,\chi,\ell}$ for all but finitely many $\ell$. Suppose that $\AJ_X^{k+t}$ is injective. Then for primes $\ell \nmid 2(2k-1)!N\varphi(N)$ split in $K$, we have  
    \[r_{\an}(f \otimes \chi) = 1 \implies \dim_{\Q_\ell} H^1_f(\Gal_\Q, M_\ell(k+t)) = 1.\]  
\end{corollary}

The technical hypotheses on $\ell$ again come from \cite{CastellaHsieh}.

\begin{remark}\label{rem: concrete examples}
    {
    If $Y$ is a rigid Calabi--Yau $(2k-1)$-fold and $A^{2t}$ is a power of a CM elliptic curve, then $h(Y \times A^{2t})$ contains a motive $M$ such that $M_\ell \simeq V_{f,\chi,\ell}$, by \cite{GouveaYui}. More elementary examples come from abelian varieties $J$ isogenous to $E^{2k-1} \times A^{2t}$, for some CM elliptic curve $E$. When $J$ is moreover a Jacobian, there is a canonical Ceresa-type cycle in codimension $k+t$, which one expects is typically non-zero in $\CH^{k+t}(M)$ and hence in the Bloch--Kato Selmer group as well. See \cite{lilienfeldtshnidman} or \cite{bdp3} for 3-dimensional examples of this type.  Part of our interest in these examples is that they tend to give  extremal behavior: e.g.,\ the Jacobian $J$  of the modular curve $X(8)$, which is the genus $5$ curve with largest  automorphism group. In these examples, $\CH^{k+t}(M)$ is the $0$-th graded piece of the coniveau filtration on the Griffiths group $\Gr^{k+t}(Z)$, where $Z = Y \times A^{2t}$ or $Z = J$.  Note that the Tate conjecture implies that $M \simeq h(f) \otimes h(\chi)$, but the Tate conjecture is open in these cases.  
    }
\end{remark}
     
One can deduce $p$-adic variants of Corollary \ref{cor: general motive} from  \cite{disegni,pGZshnidman} or \cite{bdpCM} respectively, with hypothesis  that a certain $p$-adic $L$-function  vanishes to order $1$ (resp.\  0), and where the conclusion requires taking $\ell = p$. These $p$-adic results are unconditional since Nekov\'a\v{r}'s $p$-adic height factors through the Abel--Jacobi map. Corollary \ref{cor: general motive} is the first complex result for motives of this type, but it seems difficult to make it unconditional at the moment.

\subsection{Strategy of proof}

As in \cite{GZ}, the proof of Theorem \ref{thm:intro} is divided into four steps. 

First we compute the Fourier coefficients of a newform $g=\sum_{m\geq 1} a_{g}(m) q^m \in S_{2k}(\Gamma_0(N))$ satisfying
\begin{equation}\label{intro:eq}
(f,g)=\frac{\sqrt{\vert D\vert D^t (2k-2)!}}{2^{4k-1} \pi^{2k}} L'(f,\chi, k+t), \qquad \forall \; f\in S_{2k}^{\mathrm{new}}(\Gamma_0(N)).
\end{equation}
This step is carried out using Rankin's method to arrive at an integral representation of $L(f,\chi,s)$ as the Petersson product of $f$ with the product of $\Theta_{\chi}$ with a non-holomorphic Eisenstein series. Using holomorphic projection, we compute the desired Fourier coefficients (Theorem \ref{fourcoeff}). 

The second and third steps are concerned with the calculation of the geometric $q$-expansion
\[
\sum_m \langle \Delta_{\chi}, T_m \Delta_{\chi} \rangle^{\GS} q^m.
\]
Injectivity of $\AJ_X^{k+t}$ would imply that this $q$-expansion is a modular form as well, however this is not known and is only deduced as a result of our proof.\footnote{In the $p$-adic setting, the injectivity of $\AJ_X^{k+t}$ is not necessary to deduce modularity of the analogous $q$-expansion of $p$-adic heights, which is one reason why this paper is technically more involved than \cite{pGZshnidman}.} This means we must compute Fourier coefficients even in cases where $\Delta_\chi$ and $T_m\Delta_\chi$ intersect in the generic fiber. For each $m$, the height pairing admits a decomposition as a sum of local height pairings indexed by the places of $H$. The second and third steps are the calculations of the archimedean and non-archimedean contributions, respectively. Both steps are carried out using Brylinski's pairing for local systems on curves \cite{brylinski}. The second step requires us to identify the Green's kernel attached to the polarized variation of Hodge structures $W_{k,t}$ described below. This occupies a large part of the paper (Sections \ref{s:PVHS} -- \ref{s:green}). The third step can be extracted from local height calculations performed in \cite{pGZshnidman}, at least when there is no intersection in the generic fiber. In both steps, the case of intersection in the generic fiber uses Zhang's work \cite{zhang} on limits of local height pairings. The outcome of the second and third steps is the equality
\begin{equation}\label{fouriereq}
\langle \Delta_{\chi}, T_m \Delta_{\chi} \rangle^{\GS} = \kappa \cdot a_g(m), \qquad \forall \; (m,N)=1,
\end{equation}
where $\kappa$ is an explicit non-zero constant independent of $m$ (Theorem \ref{THM}).

In the fourth and final step, we derive Theorem \ref{thm:intro} from \eqref{intro:eq} and \eqref{fouriereq} using standard arguments.

\subsection{Green's kernels}

 Generalized Heegner cycles give rise to elements in fibers of the weight $0$ polarized variation of Hodge structures
    \[
    W_{k,t}:=(\sym^{2k-2} R^1 \pi_* \Z)(k-1) \otimes \kappa_{2t} H^{2t}(A^{2t}(\C), \Z)(t),
    \]
    where $\pi \colon \mathcal{E} \to Y(N)$ is the universal elliptic curve and $\kappa_{2t}$ is the projector \eqref{eq:kappam}. 
    In Lemma \ref{def:Green} we show that the Green's kernel associated to $W_{k,t}$ (defined for $z = x + iy$ and $z' = x' + iy'$ distinct in the upper half-plane) is obtained by averaging the function
\begin{equation*}
g_{k,t}(z,z'):=-Q_{k,t}\left(1+\frac{\vert z-z'\vert^2}{2yy'}\right),
\end{equation*}
where
\begin{equation*}
Q_{k,t}(x):=\int_{-\infty}^\infty \frac{2^{2t}dw}{(x+\sqrt{x^2-1}\cosh(w))^{k-t}(x+1+\sqrt{x^2-1}e^w)^{2t}}
\end{equation*}
is a Jacobi function of the second kind.
As a result, the archimedean contribution to the height $\langle \Delta_\chi, T_m \Delta_\chi\rangle^{\mathrm{BB}}$ is a linear combination of sums of the form
\[
\mathcal{G}^m_{k, t}(z,z'):=\frac{1}{2}\sum_{\substack{a,b,c,d \in \Z \\ ad-bc=m, N\mid c}} g_{k,t}\bigg(z, \frac{az'+b}{cz'+d}\bigg) (c\bar{z}z'+d\bar{z}-az'-b)^{2t}.
\]
Let $\lambda=\{\lambda_m\}_{\geq 1}$ be a relation for $S_{2k}(\Gamma_0(N))$, in the sense of \cite[p. 316]{GZ}, so that:
\begin{itemize}
    \item[i)] $\lambda_m\in \Z$, $\lambda_m=0$ for all but finitely many $m$;
    \item[ii)] $\sum_{m\geq 1} \lambda_m a_m=0$ for all $\sum_{m\geq 1} a_m q^m \in S_{2k}(\Gamma_0(N))$;
    \item[iii)] $\lambda_m=0$ if $(m,N)\neq 1$ or $r_{\cA, \chi}(m)\neq 0$ for some $\cA\in \Cl_K$.
\end{itemize}
As a consequence of our local height and Fourier coefficient calculations, we show that for points $z$ and $z'$ in the upper half-plane with complex multiplication by $\oh_K$, the value
\[
\sum_{m\geq 1} m^{k-t-1} \lambda_m \mathcal{G}^m_{k, t}(z,z')
\]
is an algebraic multiple of the logarithm of an algebraic number, conditional on the non-degeneracy of the Beilinson--Bloch height pairing.  
See Theorems \ref{thm:algebraicity} and \ref{thm:alg2} for the precise statements. This extends the algebraicity results \cite[V (4.3)]{GZ} and \cite[Thm. 5.2.2]{zhang} of Gross--Zagier and Zhang, which  were recently generalized in a different direction by Bruinier--Ehlen--Yang \cite{BruinierEhlanYang} and Li \cite{Li23}.    
\subsection{Assumptions}

It would be desirable to relax some of the hypotheses in Theorem \ref{thm:intro}, especially in light of Remark \ref{rem: concrete examples}. Among other things, this would entail re-doing Gross and Zagier's deformation theoretic computations on the modular curve $X_0(N)$ in the presence of extra ramification. As the computations in the ``minimally ramified'' case turned out to be quite long on their own, we decided to focus in this paper on what is really new when $\chi$ is no longer of finite order. 

For the future, it would be worthwhile to prove analogous formulas for odd weight modular forms, in which case $\chi$ must be ramified as well, in order for $L(f,\chi,s)$ to be critical and self-dual in the sense of \cite{bdp1}. One could also try to allow modular forms $f$ with ``deeper'' ramification, by combining the method of \cite{qiu} with our work. It would also be desirable to relax the Heegner hypothesis \eqref{HH}, in which case one should work over a compact Shimura curve (corresponding to an indefinite quaternion algebra) as in \cite{yzz}. 
Generalized Heegner cycles over Shimura curves are well-studied \cite{brooks, masdeu}, but there are obstacles to overcome in order to compute their Beilinson--Bloch heights (see \cite{boya}).  Finally, the case $t \geq k$ (with global root number still $-1$) is quite mysterious. The arithmetic Gan--Gross--Prasad program for $\SO(3) \times \SO(2)$ (and the results of Saito--Tunnel) suggest that one should work over the Shimura variety corresponding to a definite quaternion algebra. But the latter is $0$-dimensional, so this case seems to require a completely different construction of algebraic cycles.  One promising approach would be to relate $L(f,\chi,s)$ to modified diagonal cycles, as in \cite{CastellaDo}, by generalizing the triple product formula of Yuan--Zhang--Zhang \cite{yzztriple}. However, this type of Gross--Zagier formula would necessarily involve other auxiliary $L$-values.    

\subsection{Notations and conventions}\label{s:conv}

By default, Chow groups are with $\Q$-coefficients.
Given a $\Z$-module $M$ and a group $G$, we write $G_M$ for $G\otimes_{\Z} M$. Likewise, if $E/F$ is an extension of fields and $Y$ is a variety defined over $F$, then we write $Y_E$ for the base-change $Y\times_{\spec(F)} \spec(E)$. We use a similar notation for base-change of morphisms or algebraic cycles, though we sometimes suppress the subscript when the meaning is unambiguous. Starting from Section \ref{s:GHC}, we adopt the following conventions.
In general, we will use the notation $z$ for a point of $\cH$, while the notation $\tau$ is reserved for CM points. Points on $X(M)$ will generally be denoted using the letter $Q$, while points on $X_0(N)$ will be denoted using the letter $P$. The letters $x$ and $y$ will generally be used to denote real and imaginary parts of complex numbers respectively.

\subsection{Outline}
Section \ref{s:Jacobi} contains preliminaries on Jacobi functions, which will turn out to compute the relevant Green's kernel for the height pairing.
In Section \ref{s:Lfct}, we carry out the $L$-function calculations. More precisely, we define the Rankin--Selberg convolution $L$-function, derive its functional equation, and compute the Fourier coefficients of the cuspform $g$ (Theorem \ref{fourcoeff}). Section \ref{sec:height} recalls generalities about height pairings on algebraic cycles. Section \ref{s:GHC} introduces generalized Kuga--Sato varieties and generalized Heegner cycles. 
Sections \ref{s:PVHS} and \ref{s:laplacian} contain preliminaries on variations of Hodge structures and calculations of Laplacians, respectively. Section \ref{s:green} is dedicated to constructing the relevant Green's kernel needed for the archimedean Brylinski pairing calculations that are subsequently carried out in Section \ref{s:archii}. Section \ref{s:hodge} computes the cycle classes of generalized Heegner cycles, paving the way to the calculation of the archimedean local heights of generalized Heegner cycles in Section \ref{s:archiih}. Section \ref{s:nonarchii} is concerned with the non-archimedean local height calculations. Section \ref{s:GZ} puts everything together, giving a proof of Theorem \ref{thm:intro}. Section \ref{s:algebra} contains algebraicity results for special values of Green's kernels. Section \ref{s:coro} provides a proof of Corollary \ref{cor: combine with Yara}.

\subsection{Acknowledgements}
The authors thank Kartik Prasanna for suggesting this problem many years ago, and Francesc Castella for helpful comments.  
Both authors were funded by the European Research Council (ERC, CurveArithmetic, 101078157) and the Israel
Science Foundation (grant No.\ 2301/20). Some of the local computations in this paper appear in the last two chapters of \cite{AriThesis}, which was written while AS was partially supported
by National Science Foundation RTG grant DMS-0943832. DTBGL was supported by an Emily Erskine Endowment Fund Post-Doctoral Researcher Fellowship while at the Hebrew University of Jerusalem, and by an Edixhoven Post-Doctoral Fellowship while at Leiden University.

\section{Jacobi functions}\label{s:Jacobi}

We collect some standard results about Jacobi functions using \cite{HTFII, szego} as our main sources. We then prove an asymptotic formula near $1$ for Jacobi functions of the second kind. 

\subsection{Hypergeometric funtions}

Given real numbers $a,b,c,x$ with $\vert x\vert < 1$ and $c>0$, the associated hypergeometric function is defined by
\begin{equation}\label{def:hyp}
F(a,b,c;x):= \sum_{n=0}^\infty \frac{(a)_n(b)_n}{(c)_n}\frac{x^n}{n!},
\end{equation}
with $(a)_0=1$ and $(a)_n=\Gamma(a+n)/\Gamma(a)=a(a+1)\ldots (a+n-1)$ being the rising Pochhammer symbol. These functions satisfy Euler's transformation law 
\begin{equation}\label{eulertl}
F(a,b,c;x)=(1-x)^{c-a-b}F(c-a,c-b,c;x).
\end{equation}

\subsection{Jacobi polynomials}

Let $\alpha$ and $\beta$ be real numbers, and let $n\geq 0$ be an integer. The corresponding Jacobi polynomial $P_n^{(\alpha, \beta)}$ is defined by the formula
\begin{equation}\label{defjac}
P_n^{(\alpha, \beta)}(x):=\frac{(\alpha+1)_n}{n!} F\left(-n, n+\alpha+\beta+1, \alpha+1; \frac{1-x}{2}\right).
\end{equation}
Explicitly, we have 
\begin{equation}\label{expjac}
P_n^{(\alpha, \beta)}(x)=\frac{\Gamma(n+\alpha+1)}{n! \Gamma(n+\alpha+\beta+1)}\sum_{j=0}^n \binom{n}{j} \frac{\Gamma(n+\alpha+\beta+j+1)}{\Gamma(\alpha+j+1)} \left( \frac{x-1}{2}\right)^j.
\end{equation}
The following formula is known as Rodrigues' formula \cite[Vol. II 10.8.(10)]{HTFII}:
\begin{equation}\label{rodrig}
P_n^{(\alpha, \beta)}(x)=\frac{(-1)^n}{2^n n!(1-x)^\alpha (1+x)^{\beta}} \left( \frac{d}{dx} \right)^{n} [(1-x)^{n+\alpha}(1+x)^{n+\beta}].
\end{equation}
The Jacobi polynomials are solutions of the second order differential equation \cite[Vol. II 10.8.(14)]{HTFII}
\begin{equation}\label{diffeq1}
(1-x^2)f''(x)+[\beta-\alpha-(\alpha+\beta+2)x]f'(x)+n(n+\alpha+\beta+1)f(x)=0.
\end{equation}

For $0\leq t\leq k-1$, define the function
\begin{equation}\label{Pkt}
P_{k,t}(x)=\frac{1}{2^{k-t-1}(k+t-1)!}\left( \frac{d}{dx} \right)^{k+t-1}[(x^2-1)^{k-t-1}(x-1)^{2t}].
\end{equation}

\begin{proposition}
    For $0\leq t\leq k-1$, we have $P_{k,t}(x)=P_{k-t-1}^{(0, 2t)}(x)$.
\end{proposition}

\begin{proof}
Observe, using Rodrigues's formula \eqref{rodrig}, that
\[
P_{k+t-1}^{(0, -2t)}(x)=\frac{(-1)^{k+t-1}(1+x)^{2t}}{2^{k+t-1} (k+t-1)! } \left( \frac{d}{dx} \right)^{k+t-1} [(1-x)^{k+t-1}(1+x)^{k-t-1}]
=\frac{(1+x)^{2t}}{2^{2t}}P_{k,t}(x).
\]
Thus, using \eqref{expjac} and Euler's tranformation law \eqref{eulertl}, we see that 
\[
P_{k,t}(x)=\frac{2^{2t}}{(1+x)^{2t}} F\left(-k-t+1, k-t, 1; \frac{1-x}{2}\right)
=F\left(k+t, -k+t+1, 1; \frac{1-x}{2}\right)
=P_{k-t-1}^{(0,2t)}(x).
\]
\end{proof}

Associated with the polynomials $P_n^{(\alpha, \beta)}$ are the Jacobi polynomials of the second kind\footnote{We depart here from the notations of \cite{HTFII} where these polynomials are denoted $q_{n}^{(\alpha, \beta)}(x)$ \cite[Vol. II 10.8.(24)]{HTFII}. Instead, we adopt the notations of \cite[(9)]{grosjean}, except that the polynomials therein are normalized by a factor of $\frac{\Gamma(\alpha+\beta+2)}{2^{\alpha+\beta+1} \Gamma(\alpha+1)\Gamma(\beta+1)}$ which we choose to drop.}
\begin{equation}\label{jacpol2}
W_{n-1}^{(\alpha, \beta)}(x):=\int_{-1}^1 \frac{P^{(\alpha, \beta)}_{n}(x)-P^{(\alpha, \beta)}_{n}(u)}{x-u}(1-u)^\alpha (1+u)^\beta du,
\end{equation}
defined for all $n\in \mathbb{Z}_{\geq 1}$, all $x\in \C$, and $\alpha, \beta > -1$. The function \eqref{jacpol2} is the degree $n-1$ polynomial solution of the inhomogeneous linear differential equation
\[
(1-x^2)f''(x)+[\alpha-\beta+(\alpha + \beta - 2)x] f'(x)+(n+1)(n+\alpha + \beta)f(x)=2(\alpha +\beta +1)\frac{d}{dx} P_{n}^{(\alpha, \beta)}(x).
\]

\subsection{Jacobi functions of the second kind}

The Jacobi functions of the second kind are denoted by $Q^{(\alpha, \beta)}_n$ and defined, for $x$ in the complex plane cut along the segment $[-1,1]$, by the formula 
\cite[(4.61.1))]{szego}
\begin{equation}\label{intrep}
Q^{(\alpha, \beta)}_n(x):=\frac{2^{-n-1}}{(x-1)^\alpha (x+1)^{\beta}} \int_{-1}^1 \frac{(1-u)^{n+\alpha}(1+u)^{n+\beta}}{(x-u)^{n+1}}du.
\end{equation}
In the special case $\alpha=\beta=0$, the function $Q_n(x):=Q_n^{(0,0)}(x)$ is the Legendre function of the second kind. The function $Q_{n}^{(\alpha, \beta)}(x)$ can be defined for $-1< x < 1$ by the formula \cite[(4.62.9)]{szego}, but this definition will not be needed here. The function $Q_{n}^{(\alpha, \beta)}(x)$ is a non-polynomial solution of the differential equation \eqref{diffeq1}, which is linearly independent of $P_n^{(\alpha, \beta)}(x)$ \cite[Thm. 4.61.1]{szego}. According to \cite[Thm. 4.61.2]{szego}, the following representations hold for $x$ in the complex plane cut along the segment $[-1,1]$: 
\begin{equation}\label{jacpol3}
Q_{n}^{(\alpha, \beta)}(x)=\frac{2^{-1}}{(x-1)^\alpha (x+1)^{\beta}}\int_{-1}^1 \frac{P^{(\alpha, \beta)}_{n}(u)}{x-u}(1-u)^\alpha (1+u)^\beta du,
\end{equation}
and
\begin{equation}\label{hyperg}
Q_{n}^{(\alpha, \beta)}(x):=\frac{2^{n+\alpha+\beta}\Gamma(n+\alpha+1)\Gamma(n+\beta+1)}{\Gamma(2n+\alpha+\beta+2)(x-1)^{n+\alpha+1}(x+1)^\beta} F\left( n+1, n+\alpha+1, 2n+\alpha+\beta+2; \frac{2}{1-x} \right).
\end{equation}
Observe, using \eqref{jacpol3} and \eqref{jacpol2}, that for $x$ in the complex plane cut along the segment $[-1,1]$ we have
    \begin{align*}
        Q_{n}^{(\alpha, \beta)}(x)&=\frac{2^{-1}}{(x-1)^\alpha (x+1)^{\beta}}\int_{-1}^1 \frac{P^{(\alpha, \beta)}_{n}(u)-P^{(\alpha, \beta)}_{n}(x)+P^{(\alpha, \beta)}_{n}(x)}{x-u}(1-u)^\alpha (1+u)^\beta du \\
        &=\frac{2^{-1}P^{(\alpha, \beta)}_{n}(x)}{(x-1)^\alpha (x+1)^{\beta}}\int_{-1}^1 \frac{(1-u)^\alpha (1+u)^\beta}{x-u} du - \frac{2^{-1}W_{n-1}^{(\alpha, \beta)}(x)}{(x-1)^\alpha (x+1)^{\beta}}. \\
    \end{align*}
    Using \eqref{intrep}, we obtain
    \begin{equation}\label{eqQ0}
    Q_{n}^{(\alpha, \beta)}(x)=P^{(\alpha, \beta)}_{n}(x)Q_0^{(\alpha, \beta)}(x)- \frac{2^{-1}W_{n-1}^{(\alpha, \beta)}(x)}{(x-1)^\alpha (x+1)^{\beta}}, \qquad \forall x\not\in [-1,1].
    \end{equation}
    Equation \eqref{eqQ0} is \cite[Vol. II 10.8.(25)]{HTFII}, given our notations and normalizations.

For all $0\leq t\leq k-1$, we define the function
\begin{equation}\label{Qkt1}
Q_{k,t}(x)=\int_{-\infty}^\infty \frac{2^{2t}dw}{(x+\sqrt{x^2-1}\cosh(w))^{k-t}(x+1+\sqrt{x^2-1}e^w)^{2t}},
\end{equation}
where $x$ is in the complex plane cut along the segment $[-1,1]$.

\begin{proposition}\label{prop:Jacobi}
For $0\leq t\leq k-1$, we have 
\[
Q_{k,t}(x)=2Q^{(0,2t)}_{k-t-1}(x).
\]
\end{proposition}

\begin{proof}
Replacing $\cosh(w)=(e^{2w}+1)/(2e^w)$ in \eqref{Qkt1} and performing the change of variables $v=\sqrt{\frac{x+1}{x-1}}e^w$ gives the formula
\[
Q_{k,t}(x)=2^{k+t}\int_{0}^\infty \frac{v^{k-t-1}dv}{(2xv+(x-1)v^2+x+1)^{k-t}(x+1+(x-1)v)^{2t}}.
\]
Rewriting $2xv+(x-1)v^2+x+1=(v+1)(x(v+1)+1-v)$ and performing the change of variables $v=\frac{1+u}{1-u}$ yields
\[
Q_{k,t}(x)=2^{-k+t+1}\int_{-1}^1 \frac{(1+u)^{k-t-1}(1-u)^{k+t-1} du}{(x-u)^{k+t}}.
\]
By comparing with \eqref{intrep}, we observe that 
\[
Q_{k,t}(x)=\frac{2^{2t+1}}{(x+1)^{2t}} Q_{k+t-1}^{(0, -2t)}(x).
\]
Using the defining equation \eqref{hyperg} leads to 
\[
Q_{k,t}(x)=2^{k+t}\frac{\Gamma(k+t)\Gamma(k-t)}{\Gamma(2k)(x-1)^{k+t}} F\left( k+t, k+t, 2k; \frac{2}{1-x} \right).
\]
Applying Euler's transformation law \eqref{eulertl} yields
\begin{align*}
Q_{k,t}(x) & =2^{k+t}\frac{\Gamma(k+t)\Gamma(k-t)}{\Gamma(2k)(x-1)^{k+t}} \left( 1- \frac{2}{1-x}\right)^{-2t}F\left( k-t, k-t, 2k; \frac{2}{1-x}\right) \\
&=2 \left( 2^{k+t-1}\frac{\Gamma(k+t)\Gamma(k-t)}{\Gamma(2k)(x-1)^{k-t}(x+1)^{2t}}F\left( k-t, k-t, 2k; \frac{2}{1-x}\right)\right) \\
& = 2Q_{k-t-1}^{(0, 2t)}(x),
\end{align*}
where we used \eqref{hyperg} in the last equality.
\end{proof}

\begin{corollary}\label{coro:difeq}
The function $Q_{k,t}(x)$ is a solution of the second order differential equation
\[
(1-x^2)f''(x)+[2t-(2t+2)x]f'(x)+(k-t-1)(k+t)f(x)=0.
\]
\end{corollary}

\begin{proof}
It follows from Proposition \ref{prop:Jacobi} and the fact that $Q_n^{(\alpha, \beta)}(x)$ is a solution of \eqref{diffeq1}.
\end{proof}

\subsection{Asymptotics for Jacobi functions}\label{s:asymp}

\begin{lemma}\label{lem:asyinfty2}
    For $0 \leq t \leq k-1$, we have 
    \[
    Q_{k,t}(x)=O(x^{-k-t}), \qquad \text{ as } x\to +\infty.
    \]
\end{lemma}

\begin{proof}
    By \eqref{hyperg}, and using the definition of the hypergeometric function \eqref{def:hyp}, we have 
    \[
    Q_{n}^{(\alpha, \beta)}(x)=\frac{(-1)^{n+1}2^{\alpha+\beta-1}}{(x-1)^{\alpha}(x+1)^\beta} \sum_{j\geq 0} \binom{n+j}{j} \frac{\Gamma(n+\alpha+j+1)\Gamma(n+\beta+1)}{\Gamma(2n+\alpha+\beta+j+2)} \left( \frac{2}{1-x} \right)^{n+j+1}.
    \]
    Thus, $Q_{n}^{(\alpha, \beta)}(x)=O(x^{-n-\alpha-\beta-1})$ as $x\to +\infty$. In particular, using Proposition \ref{prop:Jacobi}, we see that $Q_{k,t}(x)=O(x^{-(k-t-1+2t+1)})=O(x^{-(k+t)})$ as $x\to +\infty$.
\end{proof}

\begin{lemma}\label{lem:asy}
For $0\leq t\leq k-1$, we have
\[
Q_{k,t}(x)=\log \left(\frac{x+1}{x-1}\right)-\left( \frac{\Gamma'}{\Gamma}(k+t)-\frac{\Gamma'}{\Gamma}(1) + \frac{\Gamma'}{\Gamma}(k-t)-\frac{\Gamma'}{\Gamma}(1) \right)+o(1), \qquad \text{ as } x\to 1^+.
\]
\end{lemma}

\begin{proof}
    Let $n:=k-t-1$ in this proof.
    Using Proposition \ref{prop:Jacobi}, \eqref{eqQ0}, and \eqref{intrep}, for all $x\in \C\setminus [-1,1]$ we have 
    \begin{equation}\label{yas}
    Q_{k,t}(x)=\frac{P^{(0, 2t)}_{n}(x)}{(x+1)^{2t}} \int_{-1}^1 \frac{(1+u)^{2t}}{x-u}du- \frac{W_{n-1}^{(0, 2t)}(x)}{(x+1)^{2t}}.
    \end{equation}
    By \cite[(87)]{grosjean}, for all $x\in \C$, we have\footnote{Beware of our normalizations that differ from the ones in \cite{grosjean}.}
\begin{equation}\label{Weq}
\frac{2t+1}{2^{2t+1}}W^{(0,2t)}_{n-1}(x)=\frac{(2t+1)(k+t)}{2} \sum_{j=0}^{k-t-2} \frac{P^{(0,-2t)}_{j}(x)P^{(0, 2t)}_{k-t-2-j}(x)}{(j+1)(k+t-1-j)}.
\end{equation}
Thus, 
\begin{align*}
\frac{W^{(0,2t)}_{n-1}(1)}{2^{2t}} & = \sum_{j=0}^{k-t-2} \frac{k+t}{(j+1)(k+t-1-j)}
=\sum_{j=0}^{k-t-2} \left(\frac{1}{j+1} + \frac{1}{k+t-1-j}\right) 
 = \sum_{j=1}^{k-t-1} \frac{1}{j} + \sum_{j=2t+1}^{k+t-1} \frac{1}{j} \\
&=\left( \frac{\Gamma'}{\Gamma}(k-t)-\frac{\Gamma'}{\Gamma}(1) \right)+\left( \frac{\Gamma'}{\Gamma}(k+t)-\frac{\Gamma'}{\Gamma}(1) \right)-\left( \frac{\Gamma'}{\Gamma}(2t+1)-\frac{\Gamma'}{\Gamma}(1) \right).
\end{align*}
We now calculate the elementary integral in \eqref{yas}:
\[
    \int_{-1}^1 \frac{(1+u)^{2t}}{x-u}du 
    = (1+x)^{2t}\log\left( \frac{x+1}{x-1}\right)+\sum_{j=1}^{2t} \binom{2t}{j} \frac{(-1)^j}{j} (x+1)^{2t-j}[(x+1)^{j}-(x-1)^j].
\]
As $x\to 1^+$, we thus obtain
\begin{multline*}
Q_{k,t}(x)=\log\left( \frac{x+1}{x-1}\right)+\sum_{j=1}^{2t} \binom{2t}{j} \frac{(-1)^j}{j} \\ -\left( \frac{\Gamma'}{\Gamma}(k-t)-\frac{\Gamma'}{\Gamma}(1) \right)-\left( \frac{\Gamma'}{\Gamma}(k+t)-\frac{\Gamma'}{\Gamma}(1) \right)+\left( \frac{\Gamma'}{\Gamma}(2t+1)-\frac{\Gamma'}{\Gamma}(1) \right)+o(1).
\end{multline*}
It is a standard fact about harmonic numbers that the formula
\begin{equation}\label{harm}
\frac{\Gamma'}{\Gamma}(2t+1)-\frac{\Gamma'}{\Gamma}(1)=\sum_{j=1}^{2t} \frac{1}{j}=-\sum_{j=1}^{2t} \binom{2t}{j} \frac{(-1)^j}{j}
\end{equation}
holds (an easy check using integral representations). This last observation concludes the proof.
\end{proof}

\begin{remark}
    When $t=0$, Lemma \ref{lem:asy} recovers the asymptotic formula \cite[p. 251]{GZ} for Legendre functions of the second kind (note that the $O(1)$ should be $o(1)$ therein).
\end{remark}

\section{$L$-functions}\label{s:Lfct}

Let $r=2k\geq 2$ and $N\geq 3$ be fixed integers. Let $K$ denote an imaginary quadratic field of odd discriminant $D$ prime to $N$. Consider 
\[
f = \sum_{n\geq 1} a_f(n) q^n \in S_r^{\new}(\Gamma_0(N)),
\]
and let $\chi \colon \A_K^\times \to \C^\times$ be a Hecke character of $K$ of infinity type $(\ell, 0)$, with $0<\ell\leq r-2$. We assume that the Dirichlet character $\chi_{|\A_\Q^\times} \cdot \mathrm{Nm}^{-\ell}$ is trivial, and hence $\chi^\tau\chi$ is an even power of the norm, where $\Gal(K/\Q) = \langle \tau \rangle$. This forces $\ell$ to be even, so $\ell = 2t$.
We will also assume that $\chi$ is unramified. There are then at most $\#\Cl_K$ characters $\chi$ of type $(\ell, 0)$, all differing by characters which factor through $\Cl_K \simeq \Gal(H/K)$, where $H$ is the Hilbert class field. 

Associated to $\chi$ is the theta series \cite[Prop. 3.13]{bdpCM}
\[
\Theta_{\chi}=\sum_{\fa\trianglelefteq \oh_K} \chi(\fa)q^{N(\fa)}\in S_{2t+1}
(\Gamma_0(\vert D\vert), \epsilon_K).
\]
Given an ideal class $\cA \in \Cl_K$, there are also partial theta series
\[
\Theta_{\cA,\chi}=\sum_{\substack{\fa\trianglelefteq \oh_K \\ [\fa] = \cA}} \chi(\fa)q^{N(\fa)}=\sum_{n\geq 1} r_{\cA, \chi}(n) q^n\in S_{2t+1}
(\Gamma_0(\vert D\vert), \epsilon_K),
\]
where 
\[  
r_{\cA, \chi}(n)=\sum_{\substack{[\fa] = \cA \\ N(\fa)=n}} \chi(\fa).
\]
We have the relation
\[
\Theta_{\chi}=\sum_{\cA \in \Cl_K} \Theta_{\cA,\chi}.
\]

\subsection{Rankin--Selberg convolution $L$-functions}\label{ss:RS}

Let $\cA \in \Cl_K$ be an ideal class, and consider the Rankin--Selberg convolution $L$-function of $f$ and the partial theta series $\Theta_{\cA, \chi}$ defined by 
\[
L_\cA(f,\chi,s):= L^{(N)}(\epsilon_K, 2s-2k-2t+1)\sum_{n\geq 1} a_f(n)r_{\cA, \chi}(n)n^{-s},
\]
where $\epsilon_K=\epsilon_D$ is the quadratic character associated to $K$.
It converges locally uniformly on the right half-plane $\Re(s)>k+t+3/2$ of the complex plane. The Rankin--Selberg convolution $L$-function of $f$ and $\Theta_{\chi}$ is then given by 
\[
L(f, \chi, s)=\sum_{\cA \in \Cl_K} L_{\cA}(f, \chi, s).
\]
Define the completed $L$-function
\[
L^*_\cA(f,\chi,s):=(2\pi)^{-s}N^s D^s \Gamma(s)\Gamma(s-2t)L_\cA(f,\chi,s).
\]

\begin{theorem}\label{thm:FE}
The $L$-function $L^*_\cA(f,\chi,s)$ admits an analytic continuation to the whole complex plane and satisfies the functional equation 
\[
L^*_\cA(f,\chi,s)=-\epsilon_K(N)L^*_\cA(f,\chi,2k+2t-s).
\]
\end{theorem}

\begin{proof}
The key is to use Rankin's method to prove the equality
\begin{equation}\label{rankin}
\frac{N^s\Gamma(s+2k-1)}{(4\pi)^{s+2k-1}}L_\cA(f,\chi,s+2k-1)=(f, \tilde \Phi_{\bar s})_{\Gamma_0(N)},
\end{equation}
where $\tilde \Phi_{s}$ is a non-holomorphic modular form of weight $2k$ and level $\Gamma_0(N)$ given by the trace of the product of the theta function $\Theta_{\bar{\cA}, \chi}$ with a non-holomorphic Eisenstein family indexed by $s$:
\[
\tilde \Phi_{s}=\mathrm{Tr}^{ND}_N(E_s^{(1)}(Nz)\Theta_{\bar{\cA}, \chi}(z)),
\]
where 
\[
E_s^{(1)}(z)=\frac{1}{2} \sum_{\substack{c,d\in \Z \\ D\vert c }} \frac{\epsilon_K(d)}{(cz+d)^{2k-2t-1}} \frac{y^s}{\vert cz+d\vert^{2s}}.
\]
By explicitly working out an expression for the trace, we obtain
\[
\tilde \Phi_{s}=D^{-t} \mathcal{E}_s(Nz)\Theta_{\bar{\cA}, \chi}(z) \vert_{U_{\vert D\vert}},
\]
where 
\[
\mathcal{E}_s(z)=\sum_{D_1 \cdot D_2} \frac{\epsilon_{D_1}(N)\chi_{D_1\cdot D_2}(\bar{\cA})}{\kappa(D_1)\vert D_1\vert^{s+\ell_1-3/2}} E_s^{(D_1)}(\vert D_2\vert z)=\sum_{n\in \Z} e_s(n, y) e(nx),
\]
with $e(x)=e^{2\pi i x}$ (exactly the same Eisenstein series as in \cite{GZ}, except that $k$ is replaced by $k-t$, or equivalently $2k$ is replaced by $\ell_1=2k-2t$). Formulas for the coefficients $e_s(n, y)$ can be found in \cite[IV (3.2)]{GZ} (after replacing $2k$ therein by $\ell_1$). In the above formula, $\kappa(D_1)=1$ or $i$ according as $D_1>0$ or $D_1<0$.

\begin{remark}
The character $\chi_{D_1 \cdot D_2}$ appearing in the above formula is the genus character associated with the partition $D=D_1 D_2$. Its value at a prime ideal $\fp$ is
\[
\chi_{D_1 \cdot D_2}(\fp)=
\begin{cases}
\epsilon_{D_1}(N(\fp)), & (N(\fp), D_1)=1 \\
\epsilon_{D_2}(N(\fp)), & (N(\fp), D_2)=1. \\
\end{cases}
\]
\end{remark}

If we let 
\[
e^*_s(n,y):=\pi^{-s} \vert D\vert^s \Gamma(s+2k-2t-1)e_s(n,y), 
\]
then the functional equation follows from the transformation property
\[
e^*_s(n,y)=-\epsilon_K(N)e^*_{2-2k+2t-s}(n,y)
\]
proved in \cite[IV \S 4]{GZ}. 
\end{proof}

\subsection{Central derivative calculations}

We assume that $K$ satisfies the Heegner hypothesis \eqref{HH}, namely that every prime dividing $N$ splits in $K$. This implies in particular that the sign of the functional equation of Theorem \ref{thm:FE} is $-1$. Moreover, equation \eqref{rankin} implies that there exists a non-holomorphic modular form $\tilde \Phi$ of weight $2k$ and level $\Gamma_0(N)$ such that the central value of the derivative is given by 
\[
L'_\cA(f,\chi,k+t)=\frac{2^{2k+2t+1}\pi^{k+t+1}}{(k+t-1)! \sqrt{\vert D\vert}D^t}(f,\tilde \Phi)_{\Gamma_0(N)}.
\]
In fact, the function is given by 
\begin{equation}\label{phifct}
\tilde \Phi =\frac{\sqrt{\vert D\vert}}{2\pi} N^{k-t-1} \frac{\partial}{\partial s} (D^t \tilde \Phi_s)\vert_{s=1-(k-t)}.
\end{equation}

\begin{proposition}\label{fourier}
The Fourier expansion of $\tilde \Phi$ is given by
\begin{align*}
\tilde \Phi (z) & = \sum_{m=-\infty}^\infty \left[ -\sum_{0<n\leq \frac{m\vert D\vert}{N}} \sigma'_{\cA}(n) r_{\bar{\cA}, \chi}(m\vert D\vert-nN) p_{k-t-1}\left( \frac{4\pi n N y}{\vert D\vert}\right) \right. \\
& \qquad + \frac{h}{u}D^t r_{\bar{\cA}, \chi}(m)\left( \log(y) + \frac{\Gamma'}{\Gamma}(k-t) + \log (N\vert D\vert) - \log(\pi) +2 \frac{L'}{L}(1, \epsilon_K) \right) \\
& \left.\qquad - \sum_{n=1}^\infty \sigma_{\cA}(n) r_{\bar{\cA}, \chi}(m\vert D\vert + nN) q_{k-t-1}\left( \frac{4\pi n N y}{\vert D\vert} \right)\right] y^{1-k+t} q^m, \\
\end{align*}
where 
\[
p_{\nu-1}(x):=\sum_{j=0}^{\nu-1} \binom{\nu-1}{j}\frac{(-x)^j}{j!} \qquad \text{ and } \qquad q_{\nu-1}(x):=\int_1^{\infty} \frac{(y-1)^{\nu-1}}{y^\nu} e^{-xy} dy,
\]
and
where the two first terms are zero whenever $m\leq 0$. Moreover, 
\[
\sigma_{\cA}'(n)=\sum_{\substack{d \mid n \\ d>0}} \epsilon_{\cA}(n,d) \log\left( \frac{n}{d^2} \right)
\]
and 
\[
\sigma_{\cA}(n)=\sum_{\substack{d \mid n \\ d>0}} \epsilon_{\cA}(n,d),
\]
with
\[
\epsilon_{\cA}(n,d)=
\begin{cases}
0 & (d,n/d, D) \neq 1 \\
\epsilon_{D_1}(d)\epsilon_{D_2}(-N n/d) \chi_{D_1 \cdot D_2}(\cA) & (d,n/d, D) = 1,
\end{cases}
\]
for $(d, D)= \vert D_2\vert$ and $D_1 D_2=D$.
\end{proposition}

\begin{remark}
    The fact that $r_{\bar{\cA}, \chi}(0)=0$ is proper to the $t> 0$ case, as $\Theta_{\cA, \chi}$ is a cusp form. This is not the case when dealing with a class group character as in \cite{GZ}. This simple observation leads to several simplifications when applying holomorphic projection in the next section.
\end{remark}

\begin{proof}
We have $D^t \tilde \Phi_s(z)=\mathcal{E}_s(Nz)\Theta_{\bar{\cA}, \chi}(z) \vert_{U_{\vert D\vert}}$ and the Fourier expansion of $\mathcal{E}_s(z)$ is given by $\sum_{n\in \Z} e_s(n, y) e(nx)$. Hence, the Fourier expansion of $\mathcal{E}_s(Nz)\Theta_{\bar{\cA}, \chi}(z)$ is 
\begin{align*}
\mathcal{E}_s(Nz)\Theta_{\bar{\cA}, \chi}(z) & =
\left(\sum_{n\in \Z} e_s(n, Ny) e(nNx) \right) \left( \sum_{m\geq 0} r_{\bar{\cA}, \chi}(m) e^{-2\pi m y} e(mx) \right) \\
& = \sum_{v\in \Z} \left( \sum_{\substack{m\geq 0 \\ v-m\equiv 0\pmod N}} e_s\left( \frac{v-m}{N}, Ny \right) r_{\bar{\cA}, \chi}(m) e^{-2\pi m y} \right) e(vx) \\
& =\sum_{v\in \Z} A_v(y) e(vx), \\
\end{align*}
where we made the substitution $v=nN+m$. Applying the operator $U_{\vert D\vert}$, we deduce that 
\begin{align*}
D^t \tilde \Phi_s(z) & = \sum_{v \in \Z} A_{v\vert D\vert}\left( \frac{y}{\vert D\vert} \right) e(vx) \\ 
& =\sum_{v\in \Z} \left( \sum_{\substack{m\geq 0 \\ v-m\equiv 0\pmod N}} e_s\left( \frac{v\vert D\vert-m}{N}, \frac{Ny}{\vert D\vert} \right) r_{\bar{\cA}, \chi}(m) e^{-2\pi m \frac{y}{\vert D\vert}} \right) e(vx) \\
& = \sum_{n\in \Z} \left( \sum_{\substack{m\geq 0 \\ nN+m\equiv 0 \pmod{\vert D\vert}}} e_s\left( n, \frac{Ny}{\vert D\vert}\right) r_{\bar{\cA}, \chi}(m) e^{-2\pi m \frac{y}{\vert D \vert}} \right) e\left( \frac{nN+m}{\vert D\vert} x \right), \\
\end{align*}
where we used the substitution $n=(v\vert D\vert -m)/N$.

We are going to focus on the terms of the Fourier expansion where $n=0$ (the cases where $n>0$ or $n<0$ are similar to what is done in \cite{GZ}), namely
\begin{align*}
\sum_{\substack{m\geq 0 \\ m\equiv 0 \pmod{\vert D\vert}}} e_s\left( 0, \frac{Ny}{\vert D\vert}\right) r_{\bar{\cA}, \chi}(m) e^{-2\pi m \frac{y}{\vert D \vert}} e\left( \frac{m}{\vert D\vert} x \right) & =\sum_{\substack{m\geq 0 \\ m\equiv 0 \pmod{\vert D\vert}}} e_s\left( 0, \frac{Ny}{\vert D\vert}\right) r_{\bar{\cA}, \chi}(m) q^{\frac{m}{\vert D\vert}} \\
& = \sum_{m'\geq 0} e_s\left( 0, \frac{Ny}{\vert D\vert}\right) r_{\bar{\cA}, \chi}(m' \vert D\vert) q^{m'}, \\
\end{align*}
where we used the substitution $m'=m/\vert D\vert$.
We need to calculate 
\[
\sum_{m\geq 0} \left.\frac{\partial}{\partial s} e_s\left( 0, \frac{Ny}{\vert D\vert}\right)\right\vert_{s=1-(k-t)} r_{\bar{\cA}, \chi}(m \vert D\vert) q^{m}.
\]
But this was done already in \cite[p. 283]{GZ} (they work with the same Eisenstein series but with $k$ instead of $k-t$):
\[
\left.\frac{\partial}{\partial s}e_s\left( 0, y\right)\right\vert_{s=1-(k-t)}= 2L(1, \epsilon_K) (\vert D\vert y)^{1-k+t}\left( \frac{\Gamma'}{\Gamma}(k-t) + \log\left(\frac{\vert D\vert^2 y}{\pi}\right) + 2 \frac{L'}{L}(1, \epsilon_K)\right).
\]
We have
\begin{multline*}
\left.\frac{\partial}{\partial s}e_s\left( 0, \frac{Ny}{\vert D\vert}\right)\right\vert_{s=1-(k-t)} = 2L(1, \epsilon_K) (Ny)^{1-k+t}\left( \frac{\Gamma'}{\Gamma}(k-t) + \log\left(\frac{\vert D\vert Ny}{\pi}\right) + 2 \frac{L'}{L}(1, \epsilon_K)\right) \\
= 2L(1, \epsilon_K) (Ny)^{1-k+t}\left( \log(y) + \frac{\Gamma'}{\Gamma}(k-t) + \log\left(N\vert D\vert\right) -\log(\pi) + 2 \frac{L'}{L}(1, \epsilon_K)\right).
\end{multline*}
We conclude that the $n=0$ term of the Fourier expansion of $\tilde \Phi(z)$ is 
\[
\sum_{m\geq 0} \frac{\sqrt{\vert D\vert}}{\pi} L(1, \epsilon_K)  r_{\bar{\cA}, \chi}(m \vert D\vert) \left( \log(y) + \frac{\Gamma'}{\Gamma}(k-t) + \log\left(N\vert D\vert\right) -\log(\pi) + 2 \frac{L'}{L}(1, \epsilon_K)\right) y^{1-k+t} q^{m}.
\]
The result follows from using the formula $\frac{\sqrt{\vert D\vert}}{\pi} L(1, \epsilon_K)=\frac{h}{u}$ \cite[p. 249]{GZ} and observing that
\[
r_{\bar{\cA}, \chi}(m \vert D\vert)= D^t r_{\bar{\cA}, \chi}(m).
\]
Note that 
\[
r_{\bar{\cA}}(m \vert D\vert)= r_{\bar{\cA}}(m)=r_{\cA}(m \vert D\vert)= r_{\cA}(m), 
\]
and \cite{GZ} opted for the simpler expression $r_{\cA}(m)$ in their formula, making it easy to miss the factor of $D^t$ in our formula (the factor of $D^t$ is in fact missing in \cite[Prop. X.7]{AriThesis}).
\end{proof}

\subsection{Holomorphic projection}

Let $g_\cA=\sum_{m\geq 1} a_m(\cA)q^m\in S_{2k}^{\new}(\Gamma_0(N))$ represent the linear functional on $S_{2k}^{\new}(\Gamma_0(N))$ given by 
\[
f \longmapsto \frac{(2k-2)! \sqrt{\vert D\vert}D^t}{2^{4k-1}\pi^{2k}} L'_\cA(f,\chi,k+t).
\]
Then $g_\cA$ is obtained by applying holomorphic projection lemma \cite[IV (5.1)]{GZ} to the function $\tilde \Phi$ defined by \eqref{phifct}. 
More precisely, $g_{\cA}$ is the holomorphic projection of 
\[
\frac{(2k-2)!}{(k+t-1)!} (4\pi)^{t-k+1} \tilde{\Phi}.
\]

\begin{lemma}\label{lem:est}
Assume that $0< t \leq k-1$.
The function $\tilde \Phi(z)=\sum_{m=-\infty}^{\infty} a_m(y) q^m$ satisfies $(\tilde \Phi \vert_{2k} \alpha)(z) = O(y^{-\epsilon})$ as $y\to \infty$ for some $\epsilon >0$ and every $\alpha \in \SL_2(\Z)$.
\end{lemma}

\begin{proof}
    We have the estimates $p_{\nu-1}(x)=O(x^{\nu-1})$, $q_{\nu-1}(x)=O(x^{\nu-1}e^{-x})$, $r_{\bar\cA, \chi}(n)=O(n^{t+1/2})$, $\sigma_{\cA}(n)=o(n^{\delta})$ and $\sigma_{\cA}'(n)=O(n^{\delta})$ for all $\delta>0$. Using these estimates, and the formulas for $a_m(y)$ derived in Proposition \ref{fourier}, we obtain the following estimates. When $t<k-1$, we have
    \[
    a_m(y)=
    \begin{cases}
        O(m^{\delta+k+1/2}) & (m>0) \\
        O(e^{-4\pi \frac{N}{\vert D\vert}y}) & (m=0) \\
        O(\vert m\vert^{\delta+k-t-1} e^{-4\pi \vert m\vert y}) & (m<0).
    \end{cases}
    \]
    When $t=k-1$, we have
    \[
    a_m(y)=
    \begin{cases}
        O(m^{k-1/2}\log(y)) & (m>0) \\
        O(e^{-4\pi \frac{N}{\vert D\vert}y}) & (m=0) \\
        O(\vert m\vert^{\delta} e^{-4\pi \vert m\vert y}) & (m<0).
    \end{cases}
    \]
    In particular, for $0<t\leq k-1$, all terms in the Fourier expansion of $\tilde\Phi$ are exponentially small as $y\to \infty$. Similar estimates hold at the other cusps. 
\end{proof}

\begin{theorem}\label{fourcoeff}
If $0<t\leq k-1$, then for all $m\geq 1$ such that $(m,N)=1$, we have 
\begin{align*}
a_m(\cA) & = m^{k-t-1} \left[ -\sum_{0<n\leq \frac{m\vert D\vert}{N}} \sigma'_{\cA}(n) r_{\bar{\cA}, \chi}(m\vert D\vert-nN) P_{k,t}\left( 1-\frac{2nN}{m\vert D\vert}\right) \right. \\
& \qquad + \frac{h}{u}D^t r_{\bar{\cA}, \chi}(m)\left( \frac{\Gamma'}{\Gamma}(k+t) + \frac{\Gamma'}{\Gamma}(k-t) + \log \left(\frac{N\vert D\vert}{4\pi^2 m}\right) +2 \frac{L'}{L}(1, \epsilon_K) \right) \\
& \left.\qquad - \sum_{n=1}^\infty \sigma_{\cA}(n) r_{\bar{\cA}, \chi}(m\vert D\vert + nN) Q_{k,t}\left( 1+\frac{2 n N}{m\vert D\vert} \right)\right], \\
\end{align*}
where $P_{k,t}$ and $Q_{k,t}$ are the Jacobi functions defined by the formulas \eqref{Pkt} and \eqref{Qkt1} respectively.
\end{theorem}

\begin{proof}
By Lemma \ref{lem:est}, we may apply the holomorphic projection lemma \cite[IV (5.1)]{GZ} to get the result.
\end{proof}

\begin{remark}
   We have $Q_{k,t}(x)=O(x^{-k-t})$ as $x\to +\infty$ by Lemma \ref{lem:asyinfty2}. In particular, the $n^\text{th}$ term of the infinite sum in Theorem \ref{fourcoeff} is $O(n^{\delta+t+1/2-k-t})=O(n^{\delta-k+1/2})$. Since $k\geq 2$, we see that the $n^\text{th}$ term is $O(n^{\delta-3/2})$ and the sum converges. This is true even in the edge case $t=k-1$, in notable contrast with the situation in \cite{GZ} where the parameters $(k,t)=(1,0)$ lead to the infinite sum being divergent. 
\end{remark}

\section{Heights}\label{sec:height}

Let $Y$ be a smooth projective variety over a number field $F$ of dimension $d=2n-1$. Assume that $Y$ admits a regular model $\Y$ that is projective and flat over $\spec \oh_F$. Let $Z_1$ and $Z_2$ be algebraic cycles on $Y$ of codimension $n$. Assume the following conditions: 
\begin{itemize}
    \item[(a)] $Z_1$ has an integral model $\bZ_1$ in $\Y$ with the property that $\bZ_1$ has zero intersection with any algebraic cycle of codimension $n$ on $\Y$ supported on the special fibers.
    \item[(b)] $Z_2$ admits an integral model $\bZ_2$ in $\Y$.
    \item[(c)] $Z_1$ is null-homologous, i.e., its class in $H^{2n}(Y(\C), \C)$ is zero. 
\end{itemize}
Condition $(c)$ implies that there is a Green's current $g_1$ on $Y(\C)$ such that $\frac{\partial \bar{\partial}}{\pi i} g_1 = \delta_{Z_1}$, where $\delta_{Z_1}$ is the current associated to $Z_1$ which maps any real valued differential $\eta\in A^{n,n}(Y_\R)$ to $\int_{Z_{1,\C}(\C)} \eta$. 

Under the above assumptions, the height pairing of $Z_1$ and $Z_2$ is defined as 
\[
\langle Z_1, Z_2 \rangle^{\GS} := (-1)^n (\bZ_1, g_1) \cdot (\bZ_2, g_2),
\]
where $g_2$ is any Green's current for $Z_{2, \C}$. 
The intersection on the right hand side is the arithmetic intersection pairing of Gillet--Soul\'e 
\[
\widehat{\CH}^n(\Y)_{\R} \times \widehat{\CH}^n(\Y)_{\R} \lra \R,
\]
defined for arithmetic cycles $\widehat{T}_i=(\mathbf{T}_i, \alpha_i)$, $i=1,2$, that are irreducible and intersect properly, by
\[
\widehat{T}_1\cdot \widehat{T}_2 := \log(\vert \Gamma(\mathbf{T}_1\cdot \mathbf{T}_2, \oh) \vert) + \int_{\mathbf{T}_{2,\C}(\C)} \alpha_1 + \int_{\Y_{\C}(\C)} \alpha_2 \beta_1, 
\]
where $\beta_1=\delta_{T_1}-\frac{\partial \bar{\partial}}{\pi i} \alpha_1$ \cite[(4.2.2.2)]{GS}. Because $\Y$ is regular, there is a moving lemma for arithmetic Chow groups, and the arithmetic intersection pairing is defined for all $\widehat{T}_1$ and $\widehat{T}_2$ \cite[p. 104]{zhang}.

\subsection{Local height decomposition}

The arithmetic intersection pairing admits a decomposition into a sum of local intersections for all places $v$ of $F$: 
\[
\widehat{T}_1\cdot \widehat{T}_2= \sum_v (\mathbf{T}_{1,v}\cdot \mathbf{T}_{2,v}) \epsilon_v,
\]
where $\epsilon_v=\log q_v$ with $q_v=\vert \oh_F/\fp \vert$ if $v$ is a non-archimedean place corresponding to a prime ideal $\fp$, and $\epsilon_v=2$ if $v$ is a complex archimedean place. 
It follows that the height pairing also admits a decomposition into local components
\[
\langle Z_1, Z_2 \rangle^{\GS}=\sum_v \langle Z_1, Z_2 \rangle_{v} \epsilon_v. 
\]
These local intersections are only defined when $\bZ_1$ and $\bZ_2$ do not intersect in the generic fiber $Y=\Y_F$, and they depend on the representatives of the cycles $Z_1$ and $Z_2$. 
\subsection{Local heights in cases of improper intersection}

For the explicit cycles that we work with, we do not wish to use a moving lemma and it becomes crucial to define local heights even when the cycles do intersect in the generic fiber. Below we give such a local decomposition when $Y$ is a compactification of an abelian scheme $f \colon \mathcal{A} \to C$ of relative dimension $2n-2$ over a smooth curve $C$, and when the cycles $Z_1$ and $Z_2$ are supported on sub-abelian varieties in fibers of $f$.\footnote{We prove this below under an additional technical assumption but we suspect the method works in general.}  This is done by combining a result of Zhang \cite{zhang} with Brylinski's formalism for heights on local systems over curves \cite{brylinski}. 

 Now suppose $Z_1$ and $Z_2$ are such cycles (of codimension $n$ on $\mathcal{A}$) supported on the fiber of $f$ above a single point $x \in C$. 

\subsubsection{Archimedean local height}
Let $v$ be an archimedean place of $F$. A choice of polarization on $\mathcal{A}$ makes $V = R^{2n-2}f_*\Z(n-1)$ a weight 0 polarized variation of Hodge structures (PVHS) on $C_v=C\times_v \C$ with associated vector bundle $\mathcal{V}$  (see Definition \ref{def:PVHS} for the precise definition). Brylinski \cite{brylinski} attaches to such a PVHS a Green's kernel and defines a local height pairing on Hodge classes. Let us write $\bar{C}$ for the compactification of $C$. There is a canonical way to extend $\mathcal{V}$ to a vector bundle $\bar{\mathcal{V}}$ on $\bar{C}_v$ due to Deligne (characterized by certain properties) \cite[\S 1 p. 4]{brylinski}. Let us write $\bar{\mathcal{V}}_{0, \mathbb{R}}$ for the $C^\infty$ vector bundle of sections of $\bar{\mathcal{V}}$ which are real and of type $(0,0)$. Also write $\pr_1$ and $\pr_2$ for the projections $\bar{C}\times \bar{C} \lra \bar{C}$. 

\begin{proposition}[Proposition 2.9 \cite{brylinski}]\label{prop:bry}
    If $V$ has no non-zero global section, there exists a unique $C^\infty$-section $G$ of $\mathrm{Hom}(\pr_1^{-1} \bar{\mathcal{V}}_{0,\mathbb{R}}, \pr_2^{-1} \bar{\mathcal{V}}_{0,\mathbb{R}})$ over the complement of the diagonal in $\bar{C}_v\times \bar{C}_v$, such that 
    \begin{enumerate}
        \item $\square_2 G=0$, where $\square_2$ is the Laplacian associated with $\mathcal{V}$ in the second variable;
        \item $G(x,y)-\log\vert z(x)-z(y)\vert$ is bounded near any point $(a,a)\in \bar{C}_v\times \bar{C}_v$, if $z$ is a local coordinate on $\bar{C}_v$ near $a$.
    \end{enumerate}
\end{proposition}

The $C^\infty$-section $G$, called the Green's kernel attached to $V$, can be used to define the Brylinski pairing. If $x \neq y \in C_v(\C)$ and $v_x^1$ and $v_y^2$ are type $(0,0)$ classes supported in the fibers of $V$ above $x$ and $y$ respectively, Brylinski defines \cite[Prop. 2.11]{brylinski}
\begin{equation}\label{eq: brylinski formula}\left\langle  v^1_x, v^2_y \right\rangle^{\mathrm{Br}} :=\langle G(x,y)v^1_x,v^2_y\rangle_y,
\end{equation}
where $\langle\, ,\, \rangle_y$ denotes the Hermitian pairing induced by the polarization on the fiber of $V$ above $y$. This extends by linearity to a pairing on fibral classes supported on disjoints sets of fibers. 
If $v_x^1$ and $v_y^2$ are cycle classes of homologically trivial cycles $Z$ and $Z'$ supported on the fibers above $x$ and $y$ (distinct points) respectively, then \cite[Appendix]{brylinski}
\begin{equation}\label{eq: brylinski local height}
    \left\langle v^1_x,v^2_y \right\rangle^{\mathrm{Br}} = \langle Z, Z'\rangle_v.
\end{equation}
(See also \cite{mullerstach} for various definitions of archimedean local height pairing and comparisons between them.)

Recall that $Z_1$ and $Z_2$ are supported on the fiber of $f$ above a single point $x \in C$. Their images, denoted $v^1_x$ and $v^2_x$ respectively, under the cycle class map relative to the fiber $f^{-1}(x)$ (after base-change to $\C$ via $v$) lie in the fiber of $V$ above $x$. Let $\eta$ be a $(n-1,n-1)$-form on $Y_v(\C)$ such that for all points $y$ close enough to $x$, the restriction $\eta_y$ to the fiber above $y$ satisfies $\partial \eta_y = 0 = \bar\partial \eta_y$, and moreover the restriction $\eta_x$ recovers the cycle class of $v^2_x$, thought of as a differential form.  For $y \in C_v(\C)$ close to $X$, let $v^2_y \in V_y$ correspond to $\eta_y$. 

 Let $t$ be a uniformizer on $C_v(\C)$ at $x$.  We let   
 \begin{equation}\label{locint}
 \langle Z_1,Z_2\rangle_v = \lim_{y \to x} \left(\langle G(x,y)(v^1_x, v^2_y)\rangle_y + (-1)^n(Z_1\cdot Z_2)_x \log|t(y)|_v\right),
 \end{equation}
 where $(Z_1 \cdot Z_2)_x$ is the usual (geometric) intersection pairing in the fiber $f^{-1}(x)$. 

\subsubsection{Non-archimedean local height}
If $v$ is a non-archimedean place of $F$, we define
\[\langle Z_1,Z_2\rangle_v = (Z_1\cdot Z_2)_x\ord_v\, d_xt,\]
where $\ord_v\, d_xt$ is defined as in \cite[p. 106]{zhang}.

\subsubsection{Global conjecture}

\begin{conjecture}\label{conj: height decomposition}
  Suppose $Z_1$ and $Z_2$ are supported on a single common fiber of $f$. Then 
  \[\langle Z_1, Z_2 \rangle^{\GS} = \sum_v \langle Z_1,Z_2\rangle_v,\]
  where the local pairings are defined as above.
\end{conjecture}

For our purposes it will be enough to prove the following special case.

\begin{theorem}\label{thm:ari}
With hypotheses and notation as above, suppose further that there exists a dense set of points $W \subset U$ such that the class $v_y^2$ of \eqref{locint} is the class of an algebraic cycle $Z_y$ in $f^{-1}(y)$ for all $y \in W$. Then Conjecture \ref{conj: height decomposition} holds. 
\end{theorem}

\begin{proof}
    This follows from the results of Zhang and Brylinski mentioned earlier. Zhang's \cite[Conj. 1.2.1 \& Thm. 1.2.2]{zhang} says that Conjecture \ref{conj: height decomposition} is true with archimedean local heights given by  
    \[
    (-1)^n\lim_{y \to x} \left(\int_{\mathcal{A}_y} g_1\eta+(Z_1\cdot Z_2)_x \log|t(y)|_v\right),
    \]
    where $g_1$ is the Green's current for $Z_1$. In the fiber above some $y \in W$, the integral $(-1)^n\int_{\mathcal{A}_y} g_1\eta$ is by definition the local Gillet--Soul\'e height $\langle Z_1, Z_y\rangle_v$, which is in turn equal to $\langle v_x^1, v_y^2\rangle^{\mathrm{Br}}$ by (\ref{eq: brylinski local height}). Applying (\ref{eq: brylinski formula}), we obtain
    \[
    (-1)^n\int_{\mathcal{A}_y} g_1\eta=\langle G(x,y)v_x^1, v_y^2\rangle_y, \qquad \forall y\in W.
    \]
    The left hand side and the right hand side can both be viewed as function on $U\setminus \{ x \}$. As such, they are continuous, and since they agree on a dense set they must be equal. We conclude that Conjecture \ref{conj: height decomposition} is true with archimedean local heights given by \eqref{locint}.
\end{proof}

\section{Generalized Heegner cycles}\label{s:GHC}

Let $K/\Q$ be an imaginary quadratic field, and let $H$ be its Hilbert class field. Let $(A,  \iota)$ be an elliptic curve $A/H$ with complex multiplication $\iota \colon \oh_K \hookrightarrow \End(A)$ by $\oh_K$.

\subsection{Generalized Kuga--Sato varieties}

For any $M \geq 3$, let $X(M)$ be the modular curve (over $H$) parameterizing pairs $(E,(\alpha,\beta))$, where $E$ is an elliptic curve and $(\alpha,\beta)$ is a basis for $E[M]$.  Note that $X(M)$ is geometrically disconnected since we impose no condition on the Weil pairing $\langle \alpha, \beta \rangle$. 
For $k \geq 1$, let $W_{2k-2}$ be the smooth and projective Kuga--Sato $(2k-1)$-fold over $X(M)$, which is birational to the $(2k-2)$-th fiber power  $\bar{\cE}^{(2k-2)}$ of the universal generalized elliptic curve over $X(M)$ \cite[\S2]{zhang}. 

For $\ell\geq 0$, we consider the variety 
\[X = X_{M, k,t} := W_{2k-2}\times_H A^{\ell},\] 
which we view as fibered over $X(M)$.  It is smooth and projective over $H$, of dimension $2k+\ell-1$, birational to $\bar{\cE}^{(2k-2)} \times_H A^{\ell}$.

\subsection{Cycles}\label{subsec:gen Heeg cycles}

Let $F$ be a finite extension of $H$ and $Q \in X(M)(F)$ a point corresponding to a full level $M$-structure $B=(\alpha,\beta)$ of $A$.  
Any isogeny $\varphi \colon A_F \to A'$ over $F$ of degree prime to $M$  gives another point $Q_\varphi:=(A', \varphi(B))$  of $X(M)(F)$. 
The fiber of $\bar{\cE} \lra X(M)$ above $Q_\varphi$ is isomorphic to $A'$. The fiber $X_{Q_\varphi}$ of $X\lra X(M)$ over $Q_{\varphi}$ is isomorphic to $(A')^{2k-2}\times A_{F}^{\ell}$. 

For a homomorphism $\varphi \in \Hom(A_1, A_2)$ between elliptic curves, we denote by $\Gamma_{\varphi}\subset A_1 \times A_2$ its graph, and by $\Gamma^T_{\varphi}\subset A_2 \times A_1$ the transpose of $\Gamma_\varphi$. We choose a square root $\sqrt{D} \in \oh_K$ and view $\sqrt{D}$ as an element of $\End(A)$ if no confusion can arise, e.g., in the notation $\Gamma_{\sqrt{D}} \subset A \times A$. 

Assume from now on that $\ell = 2t$ is even and $0<t \leq k-1$. For $\varphi \colon A_F \to A'$ as above, define 
\[
Y_{\varphi}:= \Gamma_{\sqrt{D}}^{k-1-t}\times (\Gamma_{\varphi}^T)^{2t} \subset (A' \times A')^{k-1-t} \times (A' \times  A_F)^{2t} \simeq X_{Q_\varphi} \subset X_{F}.
\]
For each ideal $\mathfrak{a} \trianglelefteq \oh_K$ coprime to $M$, we have an isogeny $\phi_{\fraka}: A\lra A/A[\fraka] =: A^\fa$. We write $Q^\fa:=Q_{\phi_\fa}$ and $Y^\fa:=Y_{\phi_\fa}$. 
When $\fa=\oh_K$, we recover the point $Q=Q^{\oh_K}=(\alpha,\beta)\in X(M)(F)$. We write $Y^\fa = Y^\fa_Q$ if we want to emphasize the dependence on the initial choice of point $Q$. 

Let $\corr^0(X,X)_K$ be the ring of algebraic correspondences modulo rational equivalence on $X$ with coefficients in $K$, as in \cite[p. 755]{bdp2}.   
We define the idempotent \[\epsilon:=\epsilon_W \otimes \epsilon_{2t}\in \corr^0(X,X)_K,\] 
where $\epsilon_W$ is the Deligne--Scholl \cite{scholl} projector on $W_{2k-2}$ (i.e.,\ the Chow motive of weight $2k$ cusp forms) and, for any $m\geq 1$, $\epsilon_m$ is the idempotent
\[
\epsilon_{m}:=\left( \frac{\sqrt{D}+[\sqrt{D}]}{2\sqrt{D}} \right)^{\otimes m} \circ \left( \frac{1-[-1]}{2} \right)^{\otimes m} \in \corr^0(A^{m},A^{m})_K.
\]
Define also
\[
\bar{\epsilon}_{m}:=\left( \frac{\sqrt{D}-[\sqrt{D}]}{2\sqrt{D}} \right)^{\otimes m} \circ \left( \frac{1-[-1]}{2} \right)^{\otimes m} \in \corr^0(A^{m},A^{m})_K,
\]
and
\begin{equation}\label{eq:kappam}
    \kappa_{m}=\epsilon_{m}
+ \bar \epsilon_{m}.
\end{equation}
Define
\[
\bar \epsilon:=\epsilon_W\otimes\bar \epsilon_{2t}\in \corr^0(X,X)_K,
\]
\[
\epsilon'=\epsilon_W\otimes\kappa_{2t}\in \corr^0(X,X)_\Q.
\]
For any Weil cohomology theory $H^*(-)$, we have \cite[Prop. 2.4]{bdp1} 
\begin{align*}
\epsilon_W H^*(W_{2k-2}) & = \epsilon_W H^{2k-1}(W_{2k-2})\\ 
\kappa_{2t}H^*(A^{2t}) & = \kappa_{2t}H^{2t}(A^{2t}),
\end{align*}
and so $\epsilon'H^*(X) = \epsilon'H^{2k + 2t-1}(X)$.  

\begin{lemma}\label{lem: null homologous}
    For all $B=(\alpha, \beta)$ and  $\varphi \colon A_F \to A'$ as above, the cycles $\epsilon Y_\varphi$ and $\bar{\epsilon}Y_\varphi$ are null-homologous. In other words,
\[
\epsilon Y_{\varphi} \in \CH^{k+t}(X_F)_{0, K} \qquad \text{ and } \qquad \bar \epsilon Y_\varphi \in \CH^{k+t}(X_F)_{0, K},
\]
where $\CH^j(X_F)_{0,K}$ denotes the kernel of the cycle class map $\CH^j(X_F)_K \to H^{2j}(X)(j) \otimes K$.
\end{lemma}
\begin{proof}
    This follows from the fact that $\epsilon'$ (and hence $\epsilon$ and $\bar{\epsilon}$) annihilates even degree cohomology.
\end{proof}

\subsection{Generalized Heegner cycles}

Let $N$ be a divisor of $M$, and let $\pi_{M,N} \colon X(M)\to X_0(N)$ be the natural morphism of curves over $H$ given by $(E, \alpha,\beta) \mapsto (E,\langle \frac{M}{N} \alpha \rangle)$.  

We assume from now on the Heegner hypothesis \eqref{HH}:  every prime dividing $N$ splits in $K$.  This guarantees that there exists an ideal $\fn \trianglelefteq \oh_K$ such that $\oh_K/\fn \simeq \Z/N\Z$, and hence $A[N] \simeq A[\fn] \oplus A[\overline{\fn}]$.  The pair $(A,A[\fn])$ corresponds to a Heegner point $P_\fn \in X_0(N)(H)$ \cite{grossHP}. Let $F/H$ be a finite extension splitting the finite $H$-scheme $\pi^{-1}_{M,N}(P_\fn)$. For each $ \fa \trianglelefteq \oh_K$ coprime to $M$, define
\begin{equation}\label{def:Za}
    Z^\fa_\fn = \sum_{Q \in \pi_{M,N}^{-1}(P_\fn)(F)} \epsilon Y_Q^\fa \in \CH^{k+t}(X)_{0,K}
\end{equation}
\begin{equation}\label{def:Zbara}
    \bar{Z}^\fa_\fn = \sum_{Q \in \pi_{M,N}^{-1}(P_\fn)(F)} \bar{\epsilon} Y_Q^\fa \in \CH^{k+t}(X)_{0,K},
\end{equation}
which a priori only live in $\CH^{k+t}(X_F)_K$, but clearly descend to $\CH^{k+t}(X)_{0,K}$. 

Let $\chi$ be an unramified Hecke character of $K$ of type $(2t,0)$. For each ideal class $\cA=[\fraka]$, we choose a representative $\fa$ of norm prime to $M$ and define 
\[
Z_\cA := \chi(\fraka)^{-1} Z_\fn^\fa \in \CH^{k+t}(X)_{0, K(\chi)},
\]
\[
\bar Z_\cA := \chi(\bar\fraka)^{-1} \bar{Z}_\fn^\fa \in  \CH^{k+t}(X)_{0, K(\chi)}.
\]
Assuming the conjectural injectivity of Abel--Jacobi maps, the cycles $Z_\cA$ and $\bar Z_\cA$ do not depend on the choice of representative ideal $\fa\in \cA$ (see \cite[Prop. 4.6]{pGZshnidman}). 
Finally, we define 
\begin{equation}\label{def:delchi}
\Delta_\chi := \frac{1}{\sqrt{\deg(\pi_{M,N})}} \sum_{\cA} (Z_{\cA}+\bar{Z}_{\bar{\cA}})\in \CH^{k+t}(X)_{0, \C}.
\end{equation}

\begin{remark}
    The cycle $\Delta_\chi$ depends on the choice of $\fn$, but its Beilinson--Bloch height 
    is independent of this choice, since the Atkin--Lehner involutions on $X_0(N)$ act transitively on the points $P_\fn$. Likewise the height of $\Delta_\chi$
    is independent of $M$ due to the square root normalization factor.
\end{remark}

\subsection{Integral models}\label{ss:integral}
In order to compute heights we must use integral models of generalized Kuga--Sato variety and generalized Heegner cycles, as discussed in Section \ref{sec:height}. 

Suppose $M$ is a multiple of $N$ such that $M = N_1N_2$ with $(N_1,N_2) = 1$ and $N_i \geq 3$. We will compute heights over a finite extension $F/H$ which is unramified at $M$ and with the property that $A$ has everywhere good reduction over $F$. These two conditions are compatible: since $N$ is prime to $D = \Disc(K)$, we can always choose $A$ and $M$ so that $A$ has good reduction at all primes dividing $M$, and we may then take $F$ to be the minimal number field over which $A$ attains good reduction. We may and do moreover assume that all primes dividing $M$ split in $K$. We write $X$, $X(M)$, et cetera,  instead of $X_F, X(M)_F$, etc. 

Let $\W \to \oh_F$ be the proper regular model for $W_{2k-2}$ constructed in \cite[\S2.2]{zhang}, and let $\mathbf{A}$ be the N\'eron model for $A$ over $\oh_F$. Since $A/F$ has everywhere good reduction, $\mathbf{A} \to \oh_F$ is smooth and we conclude:

\begin{proposition}
    $\X := \W \times_{\oh_F} \mathbf{A}^\ell$ is a proper regular model for $X$ over $\oh_F$. 
\end{proposition}

As in \cite[3.1]{zhang}, we may extend the cycles $Z_\fn^\fa$ and $\bar{Z}_\fn^\fa$ to cycles on $\X$ so that the  height pairing $\langle \Delta_\chi, \Delta_\chi\rangle^{\GS}_F$ is well-defined (i.e., conditions (a)-(c) of Section \ref{sec:height} are satisfied). Just as in \cite{zhang}, we are implicitly making use of the fact that in characteristic $p \mid M$, generalized Heegner cycles live in fibers above ordinary (and hence smooth!) points of $X(M)_{\F_p}$, so that the Zariski closure of the cycles in the generic fiber automatically satisfy condition $(a)$. 

\subsection{Hecke operators}

The usual Hecke operators on higher weight modular forms come from Hecke correspondence on $W_{2k-2}$, and these extend to $\W$ as in \cite[2.2]{zhang}. We define correspondences $T_m$ on $X$ (and similarly $\X$) by pullback from $W_{2k-2}$ (resp.\ $\W$). 

\section{Variations of Hodge structures}\label{s:PVHS}

In order to use Brylinski's formalism to compute heights of generalized Heegner cycles, we need to construct the Green's kernel of Proposition \ref{prop:bry} associated to a certain polarized variation of Hodge structures (PVHS) over $X(M)$. In this section we define the relevant PVHS, and  in the next section we compute the associated Laplacian.

\subsection{Polarized variations of Hodge structures}

Let $T$ be a subring of $\R$, and let $H$ be a $T$-Hodge structure of weight $m$.
\begin{definition}\label{def:pola}
A polarization on $H$ is a $(-1)^m$-symmetric pairing $( \;, \; ) \colon H \times H \lra T$, whose base-change $H_\C \times H_\C \lra \C$ satisfies $(H_\C^{p,q}, H_\C^{p',q'})=0$ unless $(p',q')=(q,p)$, as well as  $i^{p-q}(v, \bar{v})>0$ for all $0\neq v\in H_\C^{p,q}$. 
\end{definition}

\begin{remark}\label{rem:herm}
The Hermitian form $\langle v,w \rangle:=i^{p-q}(v, \bar{w})$ on $H_\C^{p,q}$ is positive-definite. Moreover, the orthogonal complement of $\Fil^{p} H_{\C}$ with respect to $(\;, \; )$ is $\Fil^{m-p+1} H_\C$.
\end{remark}

Let $X$ be a complex manifold and $V$ a local system of finitely generated $T$-modules over $X$. Let $V_\C=V\otimes_T \C$ be the local system of finite dimensional complex vector spaces over $X$, and let $\cV:=\oh_X \otimes_\C V_\C$ be the corresponding holomorphic vector bundle over $X$ with integrable connection $\nabla:=\partial \otimes 1$. 

\begin{definition}\label{def:PVHS}
A $T$-PVHS with underlying local system $V$ of weight $m$ is the data of:\footnote{Technically speaking, the PVHS in Definition \ref{def:PVHS} is $(V, \Fil^{\bullet} \cV)$, but in practice we will simply call $\cV$ the PVHS, when no confusion can arise.} 
\begin{itemize}
    \item a finite decreasing filtration by holomorphic subbundles $\Fil^{p+1} \cV \subset \Fil^p \cV$, each locally a direct factor in $\cV$;
    \item a flat bilinear pairing $( \: , \: ) \colon V \times V \lra T(-m)$;
\end{itemize}
with the following properties:
\begin{itemize}
    \item[(i)] For each $x\in X$, the induced filtration $\Fil^p_x:=\Fil^p \cV_x$ on the fiber $\cV_x$ is the Hodge filtration of a pure Hodge structure of weight $m$;
    \item[(ii)] $\nabla(\Fil^p \cV) \subset \Omega^1_X \otimes_{\oh_X} \Fil^{p-1} \cV$ (Griffiths transversality);
    \item[(iii)] For each $x\in X$, $(2\pi i)^m( \: , \: )_x$ gives a polarization on $\cV_x$ in the sense of Definition \ref{def:pola}. 
\end{itemize}
\end{definition}

\begin{remark}\label{rem:flat}
Let $\cV^\infty:=\mathcal{C}^\infty_X\otimes\cV$ be the associated $C^\infty$-bundle with integrable connection $D:=d\otimes \id + \id \otimes \nabla$. The bilinear pairing $(\:, \:)$ of Definition \ref{def:PVHS} (2) induces an $\oh_X\otimes \overline{\oh_X}$-linear pairing $(\:, \:) \colon \cV^\infty \otimes_{\mathcal{C}_X^\infty} \overline{\cV^\infty} \lra \mathcal{C}_X^\infty$. The flatness in Definition \ref{def:PVHS} (2) means that  
    $d(u, \bar{v})=(Du, \overline{v})+(u, \overline{Dv}).$
\end{remark}

\subsection{PVHS over modular curves}

Let $M\geq 3$ be an integer. Let $\bar{C}=X(M)_{\C}$ denote the compact modular curve of full level $M$ structure, base-changed to $\C$. The open modular curve $Y(M)_{\C}$ will be denoted by $C$, and we let $\pi=\pi_{\cE} \colon \cE \lra C$ denote the corresponding universal elliptic curve over $C$. Recall that $C(\C)$ is the disjoint union of $\varphi(M)$ copies of $\Gamma(M)\backslash\cH$.

Consider the relative de Rham cohomology sheaf on $C$
\[
\mathcal{L} := \mathbb{R}^1 \pi_* (0 \lra \mathcal{O}_{\cE}\lra \Omega^1_{\cE/C} \lra 0)=\mathbb{H}^1_{\dR}(\cE/C).
\]
It is a vector bundle of rank $2$ over $C$, whose fibers are given by $\mathcal{L}_{Q} = H^1_{\dR}(E_z)$, where $Q=\Gamma(M)z\in C(\C)$ and $E_z:=\pi^{-1}(z)=\C/\langle 1, z\rangle$. This vector bundle comes equipped with an integrable connection 
\[
\nabla \colon \mathcal{L} \lra \Omega^1_{C} \otimes \mathcal{L}
\]
called the Gauss--Manin connection, as well as a canonical pairing 
\[
( \: , \: ) \colon \mathcal{L} \times \mathcal{L} \lra \oh_{C},
\]
given on the fiber at $Q=\Gamma(M)z\in C(\C)$ by 
\begin{equation}\label{Poin}
( \: , \: )_z \colon H^1_{\dR}(E_z)\times H^1_{\dR}(E_z) \lra \C, \qquad ( \omega_1, \omega_2 )_z= \int_{E_z} \omega_1\wedge \omega_2.
\end{equation}

\begin{remark}
Together with the Hodge filtration given by the line bundle $\Fil^1 \cL=\underline\omega:=\pi_*\Omega^1_{\cE/C}$,
the pairing 
$
( \: , \: )_{-1} \colon \mathcal{L} \times \mathcal{L} \lra \oh_{C}(-1)
$
given on the fiber at $Q=\Gamma(M)z\in C(\C)$ by 
\begin{equation}\label{Poin-1}
( \: , \: )_{-1, z} \colon H^1_{\dR}(E_z)\times H^1_{\dR}(E_z) \lra \C, \qquad ( \omega_1, \omega_2 )_{-1, z}= \frac{1}{2\pi i}\int_{E_z} \omega_1\wedge \omega_2
\end{equation}
makes $\cL$ into a PVHS of weight $1$ with underlying local system $R^1\pi_* \Z$, in the sense of Definition \ref{def:PVHS}. 
\end{remark}

\begin{remark}\label{rem:canext}
The vector bundle $\mathcal{L}$ admits a canonical extension to a vector bundle $\bar{\mathcal{L}}$ over $\bar{C}$ (see \cite[p. 1043-1044]{bdp1} or \cite[p. 4]{brylinski}). The Gauss--Manin connection extends to a connection with log poles 
\[
\nabla \colon \bar{\mathcal{L}} \lra \Omega^1_{\bar{C}}(\log Z) \otimes \bar{\mathcal{L}},
\]
where $Z$ denotes the cuspidal divisor $\bar{C}\setminus C$. The above pairing also admits an extension. Because the cycles defined in Section \ref{s:GHC} live in non-cuspidal fibers of Kuga--Sato varieties, we will not have any need for the defining properties of these extensions. 
\end{remark}

Given $n\geq 1$, we define vector bundles $\mathcal{L}_n:=\sym^n \mathcal{L}$ and write $\nabla_n \colon \mathcal{L}_n \lra \Omega^1_{C} \otimes \mathcal{L}_n$ for the induced connections and 
\[
( \: , \: )_n \colon \mathcal{L}_n \times \mathcal{L}_n \lra \oh_{C},
\]
for the induced pairing given by 
\[
( x_1\ldots x_n , y_1 \ldots y_n )_n:= \frac{1}{n!} \sum_{\sigma\in S_n} ( x_1, y_{\sigma(1)})\ldots ( x_n, y_{\sigma(n)}).
\]

\begin{remark}
The normalized pairing $$( \: , \: )_{-n}:=(2\pi i )^{-n}( \: , \: )_{n} \colon \mathcal{L}_n \times \mathcal{L}_n \lra \oh_{C}(-n)$$ makes $\cL_n$ into a PVHS of weight $n$ with underlying local system $\sym^n R^1 \pi_* \Z$, in the sense of Definition \ref{def:PVHS}.
\end{remark}

Furthermore, for $m\geq 0$, we define vector bundles 
\[
\cL_{n,m}:=\cL_n \otimes \kappa_m H^{m}(A^m_\C),
\]
with $\kappa_m$ defined by \eqref{eq:kappam}.
We naturally get a connection 
\[
\nabla_{n,m} \colon \cL_{n,m}\lra \Omega^1_{C}\otimes \cL_{n,m}
\]
and a pairing 
\[
( \: , \: )_{n,m}\colon \mathcal{L}_{n,m} \times \mathcal{L}_{n,m} \lra \oh_{C},
\]
using the canonical pairing \eqref{Poin} on $H^1_{\dR}(A_{\C})$.

\begin{remark}
The normalized pairing $$( \: , \: )_{-n, -m}:=(2\pi i )^{-(n+m)}( \: , \: )_{n, m} \colon \mathcal{L}_{n,m} \times \mathcal{L}_{n,m} \lra \oh_{C}(-(n+m))$$ makes $\cL_{n,m}$ into a PVHS of weight $n+m$ with underlying local system $\sym^n R^1 \pi_* \Z\otimes \kappa_m H^m_B(A^m, \Z)$.
\end{remark}

Consider the local systems
\[
\mathbb{L}^B := R^1 \pi_* \Z, \qquad \mathbb{L}^B_n := \sym^n \bL^B, \qquad 
\mathbb{L}^B_{n,m} := \bL^B_n \otimes \kappa_m H^m_B(A^m, \Z),
\]
with fibers at $Q=\Gamma(M)z\in C(\C)$ given respectively by the Betti cohomology groups 
\[
\mathbb{L}^B(Q) := H^1_B(E_z, \Z), \; \mathbb{L}^B_{n}(Q) := \sym^n H^1_B(E_z, \Z), \; 
\mathbb{L}^B_{n,m}(Q) := \sym^n H^1_B(E_z, \Z) \otimes \kappa_m H^m_B(A^m, \Z).
\]
The associated complex local systems 
\[
\mathbb{L} := \C \otimes_\Z \mathbb{L}^B, \qquad \mathbb{L}_n= \C \otimes_\Z\mathbb{L}^B_n , \qquad 
\mathbb{L}_{n,m} := \C \otimes_\Z\mathbb{L}^B_{n,m}
\]
are the sheaves of horizontal sections of $(\cL, \nabla), (\cL_n, \nabla_n)$, and $(\cL_{n,m}, \nabla_{n,m})$ over $C(\C)$ (in the complex topology). In other words, we have 
\[
\cL=\oh_{C}\otimes_\C \bL , \qquad 
\cL_n=\oh_{C}\otimes_\C\bL_n, \qquad
\cL_{n,m}=\oh_{C}\otimes_\C\bL_{n,m}.
\]

Later, we will be interested in the case were $n=2k-2$ and $m=2t$, and we will consider the vector bundle
\begin{equation}\label{bundleW}
    \mathcal{W}=\mathcal{W}_{k,t}:= \cL_{2k-2, 2t}(k+t-1)= \cL_{2k-2}(k-1) \otimes \kappa_{2t} H^{2t}_{\dR}(A_{\C}^{2t})(t).
\end{equation}
The pairing 
\[
(\:, \:)_{k,t}^0:=(\:, \:)_{-2(k-1),-2t} \colon  \mathcal{W}\times  \mathcal{W}  \lra \cO_{C}
\]
makes $\mathcal{W}$ into a PVHS of weight $0$, in the sense of Definition \ref{def:PVHS}, with underlying local system 
\[
    W=W_{k,t}:=(\sym^{2k-2} R^1 \pi_* \Z)(k-1) \otimes \kappa_{2t} H_B^{2t}(A^{2t}, \Z)(t).
\] 
In view of Remark \ref{rem:canext}, the PVHS $\mathcal{W}$ admits a canonical extension $\bar{\mathcal{W}}$ to a vector bundle on $\bar{C}$ equipped with a connection with log poles.

\subsection{PVHS over the upper-half plane}\label{s:univellH}

We use the following conventions (see also Section \ref{s:conv}): a point of the upper-half plane $\cH$ will be denoted by $z=x+iy$, while the notation $\tau\in \cH$ will be reserved for imaginary quadratic points. The complex coordinate on the elliptic curve $\C/\langle 1, z\rangle$ will be denoted by $w$, and the same goes for a CM elliptic curve $\C/\langle 1, \tau\rangle$.

Consider the projection map 
\[
\pr \colon \bigsqcup\mathcal{H}\lra \bigsqcup \Gamma(M)\backslash \mathcal{H}=C(\C),
\]
with the disjoint union indexed by the primitive $M^{\text{th}}$ roots of unity.
Using this projection, we pull all the previously defined structures back to obtain vector bundles and local systems over $\bigsqcup\mathcal{H}$: 
\[
\tilde{\mathcal{L}}_{n}:=\pr^*(\mathcal{L}_{n}), \qquad \tilde{\mathbb{L}}_{n}:=\pr^*(\mathbb{L}_{n}), \qquad \text{etc...}
\]

Let $W_{2k-2}^0:=W_{2k-2}\times_{\bar{C}} C=(\cE)^{2k-2}$ and $X^0:=X\times_{\bar{C}} C=(\cE)^{2k-2}\times_H A^{2t}$. Denote by $\tcE:= \pr^*(\cE)$ the pull-back of the universal elliptic curve with $\Gamma(M)$-level structure with its structural map $\tilde\pi \colon \tcE \lra \bigsqcup\cH$. This is the universal elliptic curve over $\bigsqcup\cH$, or rather the disjoint union of $\varphi(M)$ copies of the universal elliptic curve over $\cH$. The latter can be described as the quotient $\Z^2\backslash (\C\times \cH)$, with $(m,n)\in \Z^2$ acting on $(w, z)\in \C\times \cH$ via the rule
\[
(m,n)\cdot (w, z):=(w+m+nz, z).
\]
The fiber over $z\in \cH$ is clearly $\tilde\pi^{-1}(z)=\C/\langle 1, z \rangle$. Let $\tW^0_{2k-2}:=(\tcE)^{2k-2}$ be the $(2k-2)$-fold fiber product of $\tcE$ over $\bigsqcup\cH$. It can be described as the disjoint union of copies of the quotient $(\Z^2)^{(2k-2)}\backslash (\C^{2k-2}\times \cH)$. We also define $\tX^0:=(\tW^0_{2k-2}\times_H A^{2t})(\C)$. We then have $W^0_{2k-2}(\C)=\Gamma(M)\backslash \tW^0_{2k-2}$ and $X^0(\C)=\Gamma(M)\backslash \tX^0$, where $\Gamma(M)$ acts on $\tcE^0$ via 
\[
\left(\begin{matrix} a &b \\ c & d \end{matrix}\right)\cdot (w, z)=\left( \frac{w}{cz+d}, \frac{az+b}{cz+d} \right).
\]

Taking $n=2k-2$ and $m=2t$, we define the PVHS
\begin{equation}\label{PVHS}
\tilde{\mathcal{W}}=\tilde{\mathcal{W}}_{k,t}:= \tilde{\mathcal{L}}_{2k-2}(k-1) \otimes \kappa_{2t} H^{2t}_{\dR}(A^{2t}_{\C})(t)
\end{equation}
with polarization inducing pairing  
\begin{equation}\label{polaa}
(\:, \:)_{k,t}^0:=(\:, \:)_{-2(k-1),-2t} \colon  \tilde{\mathcal{W}}\times  \tilde{\mathcal{W}}  \lra \cO_{\cH}.
\end{equation}
The underlying local system is
\[
\tilde{W}=\tilde{\bL}^B_{2k-2, 2t}(k+t-1).
\]

\subsection{Connection with the formalism of Brylinski}\label{s:brylinski}

Recall that $w:=a+ib$ denotes the complex coordinate on $E_z(\C):=\tilde{\pi}^{-1}(z)=\C/\langle 1, z \rangle$. 
Then $\underline{\omega}_z$ is generated by $dw$ and $\tilde{\cL}(z)=H^1_{\dR}(E_z)$ admits the (canonical but non-holomorphic) Hodge decomposition 
\[
H^1_{\dR}(E_z)=\C dw \oplus \C d\bar{w}, \qquad H^{1,0}(E_z)=\C dw, \qquad H^{0,1}(E_z)=\C d\bar{w}.
\]
Using the Poincar\'e pairing \eqref{Poin}, we see that 
\begin{equation}\label{Poincalc}
( dw, d\bar{w})_z=\int_{E_z} dw\wedge d\bar{w}=-2i\int_{E_z} dadb = -(z-\bar{z}).
\end{equation}

Let $p_1$ and $p_z$ denote the elements of $H_1(E_z(\C), \Z)$ corresponding to a closed path from $0$ to $1$ and $0$ to $z$ in $E_z(\C)$ respectively. Write $\eta_1$ and $\eta_z$ for the elements in $H^1(E_z(\C), \Q)$, which when viewed in $H_{\dR}^1(E_z)$ satisfy 
\[
( \omega, \eta_1 )_z = \int_{p_1} \omega, \qquad ( \omega, \eta_z )_z = \int_{p_z} \omega, \qquad \forall \omega\in H^1_{\dR}(E_z),
\]
with respect to the Poincar\'e pairing \eqref{Poin}.
Then $( dw, \eta_1 )_{z}=1, ( dw, \eta_z )_{z}=z$, $( d\bar{w}, \eta_1 )_{z}=1$, and $( d\bar{w}, \eta_z )_{z}=\bar{z}$.
It follows from \eqref{Poincalc} that 
\begin{equation}\label{eq:dw}
dw = z \eta_1 - \eta_z \qquad \text{ and } \qquad d\bar{w}=\bar{z} \eta_1 - \eta_z.
\end{equation}
Define the skew-symmetric pairing 
\[
( \:, \: )_{-1} \colon H^1(E_z(\C), \R)\times H^1(E_z(\C), \R) \lra \R(-1), \qquad
(\omega_1,\omega_2)_{-1} := \frac{1}{2\pi i}( \omega_1, \omega_2 )_z=\frac{1}{2\pi i}\int_{E_z} \omega_1\wedge \omega_2.
\]
Then with respect to the basis $\{ \eta_1, -\eta_z \}$, $H^1(E_z(\C), \R)$ is a $2$-dimensional real vector space with a skew-symmetric bilinear pairing given by $(\eta_1, -\eta_z)_{-1}=-1/2\pi i$.

Following Brylinski \cite{brylinski},
we now consider the standard representation $V=\R^2$ of $G:=\GL_2(\R)^+$ with basis $\{ u_1, u_2 \}$. We denote the representation by $\rho \colon G \lra \GL(V)$. In detail, we let $g=\left( \begin{smallmatrix} a&b \\ c&d \end{smallmatrix} \right)\in G$ act on $v:=(\lambda_1, \lambda_2)=\lambda_1 u_1 + \lambda_2 u_2\in V$ via 
\[
g\cdot e= \rho(g)(e)=
\left( \begin{matrix} a&b \\ c&d \end{matrix} \right)\left(\begin{matrix} \lambda_1 \\ \lambda_2 \end{matrix} \right)
=\left(\begin{matrix} a\lambda_1+b\lambda_2 \\ c\lambda_1+d\lambda_2 \end{matrix} \right)=(a\lambda_1+b\lambda_2)u_1+(c\lambda_1+d\lambda_2)u_2.
\]

The group $G$ acts on the upper half-plane via
\[
g\cdot z=\left( \begin{matrix} a&b \\ c&d \end{matrix} \right)\cdot z=\frac{az+b}{cz+d}.
\]
Let $K\subset G$ denote the stabilizer of $i\in \cH$. Then $K$ consists of matrices $\left( \begin{smallmatrix} a&b \\ -b&a \end{smallmatrix} \right)$ such that $a^2+b^2\neq 0$, hence is identified with $\C^\times$ by sending such a matrix to $a+bi$. The upper-half plane $\cH$ is thus identified with the quotient $G\slash K$. From this perspective, it is clear that $$\cV:=G\times^K V=(G\times V)\slash [(gk, v)\sim (g, \rho(k)(v))]$$
is a $G$-equivariant vector bundle on $\cH$. It is trivialized via the (not $G$-equivariant) holomorphic identification 
\[
\varphi \colon \cH \times V \overset{\sim}{\lra} \cV, \qquad (gK, v)\mapsto [(g,\rho(g)^{-1}(v))]. 
\]
The $C^\infty$-sections of $\cV$ are the $C^\infty$-functions $F \colon G \lra V$ satisfying $F(gk)=\rho(k)^{-1}(F(g))$ for all $g\in G$, $k\in K$. In particular, the functions 
\[ 
u_i \colon G\lra V, \qquad g \mapsto \rho(g)^{-1}(u_i),  \qquad i=1,2,
\]
are global sections of $\cV$ (note the slight abuse of notation).
The holomorphic subbundle $\Fil^1 \cV$ is generated by the nowhere vanishing holomorphic section $z u_1 + u_2$.
Define a skew-symmetric bilinear pairing $(\:, \: ) \colon V\times V\lra \R(-1)$ by $(u_1, u_2)=-1/2\pi i$. Then the map
\begin{equation}\label{eq:isomV}
(V, (\:, \:)) \overset{\sim}{\lra} (H^1(\tilde{\pi}^{-1}(z), \R), (\:, \:)_{-1}), \qquad
\{ u_1, u_2 \} \mapsto \{ \eta_1, -\eta_z \}
\end{equation}
is an isomorphism of real vector spaces endowed with bilinear forms.

The pairing \eqref{Poin-1} is a polarization in the sense of Definition \ref{def:pola}. If we define $\langle \:, \: \rangle$ on $\cV_z\times \cV_z$ by $\langle v, w \rangle:= i^{p-q}2\pi i(v, \bar{w})$ for $v\in \cV_z^{p,q}$, then we recover Brylinski's \cite{brylinski} calculation that 
\begin{align*}
\langle z u_1 + u_2, z u_1 + u_2\rangle_z & =2y \\
\langle z u_1 + u_2, \bar{z} u_1 + u_2\rangle_z & =0 \\
\langle \bar{z} u_1 + u_2, \bar{z} u_1 + u_2\rangle_z & =2y.
\end{align*}

We conclude that 
\[
\left\{ \frac{1}{\sqrt{2y}}(z u_1 + u_2), \frac{1}{\sqrt{2y}}(\bar{z} u_1 + u_2) \right\}
\]
is an orthonormal $C^\infty$-basis of the vector bundle $\cV^{\infty}=\cC^{\infty}_{\cH} \otimes \cV$ with respect to the Hermitian form $\langle \:, \: \rangle$ \cite[Lem. 3.1]{brylinski}. 

Let $p\geq 1$ and consider the local system $\sym^{2p}(V)$ together with the pairing
\[
(\:, \:)_{2p} \colon \sym^{2p}(V) \times \sym^{2p}(V) \lra \R(-2p) 
\]
given by 
\begin{equation}\label{pairing:2p}
( v_1\ldots v_{2p}, w_1\ldots w_{2p} )_{2p} := \frac{1}{(2p)!} \sum_{\sigma \in S_{2p}} ( v_1, w_{\sigma(1)}) \ldots ( v_{2p}, w_{\sigma(2p)}).
\end{equation}
Together with this pairing, the associated vector bundle $\sym^{2p}(\cV)$ is a PVHS of weight $2p$ in the sense of Definition \ref{def:PVHS}. The associated positive-definite Hermitian pairing
\[
\langle \:, \: \rangle \colon \sym^{2p}(\cV)^{r,s}\times \sym^{2p}(\cV)^{r,s} \lra \cO_\cH,
\]
as in Remark \ref{rem:herm}, is given by 
\begin{align*}
\langle v_1\ldots v_{2p}, w_1\ldots w_{2p} \rangle_{2p} & := i^{r-s} (2\pi i)^{2p} ( v_1\ldots v_{2p}, \bar{w}_1\ldots \bar{w}_{2p} )_{2p} \\
& = \frac{i^{r-s} (2\pi i)^{2p}}{(2p)!} \sum_{\sigma \in S_{2p}} ( v_1, \bar{w}_{\sigma(1)}) \ldots ( v_{2p}, \bar{w}_{\sigma(2p)}) \\
& =  \frac{1}{(2p)!} \sum_{\sigma \in S_{2p}} \langle v_1, w_{\sigma(1)}\rangle \ldots \langle v_{2p}, w_{\sigma(2p)}\rangle. 
\end{align*}

We now consider the PVHS $\cV_p=\sym^{2p}(\cV)(p)$ of weight $0$ with polarization inducing pairing $( \:, \: )_{2p}^0 \colon \cV_p \times \cV_p \lra \oh_\cH$ defined by \eqref{pairing:2p}. Explicitly, we have
\[
( (2\pi i)^p v_1\ldots v_{2p}, (2\pi i)^p w_1\ldots w_{2p} )_{2p}^0 := \frac{(2\pi i)^{2p}}{(2p)!} \sum_{\sigma \in S_{2p}} ( v_1, w_{\sigma(1)}) \ldots ( v_{2p}, w_{\sigma(2p)}).
\]
Note that the twists make sense (the pairing \eqref{pairing:2p} takes values in $\R(-2p)$).
The associated Hermitian form $\langle \:, \: \rangle_{2p}^0 \colon \cV_p^{s,-s} \times \cV_p^{s,-s} \lra \oh_{\cH}$ of Definition \ref{def:pola} is 
\begin{align*}
 \langle (2\pi i)^p v_1\ldots v_{2p}, (2\pi i)^p w_1\ldots w_{2p} \rangle_{2p}^0 & 
 :=  (i)^{2s} ((2\pi i)^p v_1\ldots v_{2p}, (2\pi i)^p \bar{w}_1\ldots \bar{w}_{2p})_{2p}^0  \\
 & = (-1)^s\frac{(2\pi i)^{2p}}{(2p)!} \sum_{\sigma \in S_{2p}} ( v_1, \bar{w}_{\sigma(1)}) \ldots ( v_{2p}, \bar{w}_{\sigma(2p)}) \\
 & = \frac{1}{(2p)!} \sum_{\sigma \in S_{2p}} \langle v_1, w_{\sigma(1)}\rangle \ldots \langle v_{2p}, w_{\sigma(2p)}\rangle.
\end{align*}

We then calculate
\begin{multline*}
\left\langle (2\pi i)^{p}\left(\frac{1}{\sqrt{2y}}(z u_1 + u_2)\right)^n \left(\frac{1}{\sqrt{2y}}(\bar{z} u_1 + u_2)\right)^{2p-n}, \right. \\
\left. (2\pi i)^{p}\left(\frac{1}{\sqrt{2y}}(z u_1 + u_2)\right)^n \left(\frac{1}{\sqrt{2y}}(\bar{z} u_1 + u_2)\right)^{2p-n}\right\rangle_{2p}^0 
=\binom{2p}{n}^{-1}.
\end{multline*}
It follows that the global sections
\[
v_n:=\frac{\binom{2p}{n}^{1/2} (2\pi i)^p}{(2y)^p}(z u_1 + u_2)^n(\bar{z} u_1 + u_2)^{2p-n},
\]
for $0\leq n\leq 2p$, form an orthonormal $C^\infty$-basis of $\cV_p^{\infty}:=\cC^{\infty}_{\cH}\otimes \cV_p$ with respect to the positive-definite Hermitian form $\langle \:, \: \rangle_p$, and each $v_n$ is pure of type $(n-p, p-n)$. 

Let $\tau_0:=x_0+iy_0\in \cH$ represent the fixed CM elliptic curve $A$. Observe that 
\[
\kappa_{2t} H^{2t}_{\dR}(A_{\C}^{2t})=H^{2t,0}(A^{2t}) \oplus H^{0, 2t}(A^{2t}).
\]
We realize the constant VHS $\kappa_{2t} H^{2t}(A^{2t}, \Q)(t)$ of weight $0$ as arising from the trivial representation $\R^2$ of $G$ with basis $\{ e_1, e_2 \}$. Mimicking what was done above, an orthonormal basis of $C^\infty$-sections is suggestively written as 
\begin{align*}
\mu_0 & :=\frac{(2\pi i)^t}{(2y_0)^t}(\bar{\tau}_0e_1+e_2)^{2t} \\
\mu_{2t} & :=\frac{(2\pi i)^t}{(2y_0)^t}(\tau_0e_1+e_2)^{2t}.
\end{align*}
These sections are of type $(-t, t)$ and $(t, -t)$ respectively. For future use, we let $c:=(2\pi i)^t/(2y_0)^t$.

From now on we take $p=k-1$. Then $\tilde{\cW}=\tilde{\cW}_{k, t}=\cV_{k-1}\otimes \R^2$ is the PVHS defined in \eqref{PVHS}. The $C^\infty$-sections 
\[
w_{n,j}:=v_n \otimes \mu_j, \qquad 0\leq n\leq 2k-2, \qquad j=0, 2t,
\]
form an orthonormal basis of $\tilde{\cW}^{\infty}:=\cC^{\infty}_{\cH} \otimes \tilde{\cW}$ with respect to the Hermitian form associated to the polarization. Each $w_{n,j}$ is of pure type $(n-k+1+j-t, k-1-n+t-j)$. In particular, the $C^\infty$-subbundle $\tilde{\cW}^{\infty}_0$ of sections of pure type $(0,0)$ has rank $2$ with orthonormal basis given by
\[
\{ w_{k+t-1, 0}, w_{k-t-1, 2t}\}.
\]
We have 
\[
w_{k+t-1, 0}=\frac{\binom{2k-2}{k+t-1}^{1/2}(2\pi i)^{k+t-1}}{(2y)^{k-1}(2y_0)^{t}}(z u_1+u_2)^{k+t-1}(\bar{z} u_1+u_2)^{k-t-1}\otimes (\bar{\tau}_0 e_1+e_2)^{2t}
\]
and 
\[
w_{k-t-1, 2t}=\frac{\binom{2k-2}{k-t-1}^{1/2}(2\pi i)^{k+t-1}}{(2y)^{k-1}(2y_0)^{t}}(z u_1+u_2)^{k-t-1}(\bar{z} u_1+u_2)^{k+t-1}\otimes (\tau_0 e_1+e_2)^{2t}.
\]
Observe in particular that $\overline{w_{k+t-1,0}}=w_{k-t-1,2t}$. 

\begin{remark}\label{rem:basis}
The $C^\infty$-subbundle $\tilde{\cW}_{0, \R}$ of real $(0,0)$-vectors (considered in Proposition \ref{prop:bry}) has rank $2$ with basis given by $w(+):=w_{p+t,0}+w_{p-t, 2t}$ and $w(-):=i(w_{p+t,0}-w_{p-t, 2t})$. We prefer to work in the $\C$-vector space $\tilde{\cW}_{0,\R} \otimes K$, which has basis
\begin{equation}\label{eqwww}
w^+(z):=y^t w_{k+t-1, 0}(z) \qquad \text{ and } \qquad w^-(z):=y^{-t} w_{k-t-1, 2t}(z),
\end{equation}
for $z=x+iy\in \cH$.
In this basis, the Laplacian is represented by a  diagonal matrix. 
\end{remark}

\section{Laplacians}\label{s:laplacian}

We compute the Laplacian associated to the PVHS $\tilde{\cW}$ defined in \eqref{PVHS} on sections of type $(0,0)$. This will be used to identify the Green's kernel associated to $\cW$ in the next section. 

\subsection{Hodge star operators}

For general background on Hodge operators, see \cite[Ch. 5]{voisin}. 
Let $\mathcal{C}_\R^\infty$ denote the sheaf of real-valued smooth functions on $\mathcal{H}$ and let $\mathcal{O}_\mathcal{H}$ denote the sheaf of holomorphic functions on $\mathcal{H}$. Let $T_{\R}$ denote the real tangent bundle of $\mathcal{H}$ and let $\mathcal{A}_\R^1$ denote the sheaf of real-valued differential $1$-forms on $\mathcal{H}$. If $z \in \mathcal{H}$ then by definition the dual vector space $(T_{\R, z})^*$ is equal to the fiber $\mathcal{A}^1_{\R, z}$ of $\mathcal{A}_\R^1$ at $z$. 
Consider the Poincar\'e metric given for $z=x+iy\in \mathcal{H}$ on the tangent space $T_{\R, z}$ by 
\[ g_z=\left(\begin{smallmatrix} 1/y^2 & 0 \\ 0 & 1/y^2 \end{smallmatrix}\right). \] 
This defines a Riemannian metric on the upper half-plane $\mathcal{H}$.
The linear map $r : T_{\R, z} \longrightarrow \mathcal{A}^1_{\R, z}$ defined by $Y_z \mapsto g_{z}(\: \cdot \:, Y_z)$ is an isomorphism. If $\alpha_z \in \mathcal{A}^1_{\R, z}$, then let $\alpha_z^\#:= r^{-1}(\alpha_z)$. We define an inner product on $\mathcal{A}^1_{\R, z}$ by 
\[ \langle \alpha_z, \beta_z \rangle_z := g_z(\alpha_z^\#, \beta_z^\#).  \]
One checks that $dx^\#=y^2\frac{\partial}{\partial x}$ and $dy^\#=y^2\frac{\partial}{\partial y}$ so that the inner product is given by the matrix $\left(\begin{smallmatrix} y^2 & 0 \\ 0 & y^2 \end{smallmatrix}\right)$. Finally, we endow $\bigwedge^2 \mathcal{A}^1_{\R, z}$ with the inner product 
\[ \langle dx\wedge dy, dx\wedge dy\rangle_z=\det\left(\begin{smallmatrix} y^2 & 0 \\ 0 & y^2 \end{smallmatrix}\right)=y^4. \]
It follows that the volume form on $\mathcal{H}$ associated to the Poincar\'e metric is given by $\Vol_z=\frac{dx\wedge dy}{y^2}$.

Suppose that a complex ($C^\infty$-bundle) $B$ over $\cH$ is endowed with a positive-definite Hermitian form $\langle \;, \;\rangle$.
Let $\mathcal{A}^{p, q}$ denote the sheaf of complex valued differential forms on $\cH$ of type $(p, q)$. It is equipped with the Hermitian product induced by the Poincar\'e metric. The sheaf of $B$-valued differential forms $\mathcal{A}^{p, q}(B):=\mathcal{A}^{p, q}\otimes_{\mathcal{C}^\infty} B$ of type $(p,q)$ is thus equipped with a positive-definite Hermitian product.
The Hodge star operator for $B$ is an anti-linear isomorphism of sheaves
\[ \bar\ast_{B} : \mathcal{A}^{p, q}(B) \overset{\sim}{\lra} \mathcal{A}^{1-p, 1-q}(B^*) \]
characterised by the following property: if $\beta\otimes \eta\in \mathcal{A}^{p, q}(B)(\cH)$, then for all $z\in \cH$ and all $ \alpha_z \otimes \omega\in \mathcal{A}^{p, q}(B)_z$ we have
\[ \alpha_z \otimes \omega\wedge (\bar\ast_{\cW}(\beta\otimes \eta))_z=\langle \alpha_z\otimes \omega, \beta_z\otimes \eta\rangle_{z} \Vol_z. \]
From this, we deduce that 
\begin{equation}\label{op:star}
\bar\ast_{B}(\beta\otimes \eta)=\overline{\ast \beta} \otimes \eta^*,
\end{equation}
where $\eta^*\in \cW^*$ is the form $\eta^*=\langle -, \eta\rangle$ and 
\begin{equation}\label{hodgestar}
    \ast : \mathcal{A}^k_{\R} \overset{\sim}{\lra} \mathcal{A}^{2-k}_{\R}, \qquad \qquad k=0,1,2,
\end{equation} 
is the Hodge star operator on $\cH$ with respect to the Poincar\'e metric. The latter is a linear bundle isomorphism satisfying 
\begin{equation}\label{eq1}
\begin{array}{lcl}
\ast dx=dy & \qquad \qquad & \ast dy=-dx \\
 \ast 1 = \frac{dx\wedge dy}{y^2} & \qquad \qquad & \ast dx\wedge dy=y^2.
\end{array}
\end{equation} 
One checks that \cite[bottom of p. 168]{huybrechts}
\begin{equation}\label{astbdual}
\bar{\ast}_{B^*}\circ \bar{\ast}_{B}=(-1)^{p+q}
\end{equation}
on $\mathcal{A}^{p, q}(B)$.

We apply this theory to the complex vector bundle $\tilde{\cW}^\infty$ \eqref{PVHS} endowed with the Hermitian pairing given on fibers over $z\in \cH$ by
\begin{multline*}
\langle (2\pi i)^p v_1\ldots v_{2p}\otimes (2\pi i)^t \omega_1,  (2\pi i)^p w_1\ldots w_{2p}\otimes (2\pi i)^t  \omega_2 \rangle_z  \\  = \langle (2\pi i)^p v_1\ldots v_{2p},  (2\pi i)^p w_1\ldots w_{2p}\rangle_{2p, z}^0\langle (2\pi i)^t \omega_1, (2\pi i)^t \omega_2\rangle_{2t, \tau_0}^0.
\end{multline*}
The associated Hodge star operator will be denoted by $\bar\ast_{\cW}:=\bar\ast_{\tilde{\cW}^\infty}$. The one for the dual bundle will similarly be denoted by $\bar\ast_{\cW^*}$. 

\subsection{Gauss--Manin connections}

The holomorphic complex vector bundle $\tilde{\cW}=\oh_{\cH}\otimes_\C \tilde{W}_\C$ \eqref{PVHS} comes equipped with the Gauss--Manin connection 
\[
\tilde{\nabla}:=\partial \otimes 1 \colon \tilde{\cW} \lra \Omega^1_\cH \otimes_{\oh_{\cH}} \tilde{\cW}
\]
satisfying Griffiths transversality
\[
\tilde{\nabla}(\Fil^p)\subset \Omega^1_\cH \otimes \Fil^{p-1}.
\]

\begin{notation}
    In order to make the notation less cumbersome, we drop the tilde in the notations for the bundles and connections in this subsection and the next, and remember that we are working over $\bigsqcup\cH$. 
\end{notation}

We consider the complex $C^\infty$-bundle $\cW^\infty:= \cC^\infty_{\cH}\otimes_{\C} W_{\C}$ and extend the Gauss--Manin connection to $D:= d\otimes 1 \colon \cW^\infty \lra \cA^1_{\C} \otimes_{\cC^\infty_{\cH}} \cW^\infty$. With the description $\cW^{\infty}=\cC^\infty_{\cH}\otimes_{\oh_{\cH}} \cW$, we have 
\[
D=\bar{\partial}\otimes 1 + 1\otimes \nabla=\bar{\partial}_{\cW}+1\otimes \nabla.
\]
(While the operator $d\otimes 1$ is not defined on the holomorphic vector bundle $\cW$, the operator $\bar{\partial}_{\cW}:=\bar{\partial}\otimes 1$ is on the other hand well-defined.)
The pairing $(\: , \:)=(\:, \:)_{k,t}^0$ \eqref{polaa} inducing the structure of PVHS is flat, or horizontal, in the sense of Remark \ref{rem:flat}: 
\begin{equation}\label{horizontal}
d(v,\bar{w})=(D(v), \bar w)+(v, \overline{D(w)}).
\end{equation}
In other words, $D$ is a Hermitian connection in the sense of \cite[Def. 4.2.9]{huybrechts}. It is important to note here that if $\alpha\otimes s\in \cA^1(\cW^\infty)$ and $s'\in \cW^\infty$, then $(\alpha\otimes s, s')=\alpha(s,s')$.
In Zucker's notations \cite[p. 423]{zucker}, define the pairing $((v, w))=(v, \bar{w})$. Then \eqref{horizontal} becomes
\begin{equation}\label{horizontal2}
d((v,w))=((D(v), w))+((v, D(w))).
\end{equation}
Let $C_{\cW}$ denote the Weil operator of $\cW$ given by the direct sum of the scalar operators $i^{p-q}$ on $\cW^{p,q}$. Then the associated Hermitian pairing can be expressed as 
\[
\langle v, w \rangle=((C_{\cW} v, w))=((v, C_{\cW} w)).
\]

On $\cW^\infty= \cC^\infty_{\cH}\otimes_{\C} W_{\C}$, we have
$D=\partial \otimes 1 + \bar\partial\otimes 1$, and we observe by Griffiths transversality that 
\[
D(\cA^{r,s}(\Fil^p))\subset \cA^{r+1,s}(\Fil^{p-1})\oplus \cA^{r,s+1}(\Fil^p).
\]

The complex $C^\infty$-bundle $\cA^1_\C$ decomposes as $\cA^{1,0}\oplus \cA^{0,1}$. Similarly, the PVHS $\cW^\infty$ of weight $0$ admits a Hodge decomposition $\cW^\infty=\oplus_k \cW^{k, -k}$. Following \cite[(1.8)]{zucker}, we have
\[
D(\cA^{r,s}(\cW^{k,-k}))\subset \cA^{r+1,s}(\cW^{k,-k})\oplus \cA^{r+1,s}(\cW^{k-1,-k+1})\oplus \cA^{r,s+1}(\cW^{k,-k}) \oplus \cA^{r,s+1}(\cW^{k+1,-k-1}),
\]
and the Gauss--Manin connection $D$ splits into $4$ components: 
\begin{align*}
\partial'_\cW & \colon \cA^{r,s}(\cW^{k,-k})\lra \cA^{r+1,s}(\cW^{k,-k}) \\
\bar{\partial}'_\cW & \colon \cA^{r,s}(\cW^{k,-k})\lra \cA^{r,s+1}(\cW^{k,-k}) \\
\nabla'_\cW & \colon \cA^{r,s}(\cW^{k,-k})\lra \cA^{r+1,s}(\cW^{k-1,-k+1}) \\
\bar{\nabla}'_\cW & \colon \cA^{r,s}(\cW^{k,-k})\lra \cA^{r,s+1}(\cW^{k+1,-k-1}).
\end{align*}

\begin{lemma}\label{lem:leipniz}
If $\alpha \in \cA^{r,s}$ and $s\in \cW^{k,-k}$, then 
\begin{align*}
\partial'_\cW(\alpha \otimes s) & = \partial(\alpha) \otimes s + (-1)^{r+s}\alpha \wedge \partial'_\cW(s) \\
\bar{\partial}'_\cW(\alpha \otimes s) & = \bar\partial(\alpha) \otimes s + (-1)^{r+s}\alpha \wedge \bar{\partial}'_\cW(s) \\
\nabla'_\cW(\alpha \otimes s) & = (-1)^{r+s} \alpha \wedge \nabla'_\cW(s) \\
\bar{\nabla}'_\cW(\alpha \otimes s) & = (-1)^{r+s} \alpha \wedge \bar{\nabla}'_\cW(s).
\end{align*}
\end{lemma}

\begin{proof}
This follows from the fact that the connection is extended to $D \colon \mathcal{A}^n(\cW) \lra \mathcal{A}^{n+1}(\cW)$ via the Leibniz rule $D(\beta \otimes s)=d\beta \otimes s + (-1)^{n} \beta \wedge D(s)$.
\end{proof}

The Hermitian metric on $\cW$ induces an anti-linear isomorphism 
\begin{equation}\label{dualmap1}
(\;)^* \colon \cW \overset{\sim}{\lra} \cW^*, \qquad  w\lra w^*:=\langle -, w\rangle.
\end{equation}
Since the basis $\{ w_{n,j}$, $0\leq n\leq 2p$, $j=0,2t \}$ is orthonormal with respect to the Hermitian pairing, the dual basis $\{ w_{n,j}^\vee$, $0\leq n\leq 2p$, $j=0,2t \}$ satisfies $w_{n,j}^\vee=\langle -, w_{n,j}\rangle=w_{n,j}^*$. 
We extend \eqref{dualmap1} to an anti-linear isomorphism 
\begin{equation}\label{dualmap}
(\;)^* \colon \cA^{p,q}(\cW) \overset{\sim}{\lra} \cA^{q,p}(\cW^*)
\end{equation}
by declaring that if $\alpha\in \cA^{p,q}$ and $s\in \cA^0(\cW)$, then 
\[
(\alpha\otimes s)^*:=\bar\alpha \otimes s^*.
\]

The complex dual $C^\infty$-bundle $\mathcal{C}_{\cH}^\infty \otimes_{\cO_{\cH}} \cW^*$ has an integral connection $D_{\cW^*}$ defined by
\[
D_{\cW^*}(f)(s)=d(f(s))-f(D(s)).
\]
If $D(s)=\alpha\otimes t\in \cA^1(\cW)$ with $\alpha\in \cA^1_\C$ and $t\in \cA^0(\cW)$, then $f(D(s)):=\alpha f(t)$.
As above, this connection splits into four components $\partial'_{\cW^*}, \bar\partial'_{\cW^*}, \nabla'_{\cW^*},$ and $\bar\nabla'_{\cW^*}$. We extend the connection to $D_{\cW^*} \colon \cA^n(\cW^*) \lra \cA^{n+1}(\cW^*)$ via the usual Leibniz rule $D_{\cW^*}(\beta\otimes t)=d\beta \otimes t +(-1)^n \beta \wedge D_{\cW^*}(t)$.

\begin{lemma}\label{lem:lei}
    For all $A \in \cA^p(\cW)$ and $B\in \cA^q(\cW^*)$,
    the following Leibniz rule holds:
    \[
    d(A\wedge B)=D_{\cW}(A)\wedge B  + (-1)^p A\wedge D_{\cW^*}(B).
    \]
\end{lemma}

\begin{proof}
    Writing $A=\alpha\otimes s \in \cA^p(\cW)$ and $B=\beta\otimes t \in \cA^q(\cW^*)$, we must prove the following equality:
    \begin{equation}\label{Lei}
    d((\alpha\otimes s) \wedge (\beta\otimes t))=D_{\cW}(\alpha \otimes s)\wedge (\beta\otimes t) +(-1)^p (\alpha \otimes s)\wedge D_{\cW^*}(\beta\otimes t).
    \end{equation}
    By definition, we have
    \[
    D_{\cW}(\alpha\otimes s)=d\alpha \otimes s + (-1)^p \alpha \wedge D_{\cW}(s)
    \]
    and
    \[
        D_{\cW^*}(\beta\otimes t)=d\beta \otimes t + (-1)^q \beta \wedge D_{\cW^*}(t).
    \]
    The proof then follows from an easy calculation using the Leibniz rule for the exterior derivative and the definition of the connection $D_{\cW^*}$.
\end{proof}

\begin{lemma}\label{lem:dual}
For all $w\in \cW$, we have 
\[
D_{\cW^*}(w^*)=(C_{\cW}^{-1}D_{\cW}(C_{\cW} w))^*.
\]
\end{lemma}

\begin{proof}
For all $w'\in \cW$, we have
\begin{align*}
    D_{\cW^*}(w^*)(w')&=
    d(w^*(w'))-w^*(D_{\cW}(w'))
    =d\langle w', w\rangle-\langle D_{\cW}(w'), w  \rangle \\
    &= d(( w', C_{\cW}w ))-(( D_{\cW}(w'), C_{\cW} w  )) 
    = (( w', D_{\cW}(C_{\cW} w)  )) \\
    &= (( w', C_{\cW}C_{\cW}^{-1} D_{\cW}(C_{\cW} w)  ))
     = \langle w', C_{\cW}^{-1}D_{\cW}(C_{\cW} w)  \rangle
    =(C_{\cW}^{-1}D_{\cW}(C_{\cW} w))^*(w'),
\end{align*}
where we used \eqref{horizontal2} in the fourth equality.
\end{proof}

\begin{lemma}\label{lem:dual2}
For all $w\in \cW$, we have 
\[
\partial'_{\cW^*}(w^*)=(C_{\cW}^{-1}\bar\partial'_{\cW}(C_{\cW}w))^* \quad \text{ and } \quad \bar\partial'_{\cW^*}(w^*)=(C_{\cW}^{-1}\partial'_{\cW}(C_{\cW}w))^*,
\]
\[
\nabla'_{\cW^*}(w^*)=(C_{\cW}^{-1}\bar\nabla'_{\cW}(C_{\cW}w))^* \quad \text{ and } \quad \bar\nabla'_{\cW^*}(w^*)=(C_{\cW}^{-1}\nabla'_{\cW}(C_{\cW}w))^*.
\]
\end{lemma}

\begin{proof}
    Immediate consequence of Lemma \ref{lem:dual} after identifying the four components.
\end{proof}

Let $p=k-1$ and recall the sections 
\begin{equation}\label{eqpm}
w^-(z)=y^{-t} w_{p-t, 2t}(z) \quad \text{ and } \quad w^+(z)=y^t w_{p+t,0}(z)
\end{equation}
(with $z=x+iy\in \cH$) introduced in Remark \ref{rem:basis}.
For each $0\leq n\leq 2p$, define 
\[
\xi_n=\left(\binom{2p}{n}^{1/2}\frac{(2\pi i)^{p}}{2^p}\right)^{-1}, 
\]
so that 
\[
\xi_{p+t} w^+=y^{t-p}(z u_1+u_2)^{p+t}(\bar{z} u_1+u_2)^{p-t}\otimes \mu_{0}
\]
and 
\[
\xi_{p-t} w^{-}=y^{-p-t}(z u_1+u_2)^{p-t}(\bar{z} u_1+u_2)^{p+t}\otimes \mu_{2t}.
\]
In what follows, the calculations will mostly be carried out for $w^-$. The ones for $w^+$ can be verified similarly. (We also implicitly work in the complex vector space $\tilde{\mathcal{W}}_{0,\R} \otimes K$, as in Remark \ref{rem:basis}.)
\begin{lemma}\label{lem:components}
Given $F\in \cA^{0,0}(\cH)$, we have 
\begin{align*}
\partial'_{\cW}(Fw^-)& =\left( \frac{\partial F}{\partial z}-\frac{2tF}{z-\bar{z}} \right) dz \otimes w^- \\
\bar{\partial}'_{\cW}(Fw^-) & = \frac{\partial F}{\partial \bar{z}} d\bar{z}\otimes w^- \\
\nabla'_{\cW}(Fw^-) & = - \frac{p-t}{z-\bar{z}} F y^{-t} \xi_{p-t-1} \xi_{p-t}^{-1} dz \otimes w_{p-t-1, 2t} \\
\bar{\nabla}'_{\cW}(Fw^-) & = \frac{p+t}{z-\bar{z}} F y^{-t} \xi_{p-t+1} \xi_{p-t}^{-1} d\bar{z} \otimes w_{p-t+1, 2t}.
\end{align*}
\end{lemma}

\begin{proof}
We will often use the identity $2iy=z-\bar{z}$, from which it follows that 
\[
\frac{\partial y}{\partial z}=\frac{1}{2i} \quad \text{ and } \quad \frac{\partial y}{\partial \bar{z}}=-\frac{1}{2i}.
\]
Using $D=d\otimes 1=\partial\otimes 1 + \bar{\partial}\otimes 1$ on $\cW^{\infty}=\cC^{\infty}_{\cH} \otimes_\C W_\C$, we begin by computing 
\begin{align*}
(\partial\otimes 1)(\xi_{p-t} Fw^-)& =\frac{\partial F}{\partial z} y^{-p-t}(z u_1+u_2)^{p-t}(\bar{z} u_1+u_2)^{p+t}dz \otimes \mu_{2t} \\ 
& -\frac{p+t}{2i} F y^{-p-t-1}(z u_1+u_2)^{p-t}(\bar{z} u_1+u_2)^{p+t}dz \otimes \mu_{2t}  \\
& +(p-t) u_1 F y^{-p-t}(z u_1+u_2)^{p-t-1} (\bar{z} u_1+u_2)^{p+t}dz \otimes \mu_{2t}.
\end{align*}
From $(z-\bar{z}) u_1=(z u_1 + u_2)-(\bar{z} u_1+u_2)$, we obtain
\begin{align*}
(\partial\otimes 1)(\xi_{p-t} Fw^-)& =\left(\frac{\partial F}{\partial z} - \frac{2tF}{z-\bar{z}} \right) y^{-p-t}(z u_1+u_2)^{p-t}(\bar{z} u_1+u_2)^{p+t}dz \otimes \mu_{2t} \\ 
& - \frac{p-t}{(z-\bar{z})} F y^{-p-t}(z u_1+u_2)^{p-t-1}(\bar{z} u_1+u_2)^{p+t+1}dz \otimes \mu_{2t}.
\end{align*}
Hence, 
\begin{align*}
(\partial\otimes 1)(Fw^-)& =\left(\frac{\partial F}{\partial z} - \frac{2tF}{z-\bar{z}} \right) dz \otimes w^- - \frac{p-t}{z-\bar{z}} F y^{-t} \xi_{p-t-1} \xi_{p-t}^{-1} dz \otimes w_{p-t-1, 2t},
\end{align*}
and we conclude by separating terms in $\cA^{1,0}(\cW^{0,0})$ and $\cA^{1,0}(\cW^{-1,1})$.

On the other hand, we compute 
\begin{align*}
(\bar\partial\otimes 1)(\xi_{p-t} Fw^-)& =\frac{\partial F}{\partial \bar{z}} y^{-p-t}(z u_1+u_2)^{p-t}(\bar{z} u_1+u_2)^{p+t}d\bar{z} \otimes \mu_{2t} \\ 
& +\frac{p+t}{2i} F y^{-p-t-1}(z u_1+u_2)^{p-t}(\bar{z} u_1+u_2)^{p+t}d\bar{z} \otimes \mu_{2t}  \\
& +(p+t) u_1 F y^{-p-t}(z u_1+u_2)^{p-t} (\bar{z} u_1+u_2)^{p+t-1}d\bar{z} \otimes \mu_{2t}.
\end{align*}
Using the fact that $(z-\bar{z}) u_1=(z u_1 + u_2)-(\bar{z} u_1+u_2)$, we obtain
\begin{align*}
(\bar\partial\otimes 1)(\xi_{p-t} Fw^-)& =\frac{\partial F}{\partial \bar{z}} y^{-p-t}(z u_1+u_2)^{p-t}(\bar{z} u_1+u_2)^{p+t}d\bar{z} \otimes \mu_{2t} \\ 
& +\frac{p+t}{z-\bar{z}} F y^{-p-t}(z u_1+u_2)^{p-t+1} (\bar{z} u_1+u_2)^{p+t-1}d\bar{z} \otimes \mu_{2t}.
\end{align*}
Hence, 
\begin{align*}
(\bar\partial\otimes 1)(Fw^-)& =\frac{\partial F}{\partial \bar{z}} d\bar{z} \otimes w^- + \frac{p+t}{z-\bar{z}} F y^{-t} \xi_{p-t+1} \xi_{p-t}^{-1} d\bar{z} \otimes w_{p-t+1, 2t},
\end{align*}
and we conclude by separating terms in $\cA^{0,1}(\cW^{0,0})$ and $\cA^{0,1}(\cW^{1,-1})$. 
\end{proof}

\subsection{Computing Laplacians}

\begin{definition}\label{def:lapla}
The Laplacian associated to the complex $C^\infty$-bundle $\cW^\infty$ with connection $D$ is 
\[
\square_D:=DD^* + D^*D,
\] 
where 
\[
D^*:=-\bar{\ast}_{\cW^*}\circ D_{\cW^*}\circ \bar{\ast}_{\cW}
\]
denotes the formal adjoint operator with respect to the scalar product on $\cA_c^{p,q}(\cW)$
\[
(\alpha, \beta)_{L^2}:= \int_{\mathcal{H}} \alpha \wedge \bar{\ast}_{\cW} \beta= \int_{\mathcal{H}} \langle \alpha, \beta\rangle \Vol.
\]
\end{definition}

\begin{remark}
    For the claim that $D^*$ is indeed the formal adjoint of $D$, see \cite[Lem. 4.1.12]{huybrechts}.  
\end{remark}

Observe that on $\cA^{0,0}(\cW)$ we have 
\[
\square_{D}=D^*D,
\]
simply because the adjoint operator on $\cA^{0,0}(\cW)$ is trivial (as there are no forms of negative type). 

\begin{remark}
    Define $D':=\partial'_\cW+\bar{\nabla}'_\cW$ and $D'':=\bar{\partial}'_\cW+\nabla'_\cW$. 
    Letting 
    \[
    (D')^*=-\bar{\ast}_{\cW^*}\circ D'_{\cW^*}\circ \bar{\ast}_{\cW} \qquad \text{ and } \qquad (D'')^*=-\bar{\ast}_{\cW^*}\circ D''_{\cW^*}\circ \bar{\ast}_{\cW}
    \]
    denote the formal adjoints of $D'$ and $D''$ and defining Laplacians $\square_{D'}$ and $\square_{D''}$ as in Definition \ref{def:lapla}, we have \cite[Thm. 2.7]{zucker} 
    \begin{equation}\label{Kident}
    \square_D=2\square_{D'}=2\square_{D''}.
    \end{equation}
    In practice, we will compute $\square_{D'}=(D')^*D'$ on $\cA^{0,0}(\cW)$.   
\end{remark}

\begin{proposition}\label{form:lapl}
Let $w^+$ and $w^-$ be the sections defined by \eqref{eqpm}.
Given $F\in \cA^{0,0}(\cH)$, we have 
\[
\square_{D}(F\cdot w^-)=\left( -4y^2\frac{\partial^2}{\partial z\partial \bar{z}}-4ity \frac{\partial}{\partial \bar{z}} + (k-t-1)(k+t)\right) F \cdot w^-
\]
and
\[
\square_{D}(F\cdot w^+)=\left( -4y^2\frac{\partial^2}{\partial z\partial \bar{z}}+4ity \frac{\partial}{\partial \bar{z}} + (k+t-1)(k-t)\right) F \cdot w^+.
\]
\end{proposition}

\begin{proof}
During the course of this proof, we set $p=k-1$.
    We begin by computing the $(0,0)$-component of $(D')^*\partial'_{\cW}(Fw^-)$, using Lemma \ref{lem:components}:
    \begin{align*}
     (D')^* & \left(\left( \frac{\partial F}{\partial z}-\frac{2tF}{z-\bar{z}} \right) dz \otimes w^-\right) \\
    &=-\bar{\ast}_{\cW^*}\circ D'_{\cW^*}\circ \bar{\ast}_{\cW}\left(\left( \frac{\partial F}{\partial z}-\frac{2tF}{z-\bar{z}} \right) dz \otimes w^-\right) \\
    &=-\bar{\ast}_{\cW^*}\circ D'_{\cW^*}\left(i\left( \frac{\partial \bar{F}}{\partial \bar{z}}+\frac{2t\bar{F}}{z-\bar{z}} \right) d\bar{z} \otimes (w^-)^*\right) \\
    &=-\bar{\ast}_{\cW^*}\circ (\partial'_{\cW^*}+\bar{\nabla}'_{\cW^*})\left(i\left( \frac{\partial \bar{F}}{\partial \bar{z}}+\frac{2t\bar{F}}{z-\bar{z}} \right) d\bar{z} \otimes (w^-)^*\right) \\
    &=-\bar{\ast}_{\cW^*}\circ \partial'_{\cW^*}\left(i\left( \frac{\partial \bar{F}}{\partial \bar{z}}+\frac{2t\bar{F}}{z-\bar{z}} \right) d\bar{z} \otimes (w^-)^*\right).
    \end{align*}
    Using Lemma \ref{lem:leipniz} followed by Lemma \ref{lem:dual2}, we have 
    \begin{align*}
        \partial'_{\cW^*} & \left(i\left( \frac{\partial \bar{F}}{\partial \bar{z}}+\frac{2t\bar{F}}{z-\bar{z}} \right) d\bar{z} \otimes (w^-)^*\right) \\
        &=\partial \left(i\left( \frac{\partial \bar{F}}{\partial \bar{z}}+\frac{2t\bar{F}}{z-\bar{z}} \right) d\bar{z}\right) \otimes (w^-)^* - i\left( \frac{\partial \bar{F}}{\partial \bar{z}}+\frac{2t\bar{F}}{z-\bar{z}} \right) d\bar{z} \wedge (C_{\cW}^{-1}\bar\partial'_{\cW}(C_{\cW}w^-))^* \\
        & = i\left( \frac{\partial^2 \bar{F}}{\partial z\partial \bar{z}}+\frac{2t}{z-\bar{z}} \frac{\partial \bar{F}}{\partial z} - \frac{2t \bar{F}}{(z-\bar{z})^2} \right) dz\wedge d\bar{z} \otimes (w^-)^*,
    \end{align*}
    where in the last equality we used the equality $\bar\partial'_{\cW}(w^-)=0$ which follows from Lemma \ref{lem:components}.

    We deduce that 
    \begin{align*}
    (D')^*\partial'_{\cW}(Fw^-)
    &=-\bar{\ast}_{\cW^*}\circ \left( i\left( \frac{\partial^2 \bar{F}}{\partial z\partial \bar{z}}+\frac{2t}{z-\bar{z}} \frac{\partial \bar{F}}{\partial z} - \frac{2t \bar{F}}{(z-\bar{z})^2} \right) dz\wedge d\bar{z} \otimes (w^-)^* \right) \\
    &=- \left( -i\left( \frac{\partial^2 F}{\partial z\partial \bar{z}}-\frac{2t}{z-\bar{z}} \frac{\partial F}{\partial \bar{z}} - \frac{2t F}{(z-\bar{z})^2} \right) 2iy^2 w^- \right) \\
    &=i\left( \frac{\partial^2 F}{\partial z\partial \bar{z}}-\frac{2t}{z-\bar{z}} \frac{\partial F}{\partial \bar{z}} - \frac{2t F}{(z-\bar{z})^2} \right) 2iy^2 w^- \\
    &=-\left( 2y^2\frac{\partial^2 F}{\partial z\partial \bar{z}}+2ity \frac{\partial F}{\partial \bar{z}} + tF \right) w^-. 
    \end{align*}

    Next, we compute the $(0,0)$-component of $(D')^*\bar\nabla'_{\cW}(Fw^-)$, using Lemma \ref{lem:components}:
    \begin{align*}
     (D')^* & \left(\frac{p+t}{z-\bar{z}} F y^{-t} \xi_{p-t+1} \xi_{p-t}^{-1} d\bar{z} \otimes w_{p-t+1, 2t}\right) \\
    &=-\bar{\ast}_{\cW^*}\circ D'_{\cW^*}\circ \bar{\ast}_{\cW}\left(\frac{p+t}{z-\bar{z}} F y^{-t} \xi_{p-t+1} \xi_{p-t}^{-1} d\bar{z} \otimes w_{p-t+1, 2t}\right) \\
    &=-\bar{\ast}_{\cW^*}\circ D'_{\cW^*}\left(i\frac{p+t}{z-\bar{z}} \bar{F} y^{-t} \xi_{p-t+1} \xi_{p-t}^{-1} dz \otimes (w_{p-t+1, 2t})^*\right) \\
    &=-\bar{\ast}_{\cW^*}\circ (\partial'_{\cW^*}+\bar\nabla'_{\cW^*})\left(i\frac{p+t}{z-\bar{z}} \bar{F} y^{-t} \xi_{p-t+1} \xi_{p-t}^{-1} dz \otimes (w_{p-t+1, 2t})^*\right) \\
    &=-\bar{\ast}_{\cW^*}\circ \bar\nabla'_{\cW^*}\left(i\frac{p+t}{z-\bar{z}} \bar{F} y^{-t} \xi_{p-t+1} \xi_{p-t}^{-1} dz \otimes (w_{p-t+1, 2t})^*\right).
    \end{align*}
    Using Lemma \ref{lem:leipniz} followed by Lemma \ref{lem:dual2}, we have 
    \begin{align*}
     \bar\nabla'_{\cW^*} & \left(i\frac{p+t}{z-\bar{z}} \bar{F} y^{-t} \xi_{p-t+1} \xi_{p-t}^{-1} dz \otimes (w_{p-t+1, 2t})^*\right) \\
     & = - \left(i\frac{p+t}{z-\bar{z}} \bar{F} y^{-t} \xi_{p-t+1} \xi_{p-t}^{-1} dz \wedge \bar\nabla'_{\cW^*}((w_{p-t+1, 2t})^*)\right) \\
     & = - \left(i\frac{p+t}{z-\bar{z}} \bar{F} y^{-t} \xi_{p-t+1} \xi_{p-t}^{-1} dz \wedge (C_{\cW}^{-1}\nabla'_{\cW}(C_{\cW}w_{p-t+1, 2t}))^*\right) \\
     & = - \left(i\frac{p+t}{z-\bar{z}} \bar{F} y^{-t}  \xi_{p-t}^{-1} dz \wedge (C_{\cW}^{-1}\nabla'_{\cW}(-\xi_{p-t+1}w_{p-t+1, 2t}))^*\right) \\
     & = i\frac{p+t}{z-\bar{z}} \bar{F} y^{-t} \xi_{p-t}^{-1} dz \wedge (C_{\cW}^{-1}\nabla'_{\cW}(\xi_{p-t+1}w_{p-t+1, 2t}))^*.
    \end{align*}
    We have 
    \begin{align*}
        \nabla'_{\cW} & (\xi_{p-t+1}w_{p-t+1, 2t})=\nabla'_{\cW}(y^{-p}(zu_1+u_2)^{p-t+1}(\bar z u_1 + u_2)^{p+t-1} \otimes \mu_{2t}) \\ 
        & = -\frac{(p-t+1)}{z-\bar z} y^{-p}(zu_1+u_2)^{p-t}(\bar z u_1 + u_2)^{p+t} d z \otimes \mu_{2t} \\
        & = -\frac{(p-t+1)}{z-\bar z} y^t \xi_{p-t} d z \otimes w^-.
    \end{align*}
    Thus, 
    \begin{align*}
     \bar\nabla'_{\cW^*} & \left(i\frac{p+t}{z-\bar{z}} \bar{F} y^{-t} \xi_{p-t+1} \xi_{p-t}^{-1} dz \otimes (w_{p-t+1, 2t})^*\right) \\
     & =i\frac{p+t}{z-\bar{z}} \bar{F} y^{-t} \xi_{p-t}^{-1} dz \wedge \left(C_{\cW}^{-1}\left(-\frac{(p-t+1)}{z-\bar z} y^t \xi_{p-t} d z \otimes w^-\right)\right)^*\\
     & =i\frac{(p+t)(p-t+1)}{(z-\bar{z})^2} \bar{F} dz \wedge d\bar z \otimes (w^-)^*.
    \end{align*}
    We conclude that 
    \begin{align*}
        (D')^* & \bar\nabla'_{\cW}(Fw^-)=- \bar{\ast}_{\cW^*}\left(i\frac{(p+t)(p-t+1)}{(z-\bar{z})^2} \bar{F} dz \wedge d\bar z \otimes (w^-)^*\right) \\
        & =-\left(-\frac{(p+t)(p-t+1)}{2} F w^-\right) \\
        & =\frac{(p+t)(p-t+1)}{2} F w^-.
    \end{align*}
    
    Putting everything together yields
    \begin{align*}
        \square_{D'}(Fw^-) & = -\left( 2y^2\frac{\partial^2 F}{\partial z\partial \bar{z}}+2ity \frac{\partial F}{\partial \bar{z}} + tF \right) w^-
        +\frac{(p+t)(p-t+1)}{2} F w^- \\
        & = -\left( 2y^2\frac{\partial^2}{\partial z\partial \bar{z}}+2ity \frac{\partial}{\partial \bar{z}} + t - \frac{(p+t)(p-t+1)}{2} \right)Fw^- \\
        & = -\frac{1}{2}\left( 4y^2\frac{\partial^2}{\partial z\partial \bar{z}} + 4ity \frac{\partial}{\partial \bar{z}} - (p-t)(p+t+1) \right) Fw^-.
    \end{align*}
    We conclude by using \eqref{Kident} and recalling that $p=k-1$.
    A similar calculation yields the result for $w^+$ (replace $-t$ by $+t$ in all calculations). 
\end{proof}

\section{Green's kernels}\label{s:green}

Recall that $C=Y(M)_{\C}$.
In this section,
we construct the Green's kernel associated to the PVHS $\cW$ (Proposition \ref{prop:bry}). More precisely, we construct a harmonic $C^\infty$-section of the rank $4$ vector bundle $\Hom(\pr_1^* \tilde{\cW}_{0}, \pr_2^* \tilde{\cW}_{0})$ over the complement of the diagonal in $\bigsqcup\cH\times \bigsqcup\cH$ that descends to a section of the corresponding bundle over $C\times C$ minus the diagonal and admits an extension to a $C^\infty$-section of $\Hom(\pr_1^* \bar{\cW}_{0}, \pr_2^* \bar{\cW}_{0})$ over the complement of the diagonal in $\bar{C}\times \bar{C}$. We begin by analyzing the action of $\SL_2(\Z)$ on sections of $\Hom(\pr_1^* \tilde{\cW}_{0}, \pr_2^* \tilde{\cW}_{0})$.

\subsection{$\SL_2(\Z)$-invariant sections}

Let $G$ be a section of $\Hom(\pr_1^* \tilde\cW_{0}, \pr_2^* \tilde\cW_{0})$ over $\bigsqcup(\cH \times \cH)\setminus \Delta_{\cH}$, where we use $\Delta_{\cH}$ to denote the diagonal of $\cH$. At a point $(z, z')\in (\cH \times \cH)\setminus \Delta_{\cH}$, $G(z,z')$ is an element in the fiber $\Hom(\pr_1^* \tilde\cW_{0}, \pr_2^* \tilde\cW_{0})_{(z,z')}=\Hom((\tilde{\cW}_{0})_{z}, (\tilde{\cW}_{0})_{z'})$.
The bundle $\tilde\cW_{0}$ has rank $2$ with basis $\{ w^+, w^-\}$ defined by \eqref{eqpm}, so all we need to do to specify a section $G$ is to define $G(z,z')(w^-(z))$ and $G(z,z')(w^+(z))$ in terms of $w^-(z')$ and $w^+(z')$.

Throughout this section, we set $p=k-1$. 
Let $\gamma=\left(\begin{smallmatrix} a&b\\c&d \end{smallmatrix} \right)\in \SL_2(\Z)$. We have $w^\pm(z) \in \tilde{\cW}_z= \tilde{W}=\sym^{2p}(V)(p)\otimes \R^2$ and thus $\gamma$ acts (on the left) on $w^\pm(z)$ via the symmetric standard representation of $G=\GL_2(\R)^+$ as follows:
\begin{equation}\label{action}
\gamma\cdot w^{\pm}(z):= \xi_{p\pm t}^{-1} y^{-p\pm t} ((az+b)u_1+(cz+d)u_2)^{p\pm t}((a\bar{z}+b)u_1+(c\bar{z}+d)u_2)^{p\mp t} \otimes \mu_{t\mp t}.
\end{equation}
In particular, we see that 
\begin{equation}\label{actionw}
\gamma\cdot w^{\pm}(z)=\det(\gamma)^{p \mp t}j(\gamma, z)^{\pm 2t} w^\pm(\gamma z).
\end{equation}
The left action of $\SL_2(\Z)\times \SL_2(\Z)$ on a section $G$ of $\Hom(\pr_1^* \tilde{\cW}_{0}, \pr_2^* \tilde{\cW}_{0})$ is defined by 
\begin{equation}\label{def:actionG2}
((\gamma, \gamma') \cdot G)(z, z')(w^{\pm}(z)):=\gamma' \cdot G(\gamma^{-1} z, (\gamma')^{-1} z')(\gamma^{-1} \cdot w^{\pm}(z)).
\end{equation}
In particular, the left diagonal $\SL_2(\Z)$-action is given by 
\begin{equation}\label{def:actionG}
(\gamma\cdot G)(z, z')(w^{\pm}(z)):=\gamma\cdot G(\gamma^{-1}z, \gamma^{-1}z')(\gamma^{-1}\cdot w^{\pm}(z)).
\end{equation}
We compute this diagonal action on the section $G_0$ defined by $G_0(z,z')(w^{\pm}(z))=w^{\pm}(z')$:
\[
(\gamma\cdot G_0)(z,z')(w^{\pm}(z))=j(\gamma^{-1}, z)^{\pm 2t}j(\gamma, \gamma^{-1} z')^{\pm 2t}w^{\pm}(z').
\]
With these formulas in hand, it is not difficult to create diagonal $\SL_2(\Z)$-invariant sections, as we will now explain.

Consider the function $g=g_{k,t}$ on $(\cH\times \cH)\setminus \Delta_{\cH}$ given by
\begin{equation}\label{def:g}
g(z,z'):=-Q_{k,t}\left(1+\frac{\vert z-z'\vert^2}{2yy'}\right),
\end{equation}
where $Q_{k,t}(x)$ is the Jacobi function of the second kind \eqref{Qkt1} which we recall is defined, for all $x$ in the complex plane cut along the segment $[-1,1]$, by
\begin{equation}\label{Qkt}
Q_{k,t}(x):=\int_{-\infty}^\infty \frac{2^{2t}dw}{(x+\sqrt{x^2-1}\cosh(w))^{k-t}(x+1+\sqrt{x^2-1}e^w)^{2t}}.
\end{equation}
Note that $g$ is a function on $(\cH\times \cH)\setminus \Delta_{\cH}$ that only depends on the hyperbolic distance between $z$ and $z'$. It follows easily that $g$ is invariant under the diagonal action of $\SL_2(\Z)$, i.e., $g(\gamma z, \gamma z')=g(z,z')$ for all $\gamma \in \SL_2(\Z)$. Define the following functions on $(\cH\times \cH)\setminus \Delta_{\cH}$:
\[
\mu_g^{+}(z,z'):=g(z,z')\left( \frac{z- \bar{z}'}{2iy'} \right)^{2t} \quad \text{ and } \quad \mu_g^{-}(z,z'):=g(z,z')\left( \frac{\bar{z}- z'}{2iy} \right)^{2t}. 
\]
A quick calculation reveals that 
\[
\mu_g^{\pm}(\gamma z,\gamma z')=j(\gamma,z)^{\mp 2t}j(\gamma,z')^{\pm 2t}\mu_g^{\pm}(z,z').
\]

\begin{lemma}
The section $G_g$ of $\Hom(\pr_1^* \tilde{\cW}_{0}, \pr_2^* \tilde{\cW}_{0})$ defined by 
\[
G_g(z,z')(w^{\pm}(z))=\mu_g^\pm(z,z') w^{\pm}(z'),
\]
is invariant under the left diagonal action \eqref{def:actionG} of $\SL_2(\Z)$. 
\end{lemma}

\begin{proof}
An easy verification using the fact that for any $z'\in \cH$ we have $j(\gamma^{-1}, z')j(\gamma, \gamma^{-1}z')=1$.
\end{proof}

\subsection{The Green's kernel}\label{s:Gamma1}

Now that we have an $\SL_2(\Z)$-invariant section $G_g$, we will use it to define a section $G_{g,M}$ of the bundle $\Hom(\pr_1^* \cW_{0}, \pr_2^* \cW_{0})$ on $(C\times C)\setminus \Delta_C$, 
where $\cW$ denotes the bundle \eqref{bundleW} on $C$ and $\Delta_C$ is the diagonal in $C\times C$.

By definition, we have a Cartesian diagram 
\begin{equation*}
\begin{tikzcd}
\tilde{\cW} \arrow{r}{\pr} \arrow{d} 
 & \cW \arrow{d} \\
\bigsqcup\cH \arrow{r}[swap]{\pr} 
 & C(\C).
\end{tikzcd}
\end{equation*}
If $z\in \cH$ and $Q=\pr(z)=\Gamma(M)z \in C(\C)$, then the fibers $\tilde{\cW}_{z}$ and $\cW_{Q}$ are canonically identified since the diagram is Cartesian. Suppose that $z'\in \cH$ is another point such that $\pr(z')=Q$. Then there exists $\gamma\in \Gamma(M)$ such that $\gamma z'=z$. The action of $\gamma$ on fibers then identifies $\gamma\cdot \tilde{\cW}_{z'}$ with $\tilde{\cW}_{z}$. 
 
In order to define a section $G_{g,M}$ of $\Hom(\pr_1^* \cW_{0}, \pr_2^* \cW_{0})$, it suffices to specify $G_{g,M}(Q, Q')\in \Hom(\cW_{Q}, \cW_{Q'})$ for all $Q\neq Q'$ in $C(\C)$. For this, it is in turn enough to specify $G_{g,M}(z,z')\in \Hom((\tilde{\cW}_0)_{z}, (\tilde{\cW}_0)_{z'})$ that are independent of choices of representatives $z$ and $z'$ for $Q=\Gamma(M)z$ and $Q'=\Gamma(M)z'$. This last condition is equivalent to requiring that $(\gamma, \gamma')\cdot G_{g,M} = G_{g,M}$ for all $\gamma, \gamma'\in \Gamma(M)$, where the action is the one defined in \eqref{def:actionG2}. 

\begin{lemma}\label{def:Green}
For all $0<t\leq k-1$, the sum 
\begin{equation*}
G=G_{g,M}:=\sum_{\gamma\in \Gamma(M)} (1, \gamma^{-1})\cdot G_g
\end{equation*}
converges uniformly on compact subsets of $\mathcal{H}^2\setminus \{ (z,z') \mid z\in \Gamma(M) z' \}$ and gives a well-defined section of the bundle $\Hom(\pr_1^* \cW_{0}, \pr_2^* \cW_{0})$ on $(C\times C)\setminus \Delta_C$. 
\end{lemma}

\begin{remark}\label{def:G2}
Assuming the convergence of the sum (to be proved below), we explicitly have, for $Q=\Gamma(M)z\neq \Gamma(M)z'=Q'$,
\begin{equation}\label{eqG-}
G_{g,M}(Q,Q')(w^{\pm}(z))
=\sum_{\gamma\in \Gamma(M)}  j(\gamma, z')^{\mp 2t} \mu^{\pm}_g(z,\gamma z')w^{\pm}(z'). 
\end{equation}
Using the diagonal $\SL_2(\Z)$-invariance of $G_g$, we obtain
\[
G_{g,M}=\sum_{\gamma\in \Gamma(M)} (1, \gamma^{-1})\cdot G_g=\sum_{\gamma\in \Gamma(M)} (1, \gamma^{-1})\cdot (\gamma \cdot G_g)=\sum_{\gamma\in \Gamma(M)} (\gamma, 1) \cdot G_g,
\]
which leads to the useful formula
\begin{equation}\label{exp}
G_{g,M}(Q, Q')(w^{\pm}(z))
=\sum_{\gamma\in \Gamma(M)} j(\gamma, z)^{\pm 2t} \mu^{\pm}_g(\gamma z, z')w^{\pm}(z'). 
\end{equation}
\end{remark}

\begin{proof}[Proof of Lemma \ref{def:Green}]
Assuming the convergence of the sum, the $\Gamma(M)$-invariance is an easy verification using the diagonal $\SL_2(\Z)$-invariance of $G_g$. We focus on the proof of convergence, inspired by the proof of \cite[VI Prop. 6.2]{hejhal}. Let $Q=\Gamma(M)z\neq \Gamma(M)z'=Q'$.
In view of Remark \ref{def:G2}, it suffices to check that the sums 
\begin{equation}\label{Gpm}
\mathcal{G}^+(z,z'):=\sum_{\gamma\in \Gamma(M)}  j(\gamma, z')^{- 2t} \mu^{+}_g(z,\gamma z') \quad \text{ and } \quad \mathcal{G}^-(z,z'):=\sum_{\gamma\in \Gamma(M)} j(\gamma, z)^{- 2t} \mu^{-}_g(\gamma z, z')
\end{equation}
converge  uniformly for $(z,z')\in E_1 \times E_2$, where $E_1$ and $E_2$ are compact subsets of $\mathcal{H}$. We have 
\begin{multline*}
\mathcal{G}^+(z,z')=-\sum_{\gamma\in \Gamma(M)} \left( \frac{(z- \gamma\bar{z}')}{j(\gamma, z') 2i\Im(\gamma z')} \right)^{2t}Q_{k,t}\left(1+\frac{\vert z-\gamma z'\vert^2}{2y \Im(\gamma z')}\right) \\
=-\sum_{\gamma\in \Gamma(M)} \left( \frac{j(\gamma, \bar{z}')(z- \gamma\bar{z}')}{ 2iy'} \right)^{2t}Q_{k,t}\left(1+\frac{\vert z-\gamma z'\vert^2}{2y \Im(\gamma z')}\right).
\end{multline*}
Using Lemma \ref{lem:asyinfty2}, this is 
\begin{align*}
O & \left( \sum_{\gamma\in \Gamma(M)} \left( \frac{\vert j(\gamma, z')\vert \vert z- \gamma\bar{z}'\vert}{y'} \right)^{2t} \left(\frac{2y\Im(\gamma z')}{\vert z-\gamma z'\vert^2}\right)^{k+t} \right) \\
&=O\left( \sum_{\gamma\in \Gamma(M)} \frac{\vert j(\gamma, z')\vert^{2t} \vert z- \gamma\bar{z}'\vert^{2t}}{(y')^t} \left(\frac{2y\Im(\gamma z')}{\vert z-\gamma z'\vert^2}\right)^{k+t} \right) \\
&=O\left( \sum_{\gamma\in \Gamma(M)} \left(\frac{\vert z- \gamma\bar{z}'\vert^{2}}{\Im(\gamma z')}\right)^t \left(\frac{2y\Im(\gamma z')}{\vert z-\gamma z'\vert^2}\right)^{k+t} \right) \\
&=O\left( \sum_{\gamma\in \Gamma(M)} \left(\frac{\Im(\gamma z')}{\vert z-\gamma z'\vert^2}\right)^{k} \right) \\
&=O\left( \sum_{\gamma\in \Gamma(M)} \left(\frac{\Im(\gamma z')}{( 1+\vert\gamma z'\vert)^2}\right)^{k} \right),
\end{align*}
where the implicit constants depend on $E_1$ and $E_2$. Similarly, it can be shown that
\begin{equation}\label{blup}
\mathcal{G}^-(z,z')=O\left( \sum_{\gamma\in \Gamma(M)} \left(\frac{\Im(\gamma z)}{( 1+\vert\gamma z\vert)^2}\right)^{k} \right).
\end{equation}
Since $k\geq 2$, we may apply \cite[VI Prop. 5.1]{hejhal}, which says that for each $\xi\in \mathcal{H}$, we have 
\[
\sum_{\gamma\in \Gamma(M)} \left(\frac{\Im(\gamma \xi)}{( 1+\vert\gamma \xi\vert)^2}\right)^{k} \leq C(k, M),
\]
where $C(k,M)$ is some positive constant that only depends on $k$ and $M$. The uniform convergence follows as in the proof of \cite[VI Prop. 6.2]{hejhal}.
\end{proof}

We may use the basis $\{ w^+, w^- \}$ to trivialize $\cW_0$, but this does not extend to a $C^\infty$-trivialization of $\bar{\cW}_0$, because $w^{\pm}$ have singularities of type $y^{k\mp t-1}$ near the cusp $\infty$ (hence singularities of type $\Im(\gamma z)^{{k\mp t-1}}$ near the cusp $\gamma\infty$).
Recall from \eqref{Gpm} the functions $\mathcal{G}^{\pm}$ on $\cH\times \cH\setminus \Delta_{\cH}$. Using the trivialization $\{ w^+, w^- \}$, the section $G$ is given by the matrix 
\[
\left(\begin{matrix}
    \mathcal{G}^+ &  0 \\
     0 &  \mathcal{G}^-
\end{matrix}
\right).
\]

The remaining part of this section is dedicated to proving that $G$ defined in Lemma \ref{def:Green} is the Green's kernel associated to the bundle $\cW$ on $C$, i.e., that it satisfies the following properties (see Proposition \ref{prop:bry}): 
\begin{itemize}
\item[(i)] $G(Q,Q')-\log \vert z(Q)-z(Q')\vert$ is bounded near any point $(Q,Q)\in \Delta_{\bar{C}}$, if $z$ is a local coordinate on $\bar{C}$ near $Q$;
\item[(ii)] $G$ is harmonic with respect to the Laplacian $\square_D$ attached to $\cW$ in the second variable;
\item[(iii)] As $z'$ approaches the cusp $\infty$, the functions $\langle \Im(\gamma z)^{{k\mp t-1}} G(z,\gamma z')(w^{\pm}(z)), w^{\pm}(\gamma z') \rangle$ are $C^{\infty}$ for $z'$ in a neighborhood of $\infty$, for all $\gamma\in \Gamma(M)$.
\end{itemize}

\begin{proposition}
    $G(Q,Q')-\log \vert z(Q)-z(Q')\vert$ is bounded near any point $(Q,Q)\in \Delta_{\bar{C}}$, if $z$ is a local coordinate on $\bar{C}$ near $Q$.
\end{proposition}

\begin{proof}
After trivializing the bundle using $w^+$ and $w^-$, it suffices to show that 
\[
\mu^{\pm}(z,z')-\log \vert z-z'\vert^2
\]
is bounded as $z'\to z$ in $\mathcal{H}$. This follows from Lemma \ref{lem:asy} and the fact that $\lim\limits_{z'\to z} (\bar{z}-z')=-2iy$.
\end{proof}

\begin{proposition}\label{prop:harmonic}
The section $G$ is harmonic with respect to the Laplacian associated to $\cW$ in the second variable.
\end{proposition}

\begin{proof}
Given $(z,z')\in (\cH\times \cH)\setminus \Delta_{\cH}$, we define 
\[
s(z,z'):=1+\frac{\vert z-z' \vert^2}{2yy'}.
\]
Then $g(z,z')=-Q_{k,t}(s(z,z'))$ and by \eqref{exp}, for $Q=\Gamma(M)z\neq\Gamma(M)z'=Q'$, we have
\begin{equation}\label{eq:G2}
G(Q,Q')(w^\pm(z))=
\sum_{\gamma\in \Gamma(M)} j(\gamma, z)^{\pm 2t} \mu^{\pm}_g(\gamma z, z')w^{\pm}(z').
\end{equation}
We write $\square_2$ for the Laplacian $\square_D$ associated to $\cW$ in the second variable. Then
\[
\square_2(G(Q,Q')(w^\pm(z)))=\sum_{\gamma\in \Gamma(M)} j(\gamma, z)^{\pm 2t} \square_2(\mu_g^{\pm}(\gamma z, z')w^{\pm}(z')).
\]
In order to prove the proposition, it therefore suffices to prove that 
\[
\square_2(\mu_g^{\pm}(z, z')w^\pm(z'))=0
\] 
for any $z\in \cH$, $z\neq z'$.

Define 
\begin{equation}\label{diffope}
\Delta:=-4(y')^2 \frac{\partial^2}{\partial z' \partial \bar{z}'}, \quad \Delta_{\pm t}:=\Delta \pm 4ity' \frac{\partial}{\partial \bar{z}'} \quad \text{ and } \quad \lambda_{\pm t}=-(k\pm t-1)(k\mp t).
\end{equation}
By Proposition \ref{form:lapl}, it suffices to prove that 
\begin{equation}\label{harm:mu}
\Delta_{\pm t}(\mu_g^{\pm}(z, z'))=\lambda_{\pm t}\mu_g^{\pm}(z, z').
\end{equation}

In what follows, we will write $s$ for $s(z,z')$ to ease the notation. 
The following identities are quickly verified: 
\begin{equation}\label{der:s}
\frac{\partial s}{\partial \bar{z}'}=\frac{z'-z}{2yy'}+\frac{s-1}{2iy'}=\frac{s-1}{2iy'}\frac{z'-\bar{z}}{\bar{z}'-\bar{z}}, \qquad \qquad \frac{\partial s}{\partial z'}=\frac{\bar{z}'-\bar{z}}{2yy'}-\frac{s-1}{2iy'}=\frac{s-1}{2iy'}\frac{z-\bar{z}'}{z'-z}, 
\end{equation}
\begin{equation}\label{der:s2}
4(y')^2\frac{\partial^2 s}{\partial z'\partial \bar{z}'}=2s, \qquad \qquad 4(y')^2\left( \frac{\partial s}{\partial \bar{z}'}\right)\left( \frac{\partial s}{\partial z'}\right)=s^2-1.
\end{equation}
We now compute:  
\begin{align*}
\Delta_{-t}(\mu^-_g(z, & z'))
 =\left( -4(y')^2 \frac{\partial^2}{\partial z' \partial \bar{z}'}-4ity' \frac{\partial}{\partial \bar{z}'}\right)\left( -Q_{k,t}(s)\left( \frac{\bar{z}-z'}{2iy}\right)^{2t} \right)\\
&=\left(4(y')^2 \frac{\partial}{\partial z'}+4ity'\right)\left( \left( \frac{\partial s}{\partial \bar{z}'} \right) Q'_{k,t}(s)\left( \frac{\bar{z}-z'}{2iy}\right)^{2t}\right) \\
&= 4(y')^2 \frac{\partial^2 s}{\partial z' \partial \bar{z}'}Q'_{k,t}(s)\left( \frac{\bar{z}-z'}{2iy}\right)^{2t}+\left( \frac{\partial s}{\partial \bar{z}'} \right) \left(4(y')^2 \frac{\partial}{\partial z'}+4ity'\right)\left(Q'_{k,t}(s)\left( \frac{\bar{z}-z'}{2iy}\right)^{2t}\right) \\
&=2sQ'_{k,t}(s)\left( \frac{\bar{z}-z'}{2iy}\right)^{2t}+4(y')^2 \left( \frac{\partial s}{\partial \bar{z}'} \right) \left( \frac{\partial s}{\partial z'} \right) Q''_{k,t}(s) \left( \frac{\bar{z}-z'}{2iy}\right)^{2t} \\
& \qquad  + 4(y')^2 \left( \frac{\partial s}{\partial \bar{z}'} \right)Q'_{k,t}(s) \frac{\partial}{\partial z'}  \left(  \left( \frac{\bar{z}-z'}{2iy}\right)^{2t} \right)+4ity'\left( \frac{\partial s}{\partial \bar{z}'} \right) Q'_{k,t}(s)\left( \frac{\bar{z}-z'}{2iy}\right)^{2t} \\
&=2sQ'_{k,t}(s)\left( \frac{\bar{z}-z'}{2iy}\right)^{2t}+(s^2-1)Q''_{k,t}(s) \left( \frac{\bar{z}-z'}{2iy}\right)^{2t} \\
& \qquad  + 4(y')^2 \left( \frac{\partial s}{\partial \bar{z}'} \right)Q'_{k,t}(s) \left(  -\frac{2t}{2iy} \left( \frac{\bar{z}-z'}{2iy}\right)^{2t-1} \right)+4ity'\left( \frac{\partial s}{\partial \bar{z}'} \right) Q'_{k,t}(s)\left( \frac{\bar{z}-z'}{2iy}\right)^{2t} \\
&=2sQ'_{k,t}(s)\left( \frac{\bar{z}-z'}{2iy}\right)^{2t}+(s^2-1)Q''_{k,t}(s) \left( \frac{\bar{z}-z'}{2iy}\right)^{2t} \\
& \qquad  - \frac{2t}{2iy}\cdot 4(y')^2 \left( \frac{\partial s}{\partial \bar{z}'} \right)Q'_{k,t}(s) \frac{2iy}{\bar{z}-z'} \left( \frac{\bar{z}-z'}{2iy}\right)^{2t}+4ity'\left( \frac{\partial s}{\partial \bar{z}'} \right) Q'_{k,t}(s)\left( \frac{\bar{z}-z'}{2iy}\right)^{2t} \\
&=2sQ'_{k,t}(s)\left( \frac{\bar{z}-z'}{2iy}\right)^{2t}+(s^2-1)Q''_{k,t}(s) \left( \frac{\bar{z}-z'}{2iy}\right)^{2t} \\
& \qquad  + 4ity'\left( \frac{\partial s}{\partial \bar{z}'} \right) Q'_{k,t}(s)\left( \frac{\bar{z}-z'}{2iy}\right)^{2t}\left( 1 + \frac{2iy'}{\bar{z}-z'} \right) \\
&=\left[(s^2-1)Q''_{k,t}(s) + \left( 2s + 4ity'\left( \frac{\partial s}{\partial \bar{z}'} \right) \left( 1 + \frac{2iy'}{\bar{z}-z'} \right) \right) Q'_{k,t}(s) \right]\left( \frac{\bar{z}-z'}{2iy}\right)^{2t}.
\end{align*}
Using \eqref{der:s}, a quick calculation reveals that
\[
2s + 4ity'\left( \frac{\partial s}{\partial \bar{z}'} \right) \left( 1 + \frac{2iy'}{\bar{z}-z'} \right)= (2t+2)s -2t.
\]
By plugging this into our previous calculation, we obtain 
\[
\Delta_{-t}(\mu^-_g(z,z'))= -\left[(1-s^2)Q''_{k,t}(s) + (2t-(2t+2)s) Q'_{k,t}(s) \right]\left( \frac{\bar{z}-z'}{2iy}\right)^{2t}. 
\]
By Corollary \ref{coro:difeq}, we have 
\[
(1-s^2)Q''_{k,t}(s) + (2t-(2t+2)s) Q'_{k,t}(s)=\lambda_{-t} Q_{k,t}(s).
\]
We deduce that 
\[
\Delta_{-t}(\mu^-_g(z,z'))=-\lambda_{-t} Q_{k,t}(s)\left( \frac{\bar{z}-z'}{2iy}\right)^{2t}=\lambda_{-t}\mu_g^-(z,z').
\]

The equality
$\Delta_{+t}(\mu^+_g(z,z'))=\lambda_{+t} \mu_g^+(z,z')$ can be verified similarly and is left to the reader. 
\end{proof}

In order to prove property $(iii)$ above, it suffices to prove:

\begin{proposition}
As $z'$ approaches the cusp $\infty$, the functions 
\[
\Im(\gamma z')^{{k\mp t-1}}\langle w^{\pm}(\gamma z'), w^{\pm}(\gamma z') \rangle\mathcal{G}^{\pm}(z,\gamma z')
\]
are $C^{\infty}$ for $z'$ in a neighborhood of $\infty$, for all $\gamma \in \Gamma(M)$.
\end{proposition}

\begin{proof}
    We proceed as in \cite[VI \S 6]{hejhal}. We retain the notations established in \eqref{diffope}. As in \eqref{harm:mu}, we have
    \begin{equation}\label{diffg-}
    \Delta_{\pm t}(\mathcal{G}^{\pm}(z, z'))=\lambda_{\pm t}\mathcal{G}^{\pm}(z, z').
    \end{equation}

    Begin by observing using \eqref{eqG-} that for all $\gamma \in \Gamma(M)$, we have 
    \[
    \mathcal{G}^-(z, \gamma z')=j(\gamma, z')^{-2t} \mathcal{G}^-(z,z').
    \]
    In particular, letting $A_M:=\left( \begin{smallmatrix}
        1 & M \\ 0 & 1
    \end{smallmatrix}\right)\in \Gamma(M)$, we see that 
    \[
    \mathcal{G}^-(z, z'+M)=\mathcal{G}^-(z, A_M z')=j(A_M, z')^{-2t} \mathcal{G}^-(z,z')=\mathcal{G}^-(z,z').
    \]
    Thus, when viewed as a function in the single variable $z'=x'+iy'$, $\mathcal{G^-}$ is $M$-periodic and admits a Fourier expansion. More precisely, for $z\in E\subset \mathcal{H}$ with $E$ compact and $y'\geq M_E+1$ with $M_E>0$ as in \cite[Prop. 6.6]{hejhal}, we have 
    \[
    \mathcal{G}^-(z,z')=\sum_{n\in \Z} c^-_n(z,y') e^{2\pi i \frac{n}{M} x'}.
    \]
    As a function of $y'$, the $n^{\text{th}}$ Fourier coefficient 
    \begin{equation}\label{fouriercoe}
        c_n^-(z,y')=\frac{1}{M}\int_0^{M} \mathcal{G}^-(z,x'+iy')e^{-2\pi i \frac{n}{M} x'}dx'
    \end{equation}
    satisfies the differential equation
    \begin{equation}\label{ODE}
    u''(y')-\frac{2t}{y'}u'(y')-\bigg[ \left( \frac{2\pi n}{M} \right)^2 + \frac{4\pi n t}{y'M} + \frac{(k-t-1)(k+t)}{(y')^2}\bigg]u(y')=0.
    \end{equation}
    This is readily verified by differentiating \eqref{fouriercoe} under the sign of the integral, using the equality of differential operators
    \[
    \frac{\partial^2}{\partial (y')^2}-\frac{2t}{y'}\frac{\partial}{\partial y'}=-\frac{1}{(y')^2}\Delta_{-t} -\left( \frac{\partial^2}{\partial (x')^2}+\frac{2it}{y'}\frac{\partial}{\partial x'} \right)
    \]
    coupled with the equality \eqref{diffg-},
    and by integrating by parts. We identify the ordinary differential equation \eqref{ODE} as being a confluent differential equation. In the notations of \cite[6.2 (13)]{HTFII}, \eqref{ODE} corresponds to the case 
    \begin{align*}
        a&=0 \\
        b&=-2t \\
        \alpha&=-\left( \frac{2\pi n}{M} \right)^2 \\
        \beta&=-\frac{4\pi n t}{M} \\
        \gamma&=\lambda_{-t}=-(k-t-1)(k+t).
    \end{align*}
    
    For $n\neq 0$, we note that $a^2\neq 4\alpha$, and thus solutions of \eqref{ODE} are given by \cite[6.2 (14)]{HTFII} 
    \[
    (y')^{t}w\left(-t, k-\frac{1}{2}; \frac{4\pi n}{M}y'\right),
    \]
    where $w(\kappa, \mu; x)$ is any solution of the Whittaker differential equation \cite[6.1 (4)]{HTFII}
    \begin{equation}\label{whittaker}
    z''(x)+\left( -\frac{1}{4}+\frac{\kappa}{x}+\frac{1/4-\mu^2}{x^2} \right) z(x)=0.
    \end{equation}

    Note that when $t=0$, \eqref{ODE} reduces to the differential equation
    \begin{equation}\label{ODEt0}
    u''(y')-\bigg[ \left( \frac{2\pi n}{M} \right)^2 + \frac{k(k-1)}{(y')^2}\bigg]u(y')=0.
    \end{equation}
    This is compatible with \cite[Prop. 4.10 (ii)]{hejhal}, in which solutions are given by modified Bessel functions. Indeed, the Whittaker differential equation \eqref{whittaker} collapses in the case $\kappa=-t=0$ to 
    \[
    z''(x)+\left( -\frac{1}{4}+\frac{1/4-\mu^2}{x^2} \right) z(x)=0,
    \] 
   which can be transformed into a modified Bessel differential equation. 
    
Upon noting that $2\mu+1=2k\in \Z$, two linearly independent solutions of equation \eqref{whittaker} are given by the Whittaker functions 
    \[
    M_{\kappa, \mu}(x):=e^{-x/2} x^{\mu+1/2} \Phi(\mu-\kappa+1/2, 2\mu+1; x),
    \]
    \[
    W_{\kappa, \mu}(x):=e^{-x/2} x^{\mu+1/2} \Psi(\mu-\kappa+1/2, 2\mu+1; x).
    \]
    The functions $\Phi$ and $\Psi$ are the confluent hypergeometric functions of the first and second kind, respectively: $\Phi$ is Kummer's function defined by
    \[
    \Phi(a,c;x):=\sum_{n\geq 0} \frac{(a)_n}{(c_n)} \frac{x^n}{n!},
    \]
   and $\Psi$ is Tricomi's function defined by the integral formula \cite[6.5 (3)]{HTFII}. 
    
    We deduce that for $n\neq 0$, there exist functions $a_n$ and $b_n$ such that
    \[
    c_{n}^-(z, y')=a_n(z) (y')^t M_{-t, k-1/2}\left(\frac{4\pi n}{M}y'\right)+b_n(z)(y')^t W_{-t, k-1/2}\left(\frac{4\pi n}{M}y'\right).
    \]
    Explicitly, we have
    \[
    M_{-t, k-1/2}\left(\frac{4\pi n}{M}y'\right)=e^{-\frac{2\pi n}{M}y'} \left( \frac{4\pi n}{M}y'\right)^{k} \Phi\left(k+t, 2k; \frac{4\pi n}{M}y'\right),
    \]
    \[
    W_{-t, k-1/2}\left(\frac{4\pi n}{M}y'\right)=e^{-\frac{2\pi n}{M}y'} \left( \frac{4\pi n}{M}y'\right)^{k} \Psi\left(k+t, 2k; \frac{4\pi n}{M}y'\right).
    \]
    
    When $n=0$, 
    $c_0^-(z, y')$ satisfies the differential equation 
     \begin{equation}\label{ODE0}
   u''(y')-\frac{2t}{y'}u'(y')-\frac{(k-t-1)(k+t)}{(y')^2} u(y')=0.
    \end{equation}
    Thus, there exist functions $a_0$ and $b_0$ such that
    \[
    c_{0}^-(z, y')=a_0(z)(y')^{k+t} +b_0(z)(y')^{t-k+1}.
    \]

    By the same arguments as in \cite[Prop. 4.12]{hejhal}, we see that $a_n(z)=0$ for all $n\neq 0$ because the confluent hypergeometric function of the first kind $\Phi$ becomes exponentially large at $y'=\infty$, as is clear from the integral representation \cite[6.5 (1)]{HTFII}. (Note that the integration condition in \cite[Prop. 4.12]{hejhal} holds in our case by \eqref{blup}.) 
    Therefore, as in \cite[(6.5)]{hejhal}, we see that 
    \[
    \mathcal{G}^-(z,z')=a_0(z)(y')^{k+t} +b_0(z)(y')^{t-k+1}+\sum_{n\neq 0} b_n(z)(y')^t W_{-t, k-1/2}\left(\frac{4\pi n}{M}y'\right)e^{2\pi i \frac{n}{M} x'}.
    \]
    It follows that $(y')^{{k-t-1}}\mathcal{G}^-(z,z')$ is $C^{\infty}$ for $z'$ in a neighborhood of the cusp $\infty$, and the result follows for $\gamma=I_2$ upon noting that $(y')^{{k+t-1}}\langle w^{-}(z'), w^{-}(z') \rangle=(y')^{{k-t-1}}$. The case of the other cusps can be deduced similarly, and is left to the reader. The proof for $\mathcal{G}^+$ is obtained by replacing all occurrences of $-t$ in the above proof by $+t$.
    \end{proof}

\section{Brylinski pairings}\label{s:archii}

We continue to use the notation $C=Y(M)_{\C}$.
Having proved that the section $G=G_{k,t,M}$ of Lemma \ref{def:Green} is the Green's kernel associated to the vector bundle $\cW$ \eqref{PVHS} on $C$, we are now ready to compute archimedean Brylinski pairings \eqref{eq: brylinski formula} between what Brylinski calls ``Hodge classes'' of $\cW$. These are sums of classes of type $(0,0)$ in the fibers $W_{z}$, for $z \in C$. 

We will be interested in cycles $Z_1, Z_2 \in \CH^{k+t}(X_\C)_K$ which are homologically trivial and supported in fibers above $z_1, z_2 \in C$. Their local height pairing depends only on the image of these classes in $W_{z_i} \otimes \C$ by (\ref{eq: brylinski local height}). We therefore work with $\C$ coefficients, and compute the pairing for the basic sections $w^\pm$.

\subsection{Brylinski's pairing on $C$}

\begin{theorem}\label{thm:br}
Let $z_1=x_1+iy_1 \in \cH$ and $z_2=x_2+iy_2 \in \cH$ with $z_1 \not\in \Gamma(M)z_2$. Then the Brylinski pairing is given by 
\[
\langle w^-(z_1), w^-(z_2) \rangle^{\mathrm{Br}}=\frac{1}{(2iy_1y_2)^{2t}} \sum_{\gamma\in \Gamma(M)} g(z_1, \gamma z_2)(\bar{z}_1-\gamma z_2)^{2t} j(\gamma, z_2)^{2t},
\]
\[
\langle w^+(z_1), w^+(z_2) \rangle^{\mathrm{Br}}=\frac{1}{(2i)^{2t}} \sum_{\gamma\in \Gamma(M)} g(z_1, \gamma z_2)(z_1-\gamma \bar{z}_2)^{2t} j(\gamma, \bar{z}_2)^{2t},
\]
where 
\[
g(z,z')=-Q_{k,t}\left( 1+\frac{\vert z-z' \vert^2}{2yy'}\right).
\]
Moreover, we have
\[
\langle w^-(z_1), w^+(z_2) \rangle^{\mathrm{Br}}=\langle w^+(z_1), w^-(z_2) \rangle^{\mathrm{Br}}=0.
\]
\end{theorem}

\begin{proof}
The Green's kernel associated to the bundle $\cW$ on $C$ is given by $G=G_{k,t,M}$ of Lemma \ref{def:Green}. Let $Q_i=\pr(z_i)=\Gamma(M)z_i\in C(\C)$ for $i=1,2$. By assumption, we have $Q_1\neq Q_2$. By definition of Brylinski's pairing \eqref{eq: brylinski formula} in this case, if $w_1\in \cW_{Q_1}$ and $w_2\in \cW_{Q_2}$, then 
\[
\langle w_1, w_2 \rangle^{\mathrm{Br}}:= \langle G(Q_1, Q_2)(w_1), w_2 \rangle_{Q_2},
\]
where $\langle \:, \: \rangle_{Q_2}$ is the Hermitian pairing (associated to the polarization) of $\cW$ restricted to the fiber at $Q_2$.

We then compute
\begin{align*}
\langle w^-(z_1), w^-(z_2) \rangle^{\mathrm{Br}} 
& =\langle G(Q_1, Q_2)(w^-(z_1)), w^-(z_2) \rangle_{Q_2} \\
& =\left\langle \sum_{\gamma\in \Gamma(M)} j(\gamma, z_2)^{2t}\mu^-_g(z_1, \gamma z_2) w^-(z_2), w^-(z_2) \right\rangle_{Q_2} \\
& =\sum_{\gamma\in \Gamma(M)} g(z_1, \gamma z_2)\left( \frac{\bar{z}_1- \gamma z_2}{2iy_1} \right)^{2t} j(\gamma, z_2)^{2t} \langle w^-(z_2), w^-(z_2) \rangle_{Q_2} \\
& =\frac{1}{(2iy_1)^{2t}}\sum_{\gamma\in \Gamma(M)} g(z_1, \gamma z_2)(\bar{z}_1- \gamma z_2)^{2t} j(\gamma, z_2)^{2t}\langle w^-(z_2), w^-(z_2) \rangle_{z_2} \\
& =\frac{1}{(2iy_1y_2)^{2t}}\sum_{\gamma\in \Gamma(M)} g(z_1, \gamma z_2)(\bar{z}_1- \gamma z_2)^{2t} j(\gamma, z_2)^{2t},
\end{align*}
where in the fifth equality we used the formula $\langle w^-(z_2), w^-(z_2) \rangle_{z_2}=y_2^{-2t}$.
The calculation for $w^+$ can be done similarly, using the equality $\langle w^+(z_2), w^+(z_2) \rangle_{z_2}=y_2^{2t}$. The equalities 
\[
\langle w^-(z_1), w^+(z_2) \rangle^{\mathrm{Br}}=\langle w^+(z_1), w^-(z_2) \rangle^{\mathrm{Br}}=0
\]
are deduced in a similar way from the fact that $w_{k-t-1,2t}(z)$ and $w_{k+t-1,0}(z)$ are orthogonal with respect to the Hermitian pairing $\langle \:, \: \rangle_z$.
\end{proof}

\subsection{Working over $X_0(N)$}

So far we have worked with varieties fibered over $\bar{C}=X(M)$, as well as vector bundles over $\bar{C}$. Ultimately, we wish to compute the heights of generalized Heegner cycles ``over $X_0(N)$''.

Denote by $\pi_{M,N} \colon X(M)\lra X_0(N)$ the natural quotient map with Galois group 
\[
S=S_{M,N}:=\pm \Gamma(M)\backslash \Gamma_0(N).
\] 
Recall from \eqref{def:Za} and \eqref{def:Zbara} that we defined generalised Heegner cycles above the Heegner point $P_{\fn}=(A, A[\fn])\in X_0(N)(H)$ as the sum of the generalized Heegner cycles over the pre-image CM-points $Q\in X(M)(\C)$ of $P_{\fn}$ with respect to $\pi_{M,N}$ (similar to what is done in \cite[\S 4.1]{zhang}).

Given $P=\Gamma_0(N)z\in X_0(N)(\C)$, define 
\begin{equation}\label{eq:eBw}
w^{\pm}(P):=\sum_{s\in S} s\cdot w^{\pm}(z)\in \sum_{Q\in \pi_{M,N}^{-1}(P)} \cW_Q=:\cW_P,
\end{equation}
the action being the one defined in \eqref{action}.

\begin{remark}\label{rem:actionS}
We say a few words about the action of $S$ on elements $w^{\pm}(\Gamma(M)z)$ in fibers of the bundle $\mathcal{W}$. Let $[\gamma]=s\in S=\pm\Gamma(M)\backslash \Gamma_0(N)$ with $\gamma\in \Gamma_0(N)$. Then $s\cdot \Gamma(M)z:=\Gamma(M)\gamma z$. Consider now the variety $\pi \colon X\lra \bar{C}$ fibered over $\bar{C}$. The fiber $\pi^{-1}(\Gamma(M)z)=(\C/\langle 1, z\rangle)^{2k-2}\times A^{2t}$ consists of points $x=(\Gamma(M) z, w_1,\ldots, w_{2k-2}, P_1, \ldots, P_{2t})$ with $w_i\in \C/\langle 1, z\rangle$ and $P_j\in A$. The action of $[\gamma]=s\in S$ is given by 
\[ 
s\cdot x:=(\Gamma(M) \gamma z, w_1,\ldots, w_{2k-2}, P_1, \ldots, P_{2t})\in \pi^{-1}(\Gamma(M) \gamma z).
\]
Fixing a point $\Gamma_0(N)z\in X_0(N)(\C)$, $S$ acts simply transitively on the fibers $\pi^{-1}(\Gamma(M)z_i)$ with $\Gamma_0(N)z_i=\Gamma_0(N)z$. This action on $X$ gives rise to a correspondence on $X$, hence an action on Chow groups and cohomology groups. There is also an induced action on the bundle $\mathcal{W}$, which we now address. Recall from the beginning of \S \ref{s:Gamma1} that the fibers $\tilde{\cW_{z}}$ and $\tilde{\cW}_{\alpha z}$ are identified via the action \eqref{action} of $\gamma$ on fibers for all $\alpha\in \Gamma(M)$. In other words, we have identifications of fibers 
\begin{equation}\label{id:fibers}
\cW_{\Gamma(M)z}=\tilde{\cW}_{z}\overset{\alpha \cdot}{\simeq} \tilde{\cW}_{\alpha z}, \qquad \text{ for all } \alpha \in \Gamma(M).
\end{equation}
Now, the action of $[\gamma]=s\in S=\pm\Gamma(M)\backslash \Gamma_0(N)$ on a section $w$ of $\cW$ is given by 
\[
s\cdot w(\Gamma(M)z)=\gamma \cdot w(z).
\]
The latter action is via \eqref{action} and we must interpret $\gamma \cdot w(z)$ as lying in $\tilde{\cW}_{\gamma z}\simeq \cW_{\Gamma(M) \gamma z}=\cW_{s\cdot \Gamma(M)z}$. This is well-defined because if $\alpha\in \Gamma(M)$, then 
\[
(\alpha\gamma)\cdot w(z)=\alpha\cdot (\gamma\cdot w(z))\in \tilde{\cW}_{\alpha\gamma z},
\]
and under the identification \eqref{id:fibers} the latter element corresponds to $\gamma\cdot w(z)\in \cW_{s\cdot \Gamma(M) z}$.
\end{remark}

\begin{theorem}\label{thm:br2}
Let $z_1=x_1+iy_1 \in \cH$ and $z_2=x_2+iy_2 \in \cH$ with $P_1:=\Gamma_0(N)z_1\neq \Gamma_0(N)z_2=: P_2$. Then the Brylinski pairing is given by 
\[
\langle w^-(P_1), w^-(P_2) \rangle^{\mathrm{Br}}=\frac{\vert S\vert}{(2iy_1y_2)^{2t}} \sum_{\gamma\in \Gamma_0(N)\slash \pm 1} g(z_1, \gamma z_2)(\bar{z}_1-\gamma z_2)^{2t} j(\gamma, z_2)^{2t},
\]
\[
\langle w^+(P_1), w^+(P_2) \rangle^{\mathrm{Br}}=\frac{\vert S\vert}{(2i)^{2t}} \sum_{\gamma\in \Gamma_0(N)\slash \pm 1} g(z_1, \gamma z_2)(z_1-\gamma \bar{z}_2)^{2t} j(\gamma, \bar{z}_2)^{2t},
\]
where 
\[
g(z,z')=-Q_{k,t}\left( 1+\frac{\vert z-z' \vert^2}{2yy'}\right).
\]
Moreover, we have
\[
\langle w^-(P_1), w^+(P_2) \rangle^{\mathrm{Br}}=\langle w^+(P_1), w^-(P_2) \rangle^{\mathrm{Br}}=0.
\]
\end{theorem}

\begin{proof}
By definition of Brylinski's pairing and \eqref{eq:eBw}, we have 
\begin{align*}
\langle w^\pm(P_1), w^\pm(P_2) \rangle^{\mathrm{Br}}&=\left\langle \sum_{s\in S} s\cdot w^{\pm}(z_1), \sum_{s'\in S} s'\cdot w^{\pm}(z_2) \right\rangle^{\mathrm{Br}} \\
&= \sum_{s\in S}\left\langle s\cdot w^{\pm}(z_1), \sum_{s'\in S} ss^{-1}s'\cdot w^{\pm}(z_2) \right\rangle^{\mathrm{Br}} \\
&=\sum_{s\in S}\left\langle w^{\pm}(z_1), \sum_{s'\in S} s^{-1}s' \cdot w^{\pm}(z_2) \right\rangle^{\mathrm{Br}} \\
&= \vert S\vert \sum_{s\in S} \left\langle w^{\pm}(z_1), s\cdot w^{\pm}(z_2) \right\rangle^{\mathrm{Br}},
\end{align*}
where we used the invariance $\langle \sigma \cdot w, \sigma \cdot w'\rangle^{\mathrm{Br}}=\langle w, w'\rangle^{\mathrm{Br}}$ for any automorphism $\sigma$ of $\bar{C}$ by functoriality.
If we choose a lift $\gamma'\in \Gamma_0(N)$ of $s\in S=\pm\Gamma(M)\backslash \Gamma_0(N)$, then $\Gamma(M)sz_2=\Gamma(M) \gamma'z_2$.
Moreover, we have $\Gamma(M)z_1\neq \Gamma(M) \gamma' z_2$ since otherwise $\Gamma_0(N)z_1= \Gamma_0(N)\gamma' z_2=\Gamma_0(N)z_2$, which contradicts our assumptions. 
We may thus apply Theorem \ref{thm:br} and use \eqref{actionw} along with the fact that the Brylinski pairing is Hermitian to obtain
\begin{align*}
\langle w^-(P_1), w^-(P_2) \rangle^{\mathrm{Br}}&=
\vert S\vert \sum_{\gamma' \in \pm\Gamma(M)\backslash \Gamma_0(N)} \left\langle w^{-}(z_1), \gamma' \cdot w^{-}(z_2) \right\rangle^{\mathrm{Br}}\\
&=\vert S\vert \sum_{\gamma'\in \pm\Gamma(M)\backslash \Gamma_0(N)}\frac{j(\gamma', z_2)^{2t}}{(2iy_1y_2)^{2t}} \sum_{\gamma\in \Gamma(M)} g(z_1, \gamma \gamma' z_2)(\bar{z}_1-\gamma \gamma' z_2)^{2t} j(\gamma, \gamma' z_2)^{2t}.
\end{align*}
By noting that $j(\gamma', z_2)j(\gamma, \gamma' z_2)=j(\gamma \gamma', z_2)$, it follows that
\begin{align*}
\langle w^-(P_1), w^-(P_2) \rangle^{\mathrm{Br}}&=
\frac{\vert S\vert}{(2iy_1y_2)^{2t}}  \sum_{\gamma'\in \pm\Gamma(M)\backslash \Gamma_0(N)}\sum_{\gamma\in \Gamma(M)} g(z_1, \gamma \gamma' z_2)(\bar{z}_1-\gamma \gamma' z_2)^{2t} j(\gamma \gamma', z_2)^{2t} \\
&=\frac{\vert S\vert}{(2iy_1y_2)^{2t}} \sum_{\gamma\in \Gamma_0(N)\slash \pm 1} g(z_1, \gamma z_2)(\bar{z}_1-\gamma z_2)^{2t} j(\gamma, z_2)^{2t}.
\end{align*}
The proof for $w^+$ is similar.
\end{proof}

\subsection{Action of Hecke operators}

As the vector bundles $\tilde{\cW}$ and $\cW$ are equipped with a $\GL_2(\R)^+$-action, the Hecke operators $T_m$ act on sections such as $w^+$ and $w^-$. 
Let 
\[ 
R_N:=\left( \begin{smallmatrix} \Z & \Z \\ N\Z & \Z \end{smallmatrix} \right) 
\] 
and consider, for $m\in \mathbb{N}$,
\[
R_{N}^m:=\Gamma_0(N)\left( \begin{smallmatrix} 1 & 0 \\ 0 & m \end{smallmatrix} \right)\Gamma_0(N)=\{ \gamma \in R_N \mid \det(\gamma)=m \}\subset \GL_2(\Q)^+.
\]
The double coset $R_{N}^m$ naturally carries a left action of $\Gamma_0(N)$ by left multiplication. 

Let $z\in\cH$ with $P:=\Gamma_0(N)z$. Given $m\geq 1$, we set 
\begin{equation}\label{heckecoh}
T_m w^{\pm}(P):= \sum_{\gamma \in \Gamma_0(N) \backslash R_{N}^m} \gamma \cdot w^{\pm}(P),
\end{equation}
where the action is analogous to the one defined in \eqref{action}, i.e.,
\[
\gamma\cdot w^{\pm}(P):=\det(\gamma)^{p \mp t}j(\gamma, z)^{\pm 2t} w^\pm(\Gamma_0(N)\gamma z),
\] 
where a representative of $\gamma$ in $R^m_N$ has implicitly been chosen.  

\begin{proposition}\label{Tm0}
Let $z_1, z_2 \in \cH$ with $P_i=\Gamma_0(N)z_i,$ $i=1,2$, and let $m\geq 1$ with $(m,N)=1$. Assume that $P_1$ and $T_m P_2$ are disjoint in $X_0(N)$. Then 
\[
\langle w^-(P_1), T_m w^-(P_2) \rangle^{\mathrm{Br}}=\frac{\vert S\vert m^{p-t}}{(2iy_1 y_2)^{2t}} \sum_{\gamma\in R^m_{N}/\pm 1} g(z_1, \gamma z_2)(\bar{z}_1-\gamma z_2)^{2t} j(\gamma, z_2)^{2t},
\]
\[
\langle w^+(P_1), T_m w^+(P_2) \rangle^{\mathrm{Br}}=\frac{\vert S\vert m^{p-t}}{(2i)^{2t}} \sum_{\gamma \in R^m_{N}/\pm 1} g(z_1, \gamma z_2)(z_1-\gamma \bar{z}_2)^{2t} j(\gamma, \bar{z}_2)^{2t}.
\]
\end{proposition}

\begin{proof}
    We compute 
\begin{align*}
\langle & w^-(P_1), T_m w^-(P_2) \rangle^{\mathrm{Br}}
=\sum_{\gamma \in  \Gamma_0(N) \backslash R_{N}^m} m^{p+t} j(\gamma, \bar{z}_2)^{-2t} \langle w^-(\Gamma_0(N) z_1), w^-(\Gamma_0(N)\gamma z_2) \rangle^{\mathrm{Br}} \\
&=\sum_{\gamma \in \Gamma_0(N) \backslash R_{N}^m} m^{p+t} j(\gamma, \bar{z}_2)^{-2t}\frac{\vert S\vert}{(2iy_1\Im(\gamma z_2))^{2t}} \sum_{\gamma'\in \Gamma_0(N)\slash \pm 1} g(z_1, \gamma' \gamma z_2)(\bar{z}_1-\gamma' \gamma z_2)^{2t} j(\gamma', \gamma z_2)^{2t} \\
&=\frac{\vert S\vert m^{p-t}}{(2iy_1 y_2)^{2t}} \sum_{\gamma'\in \Gamma_0(N)\slash \pm 1}\sum_{\gamma \in \Gamma_0(N) \backslash R^m_{N}}  j(\gamma, z_2)^{2t} g(z_1, \gamma' \gamma z_2)(\bar{z}_1-\gamma' \gamma z_2)^{2t} j(\gamma', \gamma z_2)^{2t} \\
&=\frac{\vert S\vert m^{p-t}}{(2iy_1 y_2)^{2t}} \sum_{\gamma'\in \Gamma_0(N)\slash \pm 1}\sum_{\gamma \in \Gamma_0(N) \backslash R_{N}^m} g(z_1, \gamma' \gamma z_2)(\bar{z}_1-\gamma' \gamma z_2)^{2t} j(\gamma' \gamma, z_2)^{2t} \\
&=\frac{\vert S\vert m^{p-t}}{(2iy_1 y_2)^{2t}} \sum_{\gamma\in R^m_{N}\slash \pm 1} g(z_1, \gamma z_2)(\bar{z}_1-\gamma z_2)^{2t} j(\gamma, z_2)^{2t},
\end{align*}
where we used Theorem \ref{thm:br2} in the second equality and the identity $j(\gamma, z)j(\gamma', \gamma z)=j(\gamma'\gamma, z)$ in the fourth equality. The calculation for $w^+$ is similar.
\end{proof}

\subsection{Brylinski's pairing over Heegner points on $X_0(N)$}\label{s:HP}

We now specialize to the case where $P_1=\Gamma_0(N)\tau_1$ and $P_2=\Gamma_0(N)\tau_2$ are Heegner points on $X_0(N)$ of conductor $1$. In terms of moduli, this means that these points represent isomorphism classes of cyclic $N$-isogenies between elliptic curves with CM by the same order in an imaginary quadratic field (in our case the imaginary quadratic field is $K$ and the order is the maximal order $\oh_K$). Concretely, following \cite[\S 2]{GZ}, Heegner points with CM by $K$ are indexed by $(\cA, \mathfrak{n})$, where $\mathfrak{n}$ is a cyclic $N$-ideal of $\oh_K$ (whose existence is guaranteed by the Heegner hypothesis) and $\cA\in \mathrm{Cl}_K$. Over $\C$, the corresponding point is $(\C/\mathfrak{a} \lra \C/\mathfrak{a}\mathfrak{n}^{-1})$ for any $[\mathfrak{a}]=\cA$. If $W_{\mathrm{AL}}$ denotes the group of Atkin--Lehner involutions of $X_0(N)$, then $\Gal(H/K)\times W_{\mathrm{AL}}$ acts simply transitively on the set of all Heegner points of discriminant $D$, i.e., on all pairs $(\cA, \mathfrak{n})$. If $\sigma\in \Gal(H/K)$ corresponds to $\mathcal{B}\in \Cl_K$ under the Artin map of class field theory, then $\sigma(\cA, \mathfrak{n})=(\cA \mathcal{B}^{-1}, \mathfrak{n})$. 

Let $\beta\in \Z/2N\Z$ with $\beta^2\equiv D \pmod{4N}$ such that $\mathfrak{n}=\langle N, \frac{\beta+\sqrt{D}}{2}\rangle$. Let $\tau=\frac{-B+\sqrt{D}}{2A}\in \cH$\footnote{Beware of the following conflict of notation: $A$ denotes our fixed CM elliptic curve, $B$ a full level $M$-structure, and $C$ the modular curve $Y(M)$. We trust that this slight abuse of notation will not cause any confusion.} be the solution to a quadratic equation of the form
\[
AX^2+BX+C=0, \quad A>0, \quad B^2-4AC=D, \quad N\mid A, \quad B\equiv \beta \pmod{2N}
\]
and consider 
\[ 
\mathfrak{a}=\left\langle A, \frac{B+\sqrt{D}}{2}\right\rangle, \qquad \mathfrak{a}\mathfrak{n}^{-1}=\left\langle \frac{A}{N}, \frac{B+\sqrt{D}}{2}\right\rangle, \qquad N_{K/\Q}(\mathfrak{a})=A, \qquad [\fa]=\cA.
\] 
Then the Heegner point associated to $(\cA, \mathfrak{n})$ is represented by $-\bar{\tau}=(B+\sqrt{D})/2A\in \cH$.
Note that 
\[
\mathfrak{a}^{-1}=\frac{\bar{\mathfrak{a}}}{N_{K/\Q}(\mathfrak{a})}=\left\langle 1, \frac{-B+\sqrt{D}}{2A} \right\rangle=\langle 1, \tau\rangle,
\]
whence the Heegner point associated to $(\cA^{-1}, \bar{\mathfrak{n}})$ is represented by $\tau\in \cH$.\footnote{See the Appendix for a discussion of how our conventions for Heegner points differ slightly from the ones in \cite{GZ}.} 

Now, for $i=1,2$, $P_i=\Gamma_0(N)\tau_i$ is a Heegner point, hence $\langle 1, \tau_i\rangle=\fa_i^{-1}$ for some ideal $\fa_i \subset \oh_K$. Let $A_i>0$ and $B_i$ be the corresponding integers as above such that $\tau_i=\frac{-B_i+\sqrt{D}}{2A_i}=\tau_{\cA_i, \mathfrak{n}_i}$ and $N_{K/\Q}(\fa_i)=A_i$. 

\begin{proposition}\label{yes1}
Let $P_1=\Gamma_0(N)\tau_1$ and $P_2=\Gamma_0(N)\tau_2$ be two Heegner points of $X_0(N)$ and suppose that $P_1$ and $T_m P_2$ are disjoint in $X_0(N)$. With notations as above, we have 
\[
\left\langle \frac{w^-(P_1)}{\chi(\fa_1)}, T_m \frac{w^-(P_2)}{\chi(\fa_2)} \right\rangle^{\mathrm{Br}}=\frac{(-1)^t \vert S\vert 2^{2t}m^{p-t}}{\vert D\vert^{2t}}\chi(\bar{\fa}_1 \fa_2) \sum_{\gamma\in R_N^m/\pm 1} g(\tau_1, \gamma \tau_2) \alpha(\gamma, \tau_1, \tau_2)^{2t},
\]
\[
\left\langle \frac{w^+(P_1)}{y_1^{2t}\chi(\bar{\fa}_1)}, T_m \frac{w^+(P_2)}{y_2^{2t}\chi(\bar{\fa}_2)} \right\rangle^{\mathrm{Br}}=\frac{(-1)^t \vert S\vert 2^{2t}m^{p-t}}{\vert D\vert^{2t}}\chi(\fa_1 \bar{\fa}_2) \sum_{\gamma\in R_N^m/\pm 1} g(\tau_1, \gamma \tau_2) \alpha(\gamma, \bar\tau_1, \bar\tau_2)^{2t},
\]
where $\alpha(\gamma, \tau_1, \tau_2)=c\bar{\tau}_1\tau_2+d\bar{\tau}_1-a\tau_2-b$ for $\gamma=\left( \begin{smallmatrix}
    a & b \\ c & d \end{smallmatrix} \right)$.
\end{proposition}

\begin{proof}
Observe that $\chi(\fa_i\bar{\fa}_i)=\chi(N_{K/\Q}(\fa_i))=\chi(A_i)=A_i^{2t}$. In particular, $\chi(\fa_i)^{-1}=\chi(\bar{\fa}_i)/A_i^{2t}$. Using the fact that the Brylinski pairing is skew-symmetric, we obtain
\[
\left\langle \frac{w^-(P_1)}{\chi(\fa_1)}, T_m \frac{w^-(P_2)}{\chi(\fa_2)} \right\rangle^{\mathrm{Br}}=\frac{\chi(\bar{\fa}_1 \fa_2)}{(A_1A_2)^{2t}}\left\langle w^-(P_1), T_m w^-(P_2)\right\rangle^{\mathrm{Br}}.
\]
Observe that $y_i=\sqrt{\vert D\vert}/2A_i$ and apply Proposition \ref{Tm0} to obtain
\[
\left\langle \frac{w^-(P_1)}{\chi(\fa_1)}, T_m \frac{w^-(P_2)}{\chi(\fa_2)} \right\rangle^{\mathrm{Br}}=
\frac{(-1)^t \vert S\vert 2^{2t}m^{p-t}}{\vert D\vert^{2t}} \chi(\bar{\fa}_1 \fa_2) \sum_{\gamma\in R^m_{N}/\pm 1} g(\tau_1, \gamma \tau_2)(\bar{\tau}_1-\gamma \tau_2)^{2t} j(\gamma, \tau_2)^{2t}.
\]
As $(\bar{\tau}_1-\gamma \tau_2) j(\gamma, \tau_2)=(\bar{\tau}_1-\gamma \tau_2) (c\tau_2+d)=c\bar{\tau_1}\tau_2 + d\bar{\tau}_1-a\tau_2-b=\alpha(\gamma, \tau_1, \tau_2)$, the first equality follows.
The calculation for $w^+$ is similar.
\end{proof}

\section{Cycle classes}\label{s:hodge}

The goal of this section is to compute the cycle classes of generalized Heegner cycles viewed as cycles in CM fibers of $X\lra \bar{C}$. While the cycles are null-homologous in $X$ by Lemma \ref{lem: null homologous}, this is not necessarily true when viewed as cycles in the fibers. 

\subsection{Preliminary cycle class calculations}

For an elliptic curve $E/\C$, we will view $\End(E)$ as a subring of $\C$ via $\End(E) \hookrightarrow \End(H^0(E, \Omega_E)) \simeq \C$. In this subsection and the next, we view $E$ as a complex manifold and work exclusively with de Rham cohomology of complex manifolds.

\begin{proposition}\label{prop:cl}
Suppose $E=\C/\langle 1, z\rangle$ for some $z = x + iy\in \cH$ with complex coordinate $w$. For endomorphisms $\alpha, \beta\in \End(E)$, define $\Gamma_{\alpha, \beta} = (\alpha \times \beta)_*([E]) \in \CH^1(E \times E)$
and 
\[
X_{\alpha,\beta}:=\Gamma_{\alpha,\beta}-\deg(\beta) \Gamma_{0,1} - \deg(\alpha)\Gamma_{1,0} \in \CH^1(E\times E).
\]
Then 
\[
\cl(X_{\alpha,\beta})=\frac{i}{2y}(\alpha\bar\beta d\bar{w}_1 \otimes dw_2-\bar\alpha\beta dw_1 \otimes d\bar{w}_2)\in H_{\dR}^1(E)\otimes H_{\dR}^1(E)
\]
is the orthogonal projection of $\cl(\Gamma_{\alpha, \beta})$ to $H_{\dR}^1(E)\otimes H_{\dR}^1(E)$.
\end{proposition}

\begin{proof}
If $Z\subset E\times E$ is a divisor, then its de Rham cycle class $\omega_Z\in H^2_{\dR}(E\times E)$ is characterized by the property
\[
\int_{Z} \eta = \int_{E\times E} \eta\wedge \omega_Z, \qquad \text{ for all } \eta\in  H^2_{\dR}(E\times E).
\]
By K\"unneth, a basis of $H^2_{\dR}(E\times E)$ is given by 
\[
\begin{array}{ll}
b_{1\bar{1}} :=\pr_1^*(dw\wedge d\bar w), & b_{12}:=\pr_1^*(dw)\wedge \pr_2^*(dw), \\
b_{1\bar{2}} :=\pr_1^*(dw)\wedge \pr_2^*(d\bar w), & b_{\bar{1}2}:=\pr_1^*(d\bar w)\wedge \pr_2^*(dw), \\
b_{\bar{1}\bar{2}} :=\pr_1^*(d\bar{w})\wedge \pr_2^*(d\bar w), & b_{2\bar{2}} :=\pr_2^*(dw\wedge d\bar w).
 \end{array}
 \]
 Let 
 \[
 \eta=\lambda_{1\bar{1}} b_{1\bar{1}}+\lambda_{12}b_{12}+\lambda_{1\bar{2}}b_{1\bar{2}}+\lambda_{\bar{1}2}b_{\bar{1}2}+\lambda_{\bar{1}\bar{2}}b_{\bar{1}\bar{2}}+\lambda_{2\bar{2}} b_{2\bar{2}} \in H^2_{\dR}(E\times E).
 \]
Begin by observing that 
\begin{equation}\label{cl:hv1}
\cl(0\times E)=\frac{i}{2y} b_{1\bar{1}} \quad \text{ and } \quad \cl(E\times 0)=\frac{i}{2y} b_{2\bar{2}}.
\end{equation}
Indeed, on one hand we have
\[
\int_{0\times E} \eta=\int_{0\times E} \lambda_{2\bar{2}} b_{2\bar{2}}= \lambda_{2\bar{2}} \int_E dw\wedge d\bar{w}=-2iy \lambda_{2\bar{2}}, 
\]
while on the other hand
\[
\int_{E\times E} \eta\wedge \frac{i}{2y} b_{1\bar{1}}=\frac{\lambda_{2\bar{2}} i}{2y} \int_{E\times E} b_{1\bar{1}} b_{2\bar{2}}=\frac{\lambda_{2\bar{2}} i}{2y} \left(\int_E dw\wedge d\bar{w} \right)^2=\frac{-\lambda_{2\bar{2}} i 4y^2}{2y}=-2 i y \lambda_{2\bar{2}}.
\]
The case of $E\times 0$ is verified similarly.

On the one hand, we have 
\begin{align*}
\int_{\Gamma_{\alpha, \beta}} \eta & = \int_{E} \lambda_{1\bar{1}} \alpha^*dw\wedge \alpha^*d\bar{w}+\lambda_{1\bar{2}} \alpha^*dw\wedge \beta^*d\bar{w}+\lambda_{\bar{1} 2} \alpha^*d\bar{w}\wedge \beta^*dw+\lambda_{2\bar{2}} \beta^*dw\wedge \beta^*d\bar{w} \\
& = -2iy(\lambda_{1\bar{1}} \alpha\bar{\alpha}+\lambda_{1\bar{2}} \alpha\bar{\beta}-\lambda_{\bar{1}2} \bar\alpha\beta+ \lambda_{2\bar{2}} \beta\bar{\beta}). 
\end{align*}
On the other hand, we have 
\begin{align*}
\int_{E\times E} \eta\wedge & \frac{i}{2y}(\beta\bar\beta b_{1\bar{1}}+\alpha\bar\beta b_{\bar{1}2}-\bar\alpha\beta b_{1\bar{2}}+\alpha\bar\alpha b_{2\bar{2}}) \\
& =\frac{i}{2y} \int_{E\times E} (\lambda_{2\bar{2}}\beta\bar\beta+\lambda_{1\bar{2}}\alpha\bar\beta-\lambda_{\bar{1}2}\bar\alpha\beta+\lambda_{1\bar{1}}\alpha\bar\alpha)b_{1\bar{1}}\wedge b_{2\bar{2}} \\
& = -2iy(\lambda_{2\bar{2}}\beta\bar\beta+\lambda_{1\bar{2}}\alpha\bar\beta-\lambda_{\bar{1}2}\bar\alpha\beta+\lambda_{1\bar{1}}\alpha\bar\alpha).
\end{align*}
We conclude that 
\[
\cl(\Gamma_{\alpha, \beta})=\frac{i}{2y}(\beta\bar\beta b_{1\bar{1}}+\alpha\bar\beta b_{\bar{1}2}-\bar\alpha\beta b_{1\bar{2}}+\alpha\bar\alpha b_{2\bar{2}}). 
\]
Using the fact that $\alpha\bar\alpha=\deg(\alpha)$ and $\beta\bar\beta=\deg(\beta)$ combined with \eqref{cl:hv1}, we obtain
\[
\cl(\Gamma_{\alpha, \beta})=\deg(\beta)\cl(0\times E)+\deg(\alpha)\cl(E\times 0)+\frac{i}{2y}(\alpha\bar\beta b_{\bar{1}2}-\bar\alpha\beta b_{1\bar{2}}). 
\]
\end{proof}

\subsection{Cycle classes of generalized Heegner cycles}\label{s:hhodg}

For the rest of the section, we view the varieties $C$, $X$, $A$, etc.\ as being complex manifolds. 

\begin{definition}
Given $\Gamma(M)z \in 
C(\C)$, let
\[
\cl_z \colon \CH^{k+t-1}(E_z^{2k-2}\times A^{2t}) \lra \tilde{\cW}_z=\cW_{\Gamma(M)z}
\]
denote the $\epsilon_W \otimes \kappa_{2t}$-component of the cycle class map on the fiber $E_z^{2k-2}\times A^{2t}$ of $\pi \colon X\lra \bar{C}$ above $\Gamma(M)z$. If $P\in X_0(N)(\C)$, let
\[
\cl_P := \sum_{\Gamma(M)z\in \pi_{M,N}^{-1}(P)} \cl_z \colon \sum_{\Gamma(M)z\in \pi_{M,N}^{-1}(P)} \CH^{k+t-1}(E_z^{2k-2}\times A^{2t}) \lra \sum_{\Gamma(M)z\in \pi_{M,N}^{-1}(P)} \cW_{\Gamma(M)z}=:\cW_{P}.
\]
\end{definition}

\begin{proposition}\label{prop:cycleclass}
Let $\fa$ an ideal of $\oh_K$ and let $\tau_\fa=x_{\fa}+iy_\fa\in \cH$ such that $A^{\fa}=\C / \langle 1, \tau_{\fa}\rangle$.
Then
\[
\cl_{\tau_{\fa}}(\epsilon Y^\mathfrak{a})=\frac{(-1)^p 2^{p-t} \sqrt{\vert D \vert}^{p-t}}{\binom{2p}{p-t}^{1/2}} y_0^t w^-(\tau_{\mathfrak{a}})\in \tilde{\cW}_{\tau_{\fa}}
\]
and
\[
\cl_{\tau_{\fa}}(\bar\epsilon Y^\mathfrak{a})=\frac{(-1)^p 2^{p-t} \sqrt{\vert D \vert}^{p-t}}{\binom{2p}{p-t}^{1/2}} \frac{y_0^t}{y_{\fa}^{2t}} w^+(\tau_{\mathfrak{a}})\in \tilde{\cW}_{\tau_{\fa}}.
\]
\end{proposition}

\begin{proof}
We recall that 
\[
Y^\mathfrak{a}=(\Gamma_{[\sqrt{D}]})^{p-t}\times (\Gamma_{\phi_\mathfrak{a}}^T)^{2t}\subset (A^{\mathfrak{a}}\times A^{\mathfrak{a}})^{p-t}\times (A^{\mathfrak{a}}\times A)^{2t}.
\]
Replacing $\Gamma_{[\sqrt{D}]}$ by $X_{1,[\sqrt{D}]}$ in this definition has no effect on $\epsilon Y^{\mathfrak{a}}$ and $\bar\epsilon Y^{\mathfrak{a}}$ (the effect of Scholl's projector $\epsilon_W$ is precisely to project the cycle classes (in fibers) to $\sym^{2p} H_{\dR}^{1}(A^\mathfrak{a})\subset H^1_{\dR}(A^\fa)^{\otimes 2p}$). By Proposition \ref{prop:cl} applied with $(\alpha, \beta)=(1, [\sqrt{D}])$, we have 
\[
\cl(X_{1,[\sqrt{D}]})=-\frac{\sqrt{\vert D\vert}}{2y_\mathfrak{a}}(dw_1 \otimes d\bar{w}_2+d\bar{w}_1 \otimes dw_2)\in H^1_{\dR}(A^\mathfrak{a})\otimes H^1_{\dR}(A^\mathfrak{a}).
\]
Projecting to the symmetric power, it follows that 
\[
\cl(\epsilon_1X_{1,\sqrt{D}})=-\frac{\sqrt{\vert D\vert}}{y_\mathfrak{a}}dw_1 d\bar{w}_2\in \sym ^2 H^1_{\dR}(A^\mathfrak{a}).
\]
Observe that $\Gamma_{\phi_\mathfrak{a}}^T=\Gamma_{\phi_{\mathfrak{a}, 1}}=\{ (\phi_{\mathfrak{a}}(x), x) \mid x\in A \} \subset A^\mathfrak{a} \times A$. After applying the relevant projectors, only the orthogonal projection of the cycle class to $H^1_{\dR}(A^\mathfrak{a})\otimes H_{\dR}^1(A)$ will contribute. We therefore calculate the cycle class of 
\[
X_{\phi_\mathfrak{a},1}:=\Gamma_{\phi_\mathfrak{a},1}-0\times A- \frac{y_0}{y_\mathfrak{a}}\cdot A^\mathfrak{a}\times 0 \in \CH^1(A^\mathfrak{a}\times A) \; \footnote{Recall that $y_0=\frac{\sqrt{\vert D\vert }}{2}$ and $y_{\fa}=\frac{\sqrt{\vert D\vert }}{2A}$, whence $\frac{y_0}{y_{\fa}}=A$. We write $\frac{y_0}{y_{\fa}}$ to avoid the clash of notations pointed out in Footnote 7.}
\]
by the same method as in the proof of Proposition \ref{prop:cl} to obtain
\[
\cl(X_{\phi_{\mathfrak{a}}, 1})=\frac{i}{2y_\mathfrak{a}}(d\bar w_{\mathfrak{a}}\otimes d w - d w_{\mathfrak{a}}\otimes d\bar w)\in H_{\dR}^1(A^\mathfrak{a})\otimes H_{\dR}^1(A),
\]
where $w_{\fa}$ is the complex coordinate on $A^{\fa}$ and $w$ is the one on $A$. 
The effect of the projector $\epsilon$ is to kill $d\bar w$. Hence, we obtain
\[
\cl(\epsilon X_{\phi_{\mathfrak{a}}, 1})=\frac{i}{2y_\mathfrak{a}}d\bar w_{\mathfrak{a}}\otimes dw \quad \text{ and } \quad \cl(\bar\epsilon X_{\phi_{\mathfrak{a}}, 1})=-\frac{i}{2y_\mathfrak{a}}d w_{\mathfrak{a}}\otimes d\bar w.
\]
Recall from Section \ref{s:brylinski} that $dw_{\mathfrak{a}}$ corresponds to $\tau_{\mathfrak{a}}u_1+u_2$ and $dw$ corresponds to $\tau_0 e_1+e_2$. 
Putting everything together, we conclude that 
\[
\cl_{\tau_{\fa}}(\epsilon Y^\mathfrak{a})=\left(-\frac{\sqrt{\vert D\vert}}{y_\mathfrak{a}}\right)^{p-t}\left( \frac{i}{2y_\mathfrak{a}}\right)^{2t}(\tau_{\mathfrak{a}}u_1+u_2)^{p-t}(\bar{\tau}_{\mathfrak{a}}u_1+u_2)^{p+t}\otimes (\tau_0 e_1+e_0)^{2t}
\]
and
\[
\cl_{\tau_{\fa}}(\bar\epsilon Y^\mathfrak{a})=\left(-\frac{\sqrt{\vert D\vert}}{y_\mathfrak{a}}\right)^{p-t}\left( -\frac{i}{2y_\mathfrak{a}}\right)^{2t}(\tau_{\mathfrak{a}}u_1+u_2)^{p+t}(\bar{\tau}_{\mathfrak{a}}u_1+u_2)^{p-t}\otimes (\bar{\tau}_0 e_1+e_0)^{2t}.
\]
Now compare with the definitions of $w^-$ and $w^+$.
\end{proof}

Given an ideal class $\cA \in \Cl_K$ represented by an integral ideal $\fa$ of $\oh_K$ coprime to $M$, we recall from \eqref{def:Za} and \eqref{def:Zbara} the cycles 
\[
Z_{\cA}:=\chi(\fa)^{-1} \sum_{Q \in \pi_{M,N}^{-1}(P_\fn)} \epsilon Y_Q^\fa \in \CH^{k+t}(X)_{0,K(\chi)},
\]
\[
\bar{Z}_{\cA}:=\bar{\chi}(\fa)^{-1} \sum_{Q \in \pi_{M,N}^{-1}(P_\fn)} \bar{\epsilon} Y_Q^\fa \in \CH^{k+t}(X)_{0,K(\chi)}.
\]  
Recall also the notation $Z=Z_{[\oh_K]}$ and $\bar{Z}=\bar{Z}_{[\oh_K]}$.

Consider the Heegner points $P_{\fn}:=(A \lra A/A[\fn])$ and $P_\fn^{\fa}:=(A^{\mathfrak{a}} \lra A^\mathfrak{a}/A^{\mathfrak{a}}[\fn])$ in $X_0(N)(H)$. If $[\fa]=\cA$, then the generalized Heegner cycles $Z_{\cA}$ and $\bar{Z}_{\cA}$ respectively lie in the fibers above the latter point with respect to $\pi_{M,N}\circ \pi \colon X\lra X(M)\lra X_0(N)$. Remembering that $D$ is odd, we have $\oh_K=\langle 1, \tau_0\rangle$ with $\tau_0=\frac{-1+\sqrt{D}}{2}$. We choose the embedding $\sigma \colon H\hookrightarrow \C$ and the ideal $\fn$ such that $P_{\fn}=\Gamma_0(N)\tau_0$, i.e., $(A \lra A/A[\fn])$ is isomorphic to $(\C/\oh_K \lra \C/\fn^{-1})$ and $\fn^{-1}=\langle 1/N, \tau_0\rangle$. It is the Heegner point associated to $([\oh_K], \fn)$ and $\fn=\langle N, \frac{1+\sqrt{D}}{2} \rangle$ corresponds to $\beta=1 \in \Z/2N\Z$, in the sense of Section \ref{s:HP}.

By definition, we have $A^\mathfrak{a}=A/A[\mathfrak{a}]$. We have $\oh_K\subset \fa^{-1}$, $A[\mathfrak{a}]=\mathfrak{a}^{-1}/\oh_K$, $A^\mathfrak{a}=\C/\mathfrak{a}^{-1}$ and $\phi^\mathfrak{a} \colon A\lra A^\fa$ is the natural quotient isogeny $\C/\oh_K\lra \C/\fa^{-1}$. The isogeny $A^\mathfrak{a}\lra A^{\mathfrak{a}}/A^{\mathfrak{a}}[\fn]$ is isomorphic to $\C/\fa^{-1}\lra \C/\fa^{-1}\fn^{-1}$. Thus, the Heegner point $P_{\fn}^{\fa}$ corresponds via the description of Section \ref{s:HP} to $(\cA^{-1}, \fn)$. Thus, $\fa^{-1}=\langle 1, \tau_\fa\rangle$ with $\tau_{\fa}:=\frac{-B+\sqrt{D}}{2A}\in \cH$ a solution to a quadratic equation
\[
AX^2+BX+C, \qquad A>0, \qquad B^2-4AC= D, \qquad B\equiv \beta\pmod{2N}, \qquad N\mid A.
\]
If we write $\tau_\fa=x_{\fa}+iy_{\fa}$ and $\tau_0=x_0+iy_0$, then 
\[
y_{\fa}=\frac{\sqrt{\vert D\vert}}{2A}=\frac{y_0}{A}.
\]
Moreover, $N_{K/\Q}(\fa)=A$ and $P_{\fn}^{\fa}=\Gamma_0(N)\tau_{\fa}$. 

\begin{lemma}\label{lem:cl}
We have
\[
\cl_{P_{\fn}^{\fa}}(Z_{\cA})=\frac{(-1)^p 2^{p-2t} \sqrt{\vert D \vert}^{p}}{\binom{2p}{p-t}^{1/2}}\frac{w^-(P^{\fa}_{\fn})}{\chi(\fa)} \in \cW_{P_{\fn}^{\fa}}
\]
and
\[
\cl_{P_{\fn}^{\fa}}(\bar{Z}_{\cA})=\frac{(-1)^p 2^{p-2t} \sqrt{\vert D \vert}^{p}}{\binom{2p}{p-t}^{1/2}}  \frac{w^+(P^{\fa}_{\fn})}{y_{\fa}^{2t}\chi(\bar{\fa})}\in \cW_{P_{\fn}^{\fa}}.
\]
\end{lemma}

\begin{proof}
This follows by applying Proposition \ref{prop:cycleclass} with $y_0=\sqrt{\vert D\vert}/2$ and using \eqref{eq:eBw}. 
\end{proof}

\section{Archimedean local heights}\label{s:archiih}

In this section, we compute the archimedean local heights of generalized Heegner cycles using the connection \eqref{eq: brylinski local height} with Brylinski's archimedean pairing.

\subsection{The case $r_{\cA}(m)=0$}

We maintain the notations of Section \ref{s:hhodg}.

\begin{theorem}\label{thm:arch}
Let $(m, N)=1$, $\cA_1=[\fa_1], \cA_2=[\fa_2] \in \Cl_K$ and assume that $P_{\fn}^{\fa_1}$ and $T_m P_{\fn}^{\fa_2}$ are disjoint in $X_0(N)$. Let $v$ denote the infinite place of $H$ corresponding to the fixed embedding $\sigma \colon H\hookrightarrow \C$ of Section \ref{s:hhodg}. Then 
\[
\langle Z_{\cA_1}, T_m Z_{\cA_2}\rangle_{v}=\frac{\vert S\vert (4m\vert D \vert)^{p-t}}{D^t \binom{2p}{p-t}}\chi(\bar{\fa}_1 \fa_2) \sum_{\gamma\in R_N^m/\pm 1} g(\tau_1, \gamma \tau_2) \alpha(\gamma, \tau_1, \tau_2)^{2t},
\]
and 
\[
\langle \bar{Z}_{\cA_1}, T_m \bar{Z}_{\cA_2}\rangle_{v}=\frac{\vert S\vert (4m\vert D \vert)^{p-t}}{D^t \binom{2p}{p-t}}\chi(\fa_1 \bar{\fa}_2) \sum_{\gamma\in R_N^m/\pm 1} g(\tau_1, \gamma \tau_2) \alpha(\gamma, \bar\tau_1, \bar\tau_2)^{2t},
\]
where $\tau_i:=\tau_{\fa_i}$ so that $\fa_i^{-1}=\langle 1, \tau_i\rangle$ and $P^{\fa_i}_{\fn}=\Gamma_0(N)\tau_i$, for $i=1,2$.
\end{theorem}

\begin{proof}
By definition of the Brylinski pairing, and by compatibility of the actions of Hecke correspondences on Chow groups with the action of Hecke operators on cohomology classes \eqref{heckecoh} (see \cite[Prop. 4.1.3]{scholl}), we have 
\[
\langle Z_{\cA_1}, T_m Z_{\cA_2}\rangle_{v}=\langle \cl_{P^{\fa_1}_{\fn}}(Z_{\cA_1}), T_m \cl_{P^{\fa_2}_{\fn}}(Z_{\cA_2})\rangle^{\mathrm{Br}}.
\]
The result follows by combining Lemma \ref{lem:cl} and Proposition \ref{yes1}.
\end{proof}

Let $\langle \;, \;\rangle^{\GS}_\infty=\sum_{w\vert \infty} \langle \;, \;\rangle_w\epsilon_w$ be the sum of the local heights on $X$ over all the infinite places of $H$. Recall that we previously fixed a particular embedding $\sigma \colon H \hookrightarrow \C$ in Section \ref{s:hhodg} corresponding to a place $v$. If $Z_1, Z_2 \in \CH^{k+t}(X)_0$ are two cycles for which the height pairing $\langle Z_1, Z_2\rangle^{\GS}_\infty$ is defined (see Section \ref{sec:height}), and which lie in distinct fibers of $\pi \colon X\lra \bar{C}$ over $Q_1$ and $Q_2$ respectively, then $\langle Z_1, Z_2\rangle_w=\langle \cl_{Q_1}(Z_{1,w}), \cl_{Q_{2}}(Z_{2,w})\rangle^{\mathrm{Br}}$, where $Z_{j,w}$ is viewed as a cycle on the base-change $X_{w}=X\times_{w} \C$.

\begin{theorem}\label{thm:nonint}
Let $(m,N)=1$, $\cA\in \Cl_K$ and assume that $r_{\cA}(m)=0$. Then 
\[
\langle Z, T_m Z_{\cA}\rangle_{\infty}^{\GS}=-\frac{\vert S\vert (4m\vert D \vert)^{p-t}}{D^t \binom{2p}{p-t}}u^2 \sum_{n=1}^\infty \sigma_{\bar{\cA}}(n)r_{\cA, \chi}(m\vert D\vert + nN)Q_{k,t}\left(1+\frac{2nN}{m\vert D\vert}\right),
\]
\[
\langle \bar{Z}, T_m \bar{Z}_{\cA}\rangle_{\infty}^{\GS}=-\frac{\vert S\vert (4m\vert D \vert)^{p-t}}{D^t \binom{2p}{p-t}}u^2 \sum_{n=1}^\infty \sigma_{\cA}(n)r_{\bar{\cA}, \chi}(m\vert D\vert + nN)Q_{k,t}\left(1+\frac{2nN}{m\vert D\vert}\right),
\]
and
\[
\langle Z, T_m \bar{Z}_{\cA}\rangle_{\infty}^{\GS}=0=\langle \bar{Z}, T_m Z_{\cA}\rangle_{\infty}^{\GS}.
\]
\end{theorem}

\begin{proof}
The set of infinite places of $H$ is in bijection with the class group $\Cl_K$ of $K$. If $w$ is an infinite place corresponding to the ideal class $\mathcal{B}\in \Cl_K$, then 
\begin{equation}\label{galinf}
\langle Z_{\cA_1}, Z_{\cA_2} \rangle_w=\langle Z_{\mathcal{B}\cA_1}, Z_{\mathcal{B}\cA_2} \rangle_v,
\end{equation}
for all $\cA_1, \cA_2 \in \Cl_K$. The proof of \eqref{galinf} is an adaption of \cite[Lem. 4.7]{pGZshnidman}.

Using \eqref{galinf}, we see that
\begin{equation}\label{summation}
\langle Z, T_m Z_{\cA}\rangle_{\infty}^{\GS}=\sum_{w\mid \infty} \langle Z, T_m Z_{\cA}\rangle_{w}=\sum_{\mathcal{B}\in \Cl_K} \langle Z_{\cB}, T_m Z_{\mathcal{B} \cA}\rangle_v=\sum_{\substack{\cA_1, \cA_2\in \Cl_K \\ \cA_1^{-1}\cA_2=\cA}} \langle Z_{\cA_1}, T_m Z_{\cA_2}\rangle_v.
\end{equation}
Let $\cA_i=[\fa_i]$ with $\fa_i$ an ideal of $\oh_K$, $i=1,2$.
The condition $r_{\cA_1^{-1}\cA_2}(m)=0$ implies that $P_{\fn}^{\fa_1}$ and $T_m P_{\fn}^{\fa_2}$ are disjoint in $X_0(N)$. We can therefore apply Theorem \ref{thm:arch} to obtain
\[
\langle Z, T_m Z_{\cA}\rangle^{\GS}_{\infty}=\frac{\vert S\vert (4m\vert D \vert)^{p-t}}{D^t \binom{2p}{p-t}} \sum_{\substack{\cA_1, \cA_2\in \Cl_K \\ \cA_1^{-1}\cA_2=\cA}} \chi(\bar{\fa}_1 \fa_2) \sum_{\gamma\in R_N^m/\pm 1} g(\tau_1, \gamma \tau_2) \alpha(\gamma, \tau_1, \tau_2)^{2t}.
\]
Recall that for $\gamma=\left( \begin{smallmatrix} a&b\\c&d\end{smallmatrix}\right)\in R_N^m$, we have $\alpha(\gamma, \tau_1, \tau_2)=c\bar{\tau}_1\tau_2+d\bar{\tau}_1-a\tau_2-b$ with $\tau_i=\tau_{\fa_i}$ such that $\langle 1, \tau_i \rangle= \fa_i^{-1}$. In particular, $\bar{\tau}_1\in \bar{\fa}_1^{-1}$ and $\alpha(\gamma, \tau_1, \tau_2)\in (\bar{\fa}_1\fa_2)^{-1}$. Let $\fa=\bar{\fa}_1\fa_2$. This is an integral ideal of $\oh_K$ such that $[\fa]=\cA$ and we have $\alpha(\gamma, \tau_1, \tau_2)\in \fa^{-1}$. For any positive integer $j$, we have 
\[
r_{\cA, \chi}(j):=\sum_{\substack{\fa'\subset \oh_K, [\fa']=\cA \\ N(\fa')=j}} \chi(\fa').
\]
Every ideal $\fa'$ satisfying $[\fa']=\cA$ can be written as $\fa'=x\fa$ for some $x\in K$. The fact that $\fa'$ is an integral ideal translates into the requirement $x\in \fa^{-1}$. Note that $N_{K/\Q}(\fa)=A_1 A_2$. Thus, we obtain 
\begin{equation}\label{rchim}
r_{\cA, \chi}(j)=\frac{1}{\# \oh_K^\times}\sum_{\substack{x\in \fa^{-1} \\ N_{K/\Q}(x)=\frac{j}{A_1A_2}}} \chi(x\fa)=\frac{\chi(\fa)}{\# \oh_K^\times} \sum_{\substack{x\in \fa^{-1} \\ N_{K/\Q}(x)=\frac{j}{A_1A_2}}} x^{2t}.
\end{equation}
We used the fact that $\chi$ is a homomorphism $I_K \lra \C^\times$ with the property $\chi((\alpha))=\alpha^{2t}$ (since $\chi$ is assumed to be unramified). In particular, $\chi((x\alpha))=\alpha^{2t}$ for any unit $x\in \oh_K^\times$. 

Recall that  
\begin{equation}\label{eq:g}
g(\tau_1, \gamma\tau_2)=-Q_{k,t}\left( 1+\frac{\vert\tau_1-\gamma\tau_2\vert^2}{2y_1\Im(\gamma \tau_2)} \right).
\end{equation}
Define $\beta(\gamma, \tau_1, \tau_2)=\alpha(\gamma, \bar{\tau}_1, \tau_2)=c\tau_1\tau_2+d\tau_1-a\tau_2-b$ and note that 
\[
\frac{\vert\tau_1-\gamma\tau_2\vert^2}{2y_1\Im(\gamma \tau_2)}=\frac{2A_1 A_2 \vert \beta(\gamma, \tau_1, \tau_2) \vert^2}{\vert D\vert \det(\gamma)}=\frac{2A_1 A_2 N_{K/\Q}(\beta(\gamma, \tau_1, \tau_2))}{\vert D\vert \det(\gamma)}.
\]
The elements $\alpha(\gamma, \tau_1, \tau_2)$ and $\beta(\gamma, \tau_1, \tau_2)$ are the same as the ones denoted $\alpha$ and $\beta$ in \cite[II (3.6)]{GZ} except that the roles of $\tau_1$ and $\tau_2$ are swapped. We have 
\[
\alpha(\gamma, \tau_1, \tau_2)\in (\bar{\fa}_1\fa_2)^{-1}= \fa^{-1} \qquad  \text{ and } \qquad \beta(\gamma, \tau_1, \tau_2)\in (\fa_1\fa_2)^{-1}\fn,
\]
and we set 
\[ 
n=\frac{A_1 A_2}{N} N_{K/\Q}(\beta(\gamma, \tau_1, \tau_2)).
\]
This is an integer by the same calculation as in \cite[II (3.5)]{GZ} (the only difference being that the roles of $\tau_1$ and $\tau_2$ are swapped there). 
Moreover, we have 
\[
N_{K/\Q}(\alpha(\gamma, \tau_1, \tau_2))=\frac{m\vert D\vert+nN}{A_1A_2},
\]
as in \cite[II (3.9)]{GZ}.
With these notations, we have 
\begin{equation}
g(\tau_1, \gamma\tau_2)=-Q_{k,t}\left( 1+\frac{2nN}{m\vert D\vert} \right).
\end{equation}
Define 
\[
\rho^m_{\cA_1, \cA_2}(n):=\# \{ \gamma \in R_N^m/\pm 1 \mid A_1 A_2 N_{K/\Q}(\beta(\gamma, \tau_1, \tau_2)) = nN \}.
\]
According to \cite[II (3.9)]{GZ}, the association $\gamma \mapsto (\alpha(\gamma, \tau_1, \tau_2), \beta(\gamma, \tau_1, \tau_2))$ gives a bijection from $\{ \gamma \in R_N^m/\pm 1 \mid A_1 A_2 N_{K/\Q}(\beta(\gamma, \tau_1, \tau_2)) = nN \}$ to the set
\begin{multline*}
\Xi_{\cA_1, \cA_2}(n):=\left\{ (\alpha, \beta) \in (\fa^{-1} \times (\fa_1\fa_2)^{-1}\fn)/\pm 1 \mid N_{K/\Q}(\alpha)=\frac{m\vert D\vert + nN}{A_1A_2}, \right. \\ \left. N_{K/\Q}(\beta)=\frac{nN}{A_1A_2}, A_1A_2\alpha \equiv A_1A_2 \beta\pmod{\mathfrak{d}}  \right\}.
\end{multline*}
Putting all this together yields 
\[
\langle Z, T_m Z_{\cA}\rangle^{\GS}_{\infty} =-\frac{\vert S\vert (4m\vert D \vert)^{p-t}}{D^t \binom{2p}{p-t}} \sum_{n=1}^\infty Q_{k,t}\left( 1+\frac{2nN}{m\vert D\vert} \right) 
\sum_{\substack{\cA_1, \cA_2\in \Cl_K \\ \cA_1^{-1}\cA_2=\cA}}
 \chi(\fa)\sum_{(\alpha, \beta)\in \Xi_{\cA_1, \cA_2}(n)} \alpha^{2t}.
\]
In view of \eqref{rchim} and the proof of \cite[II (3.16)]{GZ}, we obtain
\[
\sum_{\substack{\cA_1, \cA_2\in \Cl_K \\ \cA_1^{-1}\cA_2=\cA}}
 \chi(\fa)\sum_{(\alpha, \beta)\in \Xi_{\cA_1, \cA_2}(n)} \alpha^{2t}=u^2\delta(n)R_{\{ \bar{\cA}[\fn] \}}(n)r_{\cA, \chi}(m\vert D\vert + nN).
 \]
Here, $R_{\{ \bar{\cA}[\fn] \}}(n)$ is the number of integral ideals of norm $n$ in the genus of $\bar{\cA}\fn$, $\delta(n)$ is $2^{\omega((n,D))}$, and $u=\vert \oh_K^\times \vert/2$. By \cite[IV (4.6)]{GZ}, we have $\sigma_{\bar{\cA}}(n)=\delta(n)R_{\{ \bar{\cA}[\fn] \}}(n)$. We deduce the final formula
 \[
 \langle Z, T_m Z_{\cA}\rangle^{\GS}_{\infty} =-\frac{
 \vert S\vert (4m\vert D \vert)^{p-t}}{D^t \binom{2p}{p-t}} u^2 \sum_{n=1}^\infty \sigma_{\bar{\cA}}(n)r_{\cA, \chi}(m\vert D\vert + nN) Q_{k,t}\left( 1+\frac{2nN}{m\vert D\vert} \right).
 \]
 The calculation for $\langle \bar{Z}, T_m \bar{Z}_{\cA}\rangle_{\infty}^{\GS}$ can be checked similarly.
\end{proof}

\subsection{The case $r_{\cA}(m)\neq 0$}

In this section, we prove the following generalization of Theorem \ref{thm:nonint}, valid even in the case of improper intersection:

\begin{theorem}\label{archht}
Given $\cA=[\fa]\in \Cl_K$ and $(m,N)=1$, we have
\begin{align*}
\langle Z, T_m Z_{\cA} \rangle^{\GS}_{\infty}& =\frac{\vert S\vert (4m\vert D \vert)^{k-t-1}}{D^t \binom{2k-2}{k-t-1}}\left[ -u^2 \sum_{n=1}^\infty \sigma_{\bar{\cA}}(n)r_{\cA, \chi}(m\vert D\vert + nN)Q_{k,t}\left(1+\frac{2nN}{m\vert D\vert}\right) \right. \\
& \left. + h u D^t r_{\cA, \chi}(m)\left( \frac{\Gamma'}{\Gamma}(k+t)+\frac{\Gamma'}{\Gamma}(k-t)-2\log(2\pi)+2\frac{L'}{L}(1, \epsilon_K)+\log(\vert D\vert) \right) \right],
\end{align*}
\begin{align*}
\langle \bar{Z}, T_m \bar{Z}_{\cA} \rangle^{\GS}_{\infty}& =\frac{\vert S\vert (4m\vert D \vert)^{k-t-1}}{D^t \binom{2k-2}{k-t-1}}\left[ -u^2 \sum_{n=1}^\infty \sigma_{\cA}(n)r_{\bar{\cA}, \chi}(m\vert D\vert + nN)Q_{k,t}\left(1+\frac{2nN}{m\vert D\vert}\right) \right. \\
& \left. + h u D^t r_{\bar{\cA}, \chi}(m)\left( \frac{\Gamma'}{\Gamma}(k+t)+\frac{\Gamma'}{\Gamma}(k-t)-2\log(2\pi)+2\frac{L'}{L}(1, \epsilon_K)+\log(\vert D\vert) \right) \right],
\end{align*}
and
\[
\langle Z, T_m \bar{Z}_{\cA}\rangle_{\infty}^{\GS}=0=\langle \bar{Z}, T_m Z_{\cA}\rangle_{\infty}^{\GS}.
\]
\end{theorem}

We will give a detailed proof of the formula for $\langle Z, T_m Z_{\cA} \rangle^{\GS}_{\infty}$. The one for $\langle \bar{Z}, T_m \bar{Z}_{\cA} \rangle^{\GS}_{\infty}$ can be deduced similarly and is left to the reader.

\subsubsection{Gross and Zagier's modification}

We begin by reviewing the modifications required in \cite{GZ}, where the height pairing is between divisors of degree zero on $X_0(N)$. Let $a$ and $b$ be two divisors of degree zero with common support equal to a point $x\in X_0(N)(\C)$. For every point $y\in X_0(N)(\C)$ near $x$ but not in the support of $b$, define $a_y$ to be the divisor obtained from $a$ by replacing every occurrence of $x$ by $y$. In other words, if $a=\sum_{z\neq x} m_z (z)+m_x (x)$, then $a_y=\sum_{z\neq x} m_z (z)+m_x(y)$. Then the local intersection $\langle a_y, b \rangle_v$ is well-defined and can be used to approximate $\langle a, b\rangle_v$. More precisely, let $t_0$ be a uniformizing parameter for $x$ on $X_0(N)$, i.e., a function such that $\ord_x(t_0)=1$. Then Gross defined in \cite{grossCS} the local archimedean height pairing 
\begin{equation}\label{GZinter}
\langle a, b\rangle_v:=\lim\limits_{y\rightarrow x} (\langle a_y, b \rangle_v - \ord_x(a)\ord_x(b)\log \vert t_0(y)\vert_v)
\end{equation}
and showed that, with an appropriate definition of local non-archimedean height pairings, $\sum_v \langle a, b\rangle_v$ is independent of the choice of $t_0$ and equal to the global N\'eron--Tate height pairing $\langle a, b\rangle^{\mathrm{NT}}$ (which is well-defined on linear equivalence classes of divisors even if there is common support since one can use a moving lemma).

If $x$ is a Heegner point, $c=(x)-(\infty)$, $d=(x)-(0)$, $(m,N)=1$, and $\sigma=\sigma_{\cA}\in \Gal(H/K)$, then $\ord_{x}(T_m d^\sigma)=r_{\cA}(m)$. Hence, when $r_{\cA}(m)\neq 0$, the common support of $c$ and $T_m d^{\sigma}$ is non-empty and equal to $x$. Applying \eqref{GZinter} yields
\[
\langle c, T_m d^\sigma\rangle_v=\lim\limits_{y\rightarrow x} (\langle (y)-(\infty), T_m d^\sigma \rangle_v - r_{\cA}(m)\log \vert t_0(y)\vert_v).
\]
The uniformizing parameter in \cite{GZ} is chosen as follows: consider the differential 
\[
\omega=\eta^4(z)\frac{dq}{q}=2\pi i \eta^4(z)dz,
\]
where $\eta(z)=q^{1/24}\prod_{n} (1-q^n)$ is the Dedekind eta function. 
The differential $\omega$ is only well-defined up to sixth roots of unity, but this does not affect the later calculations. Letting $u(x)=\vert \Aut(x)\vert/2$, we have 
\[
\ord_x(\omega)=\frac{1}{u(x)}-1,
\]
and $\omega$ may be normalized such that in a neighborhood of $x$ it takes the form
\[
\omega=(1+O(t_0))t_0^{\frac{1}{u(x)}}\frac{dt_0}{t_0}.
\]
This normalization ensures that for a complex place $v$, we have 
\begin{equation}\label{unifnorm}
\lim\limits_{y\rightarrow x} (\log \vert t_0(y) \vert_v-u(x)\log\vert 2\pi i \eta^4(z)(w-z)\vert_v)=0,
\end{equation}
where $x=\Gamma_0(N)z$ and $y=\Gamma_0(N)w$. Setting $t_1=t_0^{1/u(x)}$ yields a uniformizing parameter for $\Gamma(M)z$ on $X(M)$. 

\subsubsection{Archimedean self-intersections of generalized Heegner cycles over $X(M)$}\label{s:gogo}

Let $\cA=[\fa]\in \Cl_K$ and let $v$ denote our fixed archimedean place. The aim of this section is to compute
$\left\langle Z_\cA, Z_\cA\right\rangle_{v}$.

Define 
\[
\mathcal{G}_1(z, z'):=\frac{4^{p-2t} \vert D \vert^{p}}{\binom{2p}{p-t}(y')^{2t}} \sum_{\gamma\in \Gamma_1(N)} j(\gamma, z)^{-2t}\mu^-_g(\gamma z, z'),
\]
which makes sense for any $z,z'\in \cH$ such that $\Gamma(M)z\neq \Gamma(M)z'$. We then have 
\[
\mathcal{G}_1(\tau_{\fa_1}, \tau_{\fa_2})=\left\langle \epsilon Y^{\fa_1}, \epsilon Y^{\fa_2}\right\rangle_{v}
\]
by Theorem \ref{thm:br} and Proposition \ref{prop:cycleclass}, whenever $\Gamma(M) \tau_{\fa_1}\neq \Gamma(M) \tau_{\fa_2}$.

\begin{proposition}\label{eqint}
We have
\[
\left\langle \epsilon Y^{\fa}, \epsilon Y^{\fa}\right\rangle_{v}=\lim\limits_{w\rightarrow \tau_{\fa}} \left( \mathcal{G}_1(\tau_{\fa}, w)+(-1)^{k+t}(\epsilon Y^{\fa} \cdot \epsilon Y^{\fa})_{\pi^{-1}(\tau_{\fa})} \log \vert t_1(w)\vert_v \right).
\]
\end{proposition}

\begin{proof}
We apply Theorem \ref{thm:ari} to $X \lra \bar{C}=X(M)$. Indeed, the set of Heegner points with CM by orders $\oh \subset \oh_K$ is dense in $X_0(N)$, and the corresponding set of preimages of such points under $\pi_{M,N} \colon \bar{C}\lra X_0(N)$ is dense. In the fiber of $\pi \colon X\lra \bar{C}$ above each such point there is a generalized Heegner cycle whose cycle class provides the desired cohomology class in Theorem \ref{thm:ari} corresponding to $\cl_{\tau_{\fa}}(\epsilon Y^{\fa})$. When $\oh = \oh_K$, it follows from Lemma \ref{lem:cl} that these Hodge classes span the fiber of $\cW_0 \otimes \C$ above that point. The general case follows via isogeny.
\end{proof}

\begin{remark}
    The self-intersection of $\epsilon Y^{\fa}$ in the fiber above $\tau_{\fa}$ is with respect to the intersection form $(\;, \;)_{\tau_{\fa}}$ on $H_{\dR}^{2k+2t-2}(\pi^{-1}(\tau_{\fa}))$ defined by 
\[
(\alpha, \beta)_{\tau_{\fa}}=\int_{\pi^{-1}(\tau_{\fa})} \alpha \wedge \beta.
\]
More precisely, we have 
\[
(\epsilon Y^{\fa} \cdot \epsilon Y^{\fa})_{\pi^{-1}(\tau_{\fa})}=(\cl_{\tau_{\fa}}(\epsilon Y^{\fa}), \overline{\cl_{\tau_{\fa}}(\epsilon Y^{\fa})})_{\tau_{\fa}}=(-1)^{k+t-1} \langle \cl_{\tau_{\fa}}(\epsilon Y^{\fa}), \cl_{\tau_{\fa}}(\epsilon Y^{\fa}) \rangle_{\tau_{\fa}}.
\]
Recalling that $p=k-1$, we deduce that 
\[
(\epsilon Y^{\fa} \cdot \epsilon Y^{\fa})_{\pi^{-1}(\tau_{\fa})}=(-1)^{p+t}\frac{4^{p-2t} \vert D \vert^{p}}{\binom{2p}{p-t}} \langle w^-(\tau_{\fa}), w^-(\tau_{\fa}) \rangle^{\mathrm{Br}}=(-1)^{p+t}\frac{4^{p-2t} \vert D \vert^{p}}{\binom{2p}{p-t}y_{\fa}^{2t}}
\]
by Proposition \ref{prop:cycleclass}. This agrees with the Hodge Index Theorem \cite[Thm. 6.33]{voisin}, which implies that the signature of $(\;, \;)_{\tau_{\fa}}$ restricted to a primitive class of type $(p+t, p+t)$ is non-trivial of signature $(-1)^{p+t}$. See also Zhang's remark \cite[p. 123]{zhang}.
\end{remark}

Putting everything together yields the formula
\[
\left\langle \epsilon Y^{\fa}, \epsilon Y^{\fa}\right\rangle_{v}
=\frac{4^{k-2t-1} \vert D \vert^{k-1}}{\binom{2k-2}{k-t-1}y_{\fa}^{2t}} \lim\limits_{w\to \tau_{\fa}} \left( \sum_{\gamma\in \Gamma(M)} g(\tau_{\fa}, \gamma w)\left(\frac{\bar{\tau}_{\fa}-\gamma w}{2 i y_{\fa}}\right)^{2t} j(\gamma, w)^{2t} - \log \vert t_1(w)\vert_v \right).
\]

\subsubsection{Archimedean self-intersections of generalized Heegner cycles over $X_0(N)$}

Let $t_0$ be a uniformizing parameter for $P_{\fn}^{\fa}=\Gamma_0(N)\tau_{\fa}$ on $X_0(N)$, chosen with the same normalization as in \cite{GZ}. Let $u_{\fn}^{\fa}=u(P_{\fn}^{\fa})$.

\begin{proposition}\label{selfint}
We have
\[
\left\langle Z_{\fn}^{\fa}, Z_{\fn}^{\fa}\right\rangle_{v}=
\frac{\vert S\vert 2^{2k-4t-2} \vert D \vert^{k-1}}{\binom{2k-2}{k-t-1}y_{\fa}^{2t}} \lim\limits_{w\to \tau_{\fa}} \left( \sum_{\gamma\in \Gamma_0(N)/\pm 1} g(\tau_{\fa}, \gamma w)\left(\frac{\bar{\tau}_{\fa}-\gamma w}{2 i y_{\fa}}\right)^{2t} j(\gamma, w)^{2t} -\log \vert t_0(w)\vert_v \right).
\]
\end{proposition}

\begin{proof}
Write $Q_i=\Gamma(M)\tau_i$ for the points of $X(M)$ sitting above $P_{\fn}^{\fa}$. Write $Z_i=\epsilon Y_{Q_i}^{\fa}$. Just as in the proof of Theorem \ref{thm:br2}, we have 
\[
\vert S\vert^{-1}\left\langle Z_{\fn}^{\fa},  Z_{\fn}^{\fa}\right\rangle_{v}= \left\langle \sum_{i} Z_i, Z_1\right\rangle_v=\langle Z_1, Z_1\rangle_v + \sum_{i\neq 1} \langle Z_i, Z_1 \rangle_v.
\]
Setting $t_1=t_0^{\frac{1}{u_{\fn}^{\fa}}}$ gives a uniformizing parameter for $Q_1$ on $X(M)$. 
We then have 
\begin{align*}
\vert S\vert^{-1}\left\langle Z_{\fn}^{\fa},  Z_{\fn}^{\fa}\right\rangle_{v}&=\lim\limits_{w\rightarrow \tau_1} \left( \mathcal{G}_1(\tau_{1}, w)+(-1)^{k+t} (Z_1 \cdot Z_1)_{\pi^{-1}(\tau_{1})} u_{\fn}^{\fa}\log \vert t_1(w)\vert_v \right)+\sum_{i\neq 1} \mathcal{G}_1(\tau_i, \tau_1) \\
& =\lim\limits_{w\rightarrow \tau_1} \left(  \sum_{i} \mathcal{G}_1(\tau_{i}, w)+(-1)^{k+t}(Z_1 \cdot Z_1)_{\pi^{-1}(\tau_{1})} \log \vert t_0(w)\vert_v \right).
\end{align*}
The formula then follows as in the proof of Theorem \ref{thm:br2}. 
\end{proof}

\subsubsection{Archimedean local heights of generalized Heegner cycles}

Given $(m,N)=1$ and $\cA\in \Cl_K$ we ultimately wish to compute $\langle Z, T_m Z_{\cA} \rangle_\infty$ and $\langle \bar{Z}, T_m \bar{Z}_{\cA} \rangle_\infty$ in the case $r_{\cA}(m)\neq 0$. 
Let $v$ be our fixed archimedean place as chosen above. Let $\cA_1, \cA_2\in \Cl_K$ such that $\cA_1^{-1}\cA_2=\cA$ and let $P_{\fn}^{\fa_i}=\Gamma_0(N)\tau_{\fa_i}$. For the sake of notation, we write $\tau_i=\tau_{\fa_i}$. 

The following elementary observation will be useful in proving the next theorem.

\begin{lemma}\label{lem:normss}
Let $\gamma=\left( \begin{smallmatrix} a&b\\c&d \end{smallmatrix}\right)\in R_N$ with $\det(\gamma)=m$. Assume that $\gamma\tau_{2}=\tau_{1}$. Then 
\[
\vert j(\gamma, \tau_{2})\vert^2=m y_2 y_1^{-1}=m\frac{A_1}{A_2} \qquad \text{ and } \qquad \vert \alpha(\gamma, \tau_1, \tau_{2})\vert^2=4 m y_1 y_2=\frac{m\vert D\vert}{A_1A_2}.
\]
\end{lemma}

\begin{proof}
Since $\gamma\tau_2=\tau_1$, we have in particular $\Im(\gamma \tau_2)=\Im(\tau_1)$, i.e., 
\[
\frac{m y_2}{\vert c\tau_2+d\vert^2}=y_1.
\]
The first equality follows since $j(\gamma, \tau_2)=c\tau_2+d$. The second equality follows from the first after noticing that 
\[
\alpha(\gamma, \tau_1, \tau_{2})=(\bar{\tau}_1-\gamma \tau_2)j(\gamma, \tau_2)=-2iy_1 j(\gamma, \tau_2),
\]
since $\gamma\tau_2=\tau_1$.
\end{proof}

The following result generalizes Theorem \ref{thm:arch}.

\begin{theorem}\label{fullheight}
Given $\cA_1, \cA_2\in \Cl_K$, we have
\begin{align*}
\langle Z_{\cA_1}&, T_m Z_{\cA_2} \rangle_v
=\frac{\vert S \vert (4m\vert D\vert)^{k-t-1}}{D^t\binom{2k-2}{k-t-1}} \chi(\bar{\fa}_1 \fa_2) \left[ \sum_{\substack{\gamma\in R_N^m/\pm 1 \\ \gamma\tau_2\neq \tau_1}} g(\tau_1, \gamma \tau_2) \alpha(\gamma, \tau_1, \tau_2)^{2t} \right. \\
& \left. + \lim\limits_{w\to \tau_{1}} \left(\left( \sum_{\substack{\gamma\in R_N^m/\pm 1 \\ \gamma\tau_2=\tau_1}} \alpha(\gamma, \tau_1, \tau_2)^{2t} \right)   g(\tau_{1}, w) - \left( \sum_{\substack{\gamma\in \Gamma_0(N)\backslash R_N^m \\ \gamma\tau_2=\tau_1}} \alpha(\gamma, \tau_1, \tau_2)^{2t} \right) \log \vert t_0(w)\vert_v \right)\right].  
\end{align*}
\end{theorem}

\begin{proof}
We set $p=k-1$. We compute
\begin{align*}
\langle Z_{\cA_1},& T_m Z_{\cA_2} \rangle_v = \left\langle \cl_{P_{\fn}^{\fa_1}}(Z_{\cA_1}), \sum_{\gamma\in \Gamma_0(N)\backslash R_N^m} \gamma \cdot \cl_{P_{\fn}^{\fa_2}}( Z_{\cA_2})\right\rangle^{\mathrm{Br}} \\
&= \frac{4^{p-2t} \vert D\vert^{p}}{\binom{2p}{p-t} \chi(\fa_1) \chi(\bar{\fa}_2)} \left\langle w^-(P_{\fn}^{\fa_1}), \sum_{\gamma\in \Gamma_0(N)\backslash R_N^m} \gamma \cdot w^-(P_{\fn}^{\fa_2})\right\rangle^{\mathrm{Br}} \\
&=\frac{4^{p-2t} \vert D\vert^p}{\binom{2p}{p-t} \chi(\fa_1) \chi(\bar{\fa}_2)} \left\langle w^-(P_{\fn}^{\fa_1}), \sum_{\gamma\in \pm\Gamma_0(N)\backslash R_N^m} m^{p+t} j(\gamma, \tau_2)^{-2t} w^-(\gamma P_{\fn}^{\fa_2})\right\rangle^{\mathrm{Br}} \\
&=\frac{\vert S\vert 4^{p-2t} \vert D\vert^p}{\binom{2p}{p-t} \chi(\fa_1) \chi(\bar{\fa}_2)} \sum_{\gamma\in \Gamma_0(N)\backslash R_N^m} m^{p+t} j(\gamma, \bar{\tau}_2)^{-2t} \left\langle w^-(\tau_1), w^-(\gamma\tau_2)\right\rangle^{\mathrm{Br}} \\
&=\frac{\vert S\vert 4^{p-2t} \vert D\vert^p}{\binom{2p}{p-t} \chi(\fa_1) \chi(\bar{\fa}_2)} \frac{m^{p-t}}{(2iy_1y_2)^{2t}} \sum_{\gamma'\in \Gamma_0(N)/\pm 1}\sum_{\substack{\gamma\in \Gamma_0(N) \backslash R_N^m \\ \gamma\tau_2\neq \tau_1}} g(\tau_1, \gamma' \gamma \tau_2) \alpha(\gamma' \gamma, \tau_1, \tau_2)^{2t} \\
&   \qquad + \frac{1}{\chi(\fa_1) \chi(\bar{\fa}_2)} \sum_{\substack{\gamma\in \Gamma_0(N)\backslash R_N^m \\ \gamma\tau_2=\tau_1}} m^{p+t} j(\gamma, \bar{\tau}_2)^{-2t} \left\langle Z_{\fn}^{\fa_1}, Z_{\fn}^{\fa_1}\right\rangle_{v}.
\end{align*}
Using Proposition \ref{selfint}, we then obtain
\begin{align*}
\langle Z_{\cA_1},& T_m Z_{\cA_2} \rangle_v \\
&=\frac{\vert S\vert 4^{p-2t} \vert D\vert^p}{\binom{2p}{p-t} \chi(\fa_1) \chi(\bar{\fa}_2)} \left[ \frac{m^{p-t}}{(2iy_1y_2)^{2t}} \sum_{\gamma'\in \Gamma_0(N)/\pm 1}\sum_{\substack{\gamma\in \Gamma_0(N) \backslash R_N^m \\ \gamma\tau_2\neq \tau_1}} g(\tau_1, \gamma' \gamma \tau_2) \alpha(\gamma' \gamma, \tau_1, \tau_2)^{2t} \right. \\
&  \left. \qquad  + \sum_{\substack{\gamma\in \Gamma_0(N)\backslash R_N^m \\ \gamma\tau_2=\tau_1}} m^{p+t} j(\gamma, \bar{\tau}_2)^{-2t} \lim\limits_{w\to \tau_{1}} \left( \sum_{\gamma'\in \Gamma_0(N)/\pm 1} g(\tau_{1}, \gamma' w)\left(\frac{\alpha(\gamma', \tau_1, w)}{2 i y_{1} \Im(w)}\right)^{2t} \right. \right. \\ 
& \left.\left. \qquad \qquad - \frac{1}{y_1^{2t}} \log \vert t_0(w)\vert_v \right) \right].  \\
\end{align*}

Next, we consider the term
\begin{equation}\label{eq:limi}
\sum_{\substack{\gamma\in \Gamma_0(N)\backslash R_N^m \\ \gamma\tau_2=\tau_1}}  m^{p+t} j(\gamma, \bar{\tau}_2)^{-2t} \lim\limits_{w\to \tau_{1}} \sum_{\gamma'\in \Gamma_0(N)/\pm 1} g(\tau_{1}, \gamma' w)\left(\frac{\alpha(\gamma', \tau_1, w)}{2 i y_{1} \Im(w)}\right)^{2t}.
\end{equation}
We begin by replacing the limit to $\tau_1$ with a limit to $\tau_2$ by replacing $w$ by $\gamma w$, yielding 
\[
\sum_{\substack{\gamma\in \Gamma_0(N)\backslash R_N^m \\ \gamma\tau_2=\tau_1}}  m^{p+t} j(\gamma, \bar{\tau}_2)^{-2t} \lim\limits_{w\to \tau_{2}}  \sum_{\gamma'\in \Gamma_0(N)/\pm 1} g(\tau_{1}, \gamma'\gamma w)\left(\frac{\alpha(\gamma', \tau_1, \gamma w)}{2 i y_{1} \Im(\gamma w)}\right)^{2t}.
\]
By Lemma \ref{lem:normss}, we see that
$
j(\gamma, \bar{\tau}_2)^{-1}=j(\gamma, \tau_2) m^{-1}y_1 y_2^{-1},
$
and we obtain
\[
\sum_{\substack{\gamma\in \Gamma_0(N)\backslash R_N^m \\ \gamma\tau_2=\tau_1}}  \frac{m^{p-t}y_1^{2t}}{y_2^{2t}} j(\gamma, \tau_2)^{2t} \lim\limits_{w\to \tau_{2}}  \sum_{\gamma'\in \Gamma_0(N)/\pm 1} g(\tau_{1}, \gamma'\gamma w)\left(\frac{\alpha(\gamma', \tau_1, \gamma w)}{2 i y_{1} \Im(\gamma w)}\right)^{2t}.
\]
This is in turn equal to
\[
  \frac{m^{p-t}}{(2iy_1y_2)^{2t}} \lim\limits_{w\to \tau_{2}}  \sum_{\gamma'\in \Gamma_0(N)/\pm 1} \sum_{\substack{\gamma\in \Gamma_0(N)\backslash R_N^m \\ \gamma\tau_2=\tau_1}} g(\tau_{1}, \gamma'\gamma w)\alpha(\gamma', \tau_1, \gamma w)^{2t} j(\gamma, \tau_2)^{2t}.
\]
Next, observe that
\begin{multline*}
\lim\limits_{w\to \tau_{2}} \alpha(\gamma', \tau_1, \gamma w) j(\gamma, \tau_2)=\alpha(\gamma', \tau_1, \gamma \tau_2) j(\gamma, \tau_2)=(\bar{\tau_1}-\gamma'\gamma \tau_2)j(\gamma', \gamma \tau_2)j(\gamma, \tau_2)\\ 
=(\bar{\tau_1}-\gamma'\gamma \tau_2)j(\gamma'\gamma, \tau_2)=\alpha(\gamma'\gamma, \tau_1, \tau_2).
\end{multline*}
We deduce that the desired limit \eqref{eq:limi} is equal to
\begin{multline*}
\frac{m^{p-t}}{(2iy_1y_2)^{2t}} \left( \sum_{\substack{\gamma\in R_N^m/\pm 1 \\ \gamma\tau_2=\tau_1}} \alpha(\gamma, \tau_1, \tau_2)^{2t} \right) \lim\limits_{w\to \tau_{1}}  g(\tau_{1}, w)  \\
+\frac{m^{p-t}}{(2iy_1y_2)^{2t}}  \sum_{\substack{\gamma'\in \Gamma_0(N)/\pm 1 \\ \gamma' \tau_1\neq \tau_1}} \sum_{\substack{\gamma\in \Gamma_0(N)\backslash R_N^m \\ \gamma\tau_2=\tau_1}} g(\tau_{1}, \gamma' \gamma \tau_2)\alpha(\gamma' \gamma, \tau_1, \tau_2)^{2t}.
\end{multline*}

It remains to prove the equality
\begin{multline}\label{eq:limi2}
\lim\limits_{w\to \tau_{2}} \sum_{\substack{\gamma\in \Gamma_0(N)\backslash R_N^m \\ \gamma\tau_2=\tau_1}} \frac{m^{p+t}}{(y_1j(\gamma, \bar{\tau}_2))^{2t}} \log \vert t_0(\gamma w)\vert_v \\ 
= \frac{m^{p-t}}{(2 i y_1 y_2)^{2t}} \left( \sum_{\substack{\gamma\in \Gamma_0(N)\backslash R_N^m \\ \gamma\tau_2=\tau_1}} \alpha(\gamma, \tau_1, \tau_2)^{2t} \right) \lim\limits_{w\to \tau_{1}} \log \vert t(w)\vert_v. 
\end{multline}

We use the equality
\[
\alpha(\gamma, \bar{\tau}_1, \bar{\tau_2})=(\tau_1-\gamma \bar{\tau}_2)j(\gamma, \bar{\tau}_2)=2iy_1 j(\gamma, \bar{\tau}_2)
\]
in order to substitute $(y_1j(\gamma,\bar{\tau}_2))^{-2t}=(2i)^{2t}\alpha(\gamma, \bar{\tau}_1, \bar{\tau_2})^{-2t}$. Using Lemma \ref{lem:normss}, we make the substitution 
\[ 
\alpha(\gamma, \bar{\tau}_1, \bar{\tau_2})^{-1}=\alpha(\gamma, \tau_1, \tau_2)\vert \alpha(\gamma, \tau_1, \tau_2) \vert^{-2}=\frac{\alpha(\gamma, \tau_1, \tau_2)}{4my_1y_2}.
\]
We obtain, 
\[
(y_1j(\gamma,\bar{\tau_2}))^{-2t}=\frac{(2i)^{2t} \alpha(\gamma, \tau_1, \tau_2)^{2t}}{(4my_1y_2)^{2t}}=\frac{\alpha(\gamma, \tau_1, \tau_2)^{2t}}{(2 i my_1y_2)^{2t}}.
\]
Equation \eqref{eq:limi2} easily follows.
\end{proof}

\begin{theorem}\label{main.}
Given $\cA=[\fa]\in \Cl_K$ and $(m,N)=1$, and letting $p=k-1$, we have
\begin{align*}
\langle Z, T_m Z_{\cA} \rangle^{\GS}_{\infty}& =\frac{\vert S\vert (4m\vert D \vert)^{k-t-1}}{D^t \binom{2k-2}{k-t-1}}\left[ -u^2 \sum_{n=1}^\infty \sigma_{\bar{\cA}}(n)r_{\cA, \chi}(m\vert D\vert + nN)Q_{k,t}\left(1+\frac{2nN}{m\vert D\vert}\right) \right. \\
& \left. + u D^t r_{\cA, \chi}(m)\sum_{\substack{\cA_1, \cA_2 \in \Cl_K \\ \cA^{-1}\cA_2=\cA}} \lim\limits_{w\to \tau_1} (g(\tau_1, w)- \log \vert 2\pi i \eta^4(\tau_1)(w-\tau_1)\vert_v) \right],
\end{align*}
\end{theorem}

\begin{proof}
We set $p=k-1$.
We proceed as in the proof of Theorem \eqref{thm:nonint}: letting $\beta(\gamma, \tau_1, \tau_2)=\alpha(\gamma, \bar{\tau}_1, \tau_2)=(\tau_1-\gamma \tau_2)j(\gamma, \tau_2)$, we set 
\[
n:=\frac{A_1 A_2}{N} N_{K/\Q}(\beta(\gamma, \tau_1,\tau_2)),
\]
which is a non-negative integer. We remark that $n=0$ if and only if $\gamma\tau_2=\tau_1$. In particular, when $\gamma\tau_2\neq\tau_1$, we have $n>0$ and the exact same proof as for Theorem \ref{thm:nonint} shows that 
\begin{multline*}
\sum_{\substack{\cA_1, \cA_2 \in \Cl_K \\ \cA_1^{-1}\cA_2=\cA}} \frac{\vert S\vert (4m\vert D\vert)^{p-t}}{D^t\binom{2p}{p-t}} \chi(\bar{\fa}_1 \fa_2) \sum_{\substack{\gamma\in R_N^m/\pm 1 \\ \gamma\tau_2\neq \tau_1}} g(\tau_1, \gamma \tau_2) \alpha(\gamma, \tau_1, \tau_2)^{2t} \\
= -\frac{\vert S\vert (4m\vert D \vert)^{p-t}}{D^t \binom{2p}{p-t}}u^2 \sum_{n=1}^\infty \sigma_{\bar{\cA}}(n)r_{\cA, \chi}(m\vert D\vert + nN)Q_{k,t}\left(1+\frac{2nN}{m\vert D\vert}\right).
\end{multline*}

Next, we claim that 
\[
\chi(\bar{\mathfrak{a}}_1\mathfrak{a}_2)\sum_{\substack{\gamma\in R_N^m/\pm 1 \\ \gamma\tau_2=\tau_1}} \alpha(\gamma, \tau_1, \tau_2)^{2t}=D^t u r_{\cA, \chi}(m).
\]
Indeed, when $\gamma \tau_2=\tau_1$ we have the simple expression 
\[
\alpha(\gamma, \tau_1, \tau_2)=-2iy_1(c\tau_2+d)=-\frac{\sqrt{D}}{A_1}j(\gamma, \tau_2).
\]
Therefore, we have 
\[
\chi(\bar{\mathfrak{a}}_1\mathfrak{a}_2)\sum_{\substack{\gamma\in R_N^m/\pm 1 \\ \gamma\tau_2=\tau_1}} \alpha(\gamma, \tau_1, \tau_2)^{2t}=D^t\chi(\bar{\mathfrak{a}}_1\mathfrak{a}_2)\sum_{\substack{\gamma\in R_N^m/\pm 1 \\ \gamma\tau_2=\tau_1}} (A_1^{-1}j(\gamma, \tau_2))^{2t}.
\]
Let $\fa=\bar{\mathfrak{a}}_1\mathfrak{a}_2$. We claim that the assignment $\gamma \mapsto \eta(\gamma):= A_1^{-1}j(\gamma, \tau_2)$ gives a bijection between 
\[
S_1:=\{ \gamma \in R_N^m/\pm 1 \mid \gamma\tau_2=\tau_1 \}
\]
and 
\[
S_2:=\{ \eta \in \fa^{-1}/\pm 1 \mid N_{K/\Q}(\eta)=m/N(\fa) \}.
\]
Recall that $\fa_i=\langle A_i, (B_i+\sqrt{D})/2 \rangle=A_i\langle 1, \bar{\tau}_i \rangle$ and $\fa_i^{-1}=\langle 1, \tau_i \rangle$.
Because $\gamma \tau_2=\tau_1$, we then have 
\[
\eta(\gamma)=A_1^{-1}j(\gamma, \tau_2)=A_1^{-1}\tau_1^{-1}(a\tau_2+b)\in A_1^{-1}\mathfrak{a}_1 \mathfrak{a}_2^{-1}=(A_1 \fa_1^{-1} \fa_2)^{-1}=(\bar{\fa}_1 \fa_2)^{-1}=\fa^{-1}. 
\]
Using Lemma \ref{lem:normss}, we see that 
\[
N_{K/\Q}(\eta(\gamma))=A_1^{-2}\vert j(\gamma, \tau_2)\vert^2=A_1^{-2} m\frac{A_1}{A_2}=\frac{m}{A_1A_2}=\frac{m}{N(\fa)}.
\]
We conclude that $\eta(\gamma)\in S_2$ for all $\gamma \in S_1$. Suppose that $\gamma, \gamma' \in S_1$ satisfy $\eta(\gamma)=\eta(\gamma')$. Then $j(\gamma, \tau_2)=j(\gamma', \tau_2)$. In particular, $(c-c')\tau_2=d'-d$, which implies that $(c-c')y_2=0$, hence $c=c'$ and thus $d=d'$. We conclude that the assignment $\gamma \mapsto \eta(\gamma)$ is injective. There are $ur_{\cA}(m)$ elements $\gamma\in R_N^m/\pm 1$ satisfying $\gamma \tau_2=\tau_1$ \cite[p. 251]{GZ}. Since
\[
\# S_2= \frac{1}{2} \# \left\{ \eta \in \fa^{-1} \mid N_{K/\Q}(\eta)=\frac{m}{N(\fa)} \right\} = \frac{\# \oh_K^\times}{2} r_{\cA}(m)=u r_{\cA}(m),
\]
we deduce that $S_1$ and $S_2$ are indeed in bijection.

We deduce that 
\[
\chi(\fa) \sum_{\substack{\gamma\in R_N^m/\pm 1 \\ \gamma\tau_2=\tau_1}} \alpha(\gamma, \tau_1, \tau_2)^{2t} 
=D^t \frac{\chi(\fa)}{2} \sum_{\substack{\eta \in \fa^{-1} \\ N(\eta)=m/N(\fa)}} \eta^{2t}  
=D^t u r_{\cA, \chi}(m),
\]
where we used \eqref{rchim}.

To finish the proof, we must argue that 
\[
\chi(\fa)\sum_{\substack{\gamma\in \Gamma_0(N) \backslash R_N^m \\ \gamma\tau_2=\tau_1}} \alpha(\gamma, \tau_1, \tau_2)^{2t}=D^t r_{\cA, \chi}(m).
\]
Suppose that $\gamma \in R_{N}^m$ satisfies $\gamma \tau_2 = \tau_1 $ and that there exists $\gamma' \in \Gamma_0(N)/\pm 1$ such that $\gamma'\gamma \tau_2=\tau_1$. Then $\gamma'\tau_1=\tau_1$ and 
\[
\alpha(\gamma'\gamma, \tau_1, \tau_2)=-2iy_1j(\gamma'\gamma, \tau_2)=\alpha(\gamma'\gamma, \tau_1, \tau_2)=-2iy_1j(\gamma, \tau_2)j(\gamma', \tau_1)=\alpha(\gamma, \tau_1, \tau_2)j(\gamma', \tau_1).
\]
We have 
\[
D^t u r_{\cA, \chi}(m)=\chi(\fa) \sum_{\substack{\gamma\in R_N^m/\pm 1 \\ \gamma\tau_2=\tau_1}} \alpha(\gamma, \tau_1, \tau_2)^{2t}=\chi(\fa)\sum_{\substack{\gamma\in \Gamma_0(N) \backslash R_N^m \\ \gamma\tau_2=\tau_1}} \alpha(\gamma, \tau_1, \tau_2)^{2t}\sum_{\substack{\gamma'\in \Gamma_0(N)/\pm 1 \\ \gamma' \tau_1=\tau_1}} j(\gamma', \tau_1)^{2t}.
\]
The problem thus reduces to showing that 
\begin{equation}\label{desi}
\sum_{\substack{\gamma'\in \Gamma_0(N)/\pm 1 \\ \gamma' \tau_1=\tau_1}} j(\gamma', \tau_1)^{2t}=u.
\end{equation}
We have 
\[
\sum_{\substack{\gamma'\in \Gamma_0(N)/\pm 1 \\ \gamma' \tau_1=\tau_1}} j(\gamma', \tau_1)^{2t}=\sum_{\substack{\gamma'\in \Gamma_0(N)/\pm 1 \\ \gamma' \tau_1=\tau_1}} \chi((j(\gamma', \tau_1))).
\]
If $u=1$, then only $\gamma'=1$ appears and $j(\gamma', \tau_1)=1$ so that \eqref{desi} is satisfied trivially. Assume now that $u>1$.
The condition $\gamma'\tau_1=\tau_1$ implies that $N_{K/\Q}(j(\gamma', \tau_1))=1$. If $j(\gamma', \tau_1)\in \oh_K$, then $j(\gamma', \tau_1)\in \oh_K^\times$ so $\chi((j(\gamma', \tau_1)))=1$ and we are done because the sum has size $u$. Note that 
\[
j(\gamma', \tau_1)=\tau_1^{-1}(a\tau_1+b).
\]
We have $a\tau_1+b\in \fa_1^{-1}$, so if $\tau_1^{-1}\in \fa_1$, then $j(\gamma', \tau_1)\in \oh_K$. Thus, it suffices to check that $\tau_1^{-1}\in \fa_1$.
If $u=2$, then $K=\Q(i)$ and $\Gamma_0(N)\tau_1\in X_0(N)$ is a lift of $\SL_2(\Z)i\in X_0(1)$. We may thus assume that $\tau_1=i$. If $u=3$, then $K=\Q(\rho)$ where $\rho=e^{2\pi i/3}$ and we may assume $\tau_1=\rho$. In either case, we have $\vert \tau_1\vert^2=1$ and $A_1=1$, hence $\tau_1^{-1}=\bar{\tau}_1\in \fa_1=A_1\langle 1, \bar{\tau}_1 \rangle$. This concludes the proof.

The logarithmic term is dealt with by using \eqref{unifnorm}.
\end{proof}

We deal with the limit that appears in Theorem \ref{main.}:

\begin{proposition}\label{prop:lim}
We have 
\begin{multline*}
\sum_{\substack{\cA_1, \cA_2 \in \Cl_K \\ \cA^{-1}\cA_2=\cA}} \lim\limits_{w\to \tau_1} (g(\tau_1, w)-\log \vert 2\pi i \eta^4(\tau_1)(w-\tau_1)\vert_v) \\ =h\left( \frac{\Gamma'}{\Gamma}(k+t)+\frac{\Gamma'}{\Gamma}(k-t)-2\log(2\pi)+2\frac{L'}{L}(1,\epsilon_K)+\log(\vert D\vert) \right).
\end{multline*}
\end{proposition}

\begin{proof}
The proof is the same as the proof of \cite[II (5.8)]{GZ}, except that we use Lemma \ref{lem:asy}. 
\end{proof}

Combining Proposition \ref{prop:lim} with Theorem \ref{main.} completes the proof of Theorem \ref{archht}.

\section{Non-archimedean local heights}\label{s:nonarchii}

In this section we compute the non-archimedean contribution to the height pairing of generalized Heegner cycles. More precisely, we derive a formula for 
\[ 
\langle Z, T_m \bar{Z}_{\cA}\rangle_{\mathrm{fin}}^{\GS}:=\sum_{v \nmid \infty}\langle Z, T_m \bar{Z}_{\cA}\rangle_v \epsilon_v.
\]
The analogous formula for $\langle \bar{Z}, T_m Z_{\cA}\rangle_{\mathrm{fin}}^{\GS}$ can be deduced by noting, as in \cite[bottom p. 75]{AriThesis}, that
\begin{equation}\label{eq: barZ nonarch}
\langle \bar{Z}, T_m Z_{\cA}\rangle_{\mathrm{fin}}^{\GS}=
\langle Z, T_m \bar{Z}_{\bar{\cA}}\rangle_{\mathrm{fin}}^{\GS}.
\end{equation}

\subsection{The case $r_{\cA}(m)=0$}

Let $v$ be a non-archimedean place above a prime $q$. Let $m\geq 1$ and assume that $r_{\cA}(m)=0$. Let $L=H_v$ denote the completion at $v$ and let $\Lambda$ be the ring of integers of the maximal unramified extension $L^{\mathrm{ur}}$ of $L$, with residue field $\mathbb{F}=\bar{\mathbb{F}}_q$. Let $\underline{X}_0(N) \lra \spec(\Z)$ be the Katz--Mazur integral model \cite{katzmazur} and denote by $\underline{X}_0(N)_\Lambda \lra \spec(\Lambda)$ the pull-back with special fiber $\underline{X}_0(N)_{\mathbb{F}}\lra \spec(\mathbb{F})=s$ and generic fiber $X_0(N)_{L^{\mathrm{ur}}}\lra \spec(L^{\mathrm{ur}})=\eta$. Let $\iota$ denote the inclusion $Y_0(N)\times_\Q F^{\mathrm{ur}} \hookrightarrow \underline{X}_0(N)_\Lambda$.

Choose a prime $p$ that splits in $K$ and that is distinct from $q$. Consider the following smooth $\Q_p$-sheaf on $Y(M)$: 
\[
\cW_p := \sym^{2k-2} (R^1 \pi_{\cE,*} \Q_p)(k-1)\otimes \kappa_{2t} H^{2t}(\bar{A}^{2t}, \Q_p(t)),
\]
where $\pi_{\cE} \colon \cE \lra Y(M)$ is the universal elliptic curve.
The cycle $\epsilon Y^{\fa}_Q \in \CH^{k+t}(X_F)_{0, K}$ is supported in the fiber of $\pi_X \colon X\lra \bar{C}=X(M)$ over $Q^{\fa}\in Y(M)(F)$, which is identified with $(A_F^{\fa})^{2k-2}\times_F A_F^{2t}$. Its image under the $p$-adic cycle class map on $\pi_X^{-1}(Q^{\fa})$ is denoted
\[
b(Y_Q^{\fa}) \in \epsilon H_{\et}^{2k+2t-2}\left(\pi_X^{-1}(Q^{\fa})_{\bar{F}}, \Q_p(k+t-1)\right)\simeq H^0(\overline{Q^{\fa}}, \cW_p),
\]
and in fact lies in $H^0(\overline{Q^{\fa}}, \cW_p)^{G_F}$.
This is the Tate vector associated with the cycle $\epsilon Y_Q^{\fa}$. Taking the $p$-adic cycle class map with respect to $\pi_X^{-1}(Q^{\fa})$ of the cycle $\bar{\epsilon} Y_Q^{\fa}$ yields a Tate vector that we denote by $\bar{b}(Y_Q^{\fa})$. These Tate vectors completely determine the $p$-adic Abel--Jacobi images of the respective cycles with respect to the variety $X$. Recall the quotient map $\pi=\pi_{M,N} \colon Y(M)\lra Y_0(N)=Y(M) \slash S$ and let $\cW^0_p := (\pi_* \cW_p)^S$, a $\Q_p$-sheaf on $Y_0(N)$. Then $a(Y^{\fa})=\sum_{Q\in \pi^{-1}(P_{\fn})} b(Y_Q^{\fa})$ and $\bar{a}(Y^{\fa})=\sum_{Q\in \pi^{-1}(P_{\fn})} \bar{b}(Y_Q^{\fa})$ are Tate vectors with respect to the sheaf $\cW^0_p$ defined over $H$.

Following \cite{pGZshnidman}, suppose $c$ and $d$ are Tate vectors with respect to the sheaf $\cW^0_p$ supported at points $y_c\neq y_d$ of $Y_0(N)(L^{\mathrm{ur}})$ of good reduction. Let $\underline{y}_c$ and $\underline{y}_d$ be the Zariski closures of the points $y_c$ and $y_d$ in $X_0(N)_{\Lambda}$, and let $\underline{c}$ and $\underline{d}$ be extensions of $c$ and $d$ to $H^0(\underline{y}_c, \iota_* \cW^0_p)$ and $H^0(\underline{y}_d, \iota_* \cW^0_p)$ respectively. If $\underline{y}_c$ and $\underline{y}_d$ have common special fiber $z\in \underline{X}_0(N)_{\mathbb{F}}$ corresponding to an elliptic curve $E\slash \mathbb{F}$, then define 
\[
(c, d)_v:= (\underline{y}_c \cdot \underline{y}_d)_z \cdot (\underline{c}_z, \underline{d}_z), 
\]
where $(\underline{y}_c \cdot \underline{y}_d)_z$ is the usual intersection number at $z$ on the arithmetic surface $\underline{X}_0(N)_{\Lambda}$ and $(\underline{c}_z, \underline{d}_z)$ is the intersection pairing on the $p$-adic \'etale cohomology of $E^{2k-2}\times A_{\mathbb{F}}^{2t}$.

The assumption $r_{\cA}(m)=0$ guarantees that the Tate vectors $a(Y)$ and $T_m a(Y^{\fa})$ have disjoint supports. Moreover, we may assume that they are supported at points of good reduction (because CM elliptic curves have potential good reduction). The number $(a(Y), T_m a(Y^{\fa}))_v$ is thus defined. The cycle $Z_{\fn}^{\fa}$ sits above the Heegner point $P_{\fn}^{\fa}=\Gamma_0(N) \tau_{\fa}\in Y_0(N)(H)$. Letting $\underline{P}_{\fn}^{\fa}$ denote the Zariski closure of $P_{\fn}
^{\fa}$ in $\underline{X}_0(N)_{\Lambda}$, the cycle $Z_{\fn}^{\fa}$ admits an extension $\mathbf{Z}_{\fn}^{\fa}$ sitting in the fiber of $\underline{P}_{\fn}^{\fa}$. The intersection cycle $\mathbf{Z}_{\fn}\cdot T_m \mathbf{Z}_{\fn}^{\fa}$ has a flat presentation over the common support $\vert \underline{P}_{\fn}\vert \cap \vert T_m \underline{P}_{\fn}^{\fa}\vert$ by \cite[Lem. 3.3.2]{zhang} and the fact that the cycles have presentations by abelian subvarieties. Thus, following \cite[(3.3.1)]{zhang}, we obtain 
\[
\mathbf{Z}_{\fn}\cdot T_m \mathbf{Z}_{\fn}^{\fa} = (-1)^{k+t}(a(Y), T_m \bar{a}(Y^{\fa}))_v.
\]
We conclude that 
\begin{equation}\label{GSht}
\langle Z_{\fn}, T_m \bar{Z}_{\fn}^{\fa}\rangle_v \epsilon_v=(a(Y), T_m \bar{a}(Y^{\fa}))_v \log N(v).
\end{equation}

\begin{remark}
Under the assumption $r_{\cA}(m)=0$, the local Brylinski pairing at $v$ of the vectors $a(Y)$ and $T_m \bar{a}(Y^{\fa})$ is defined and, according to \cite[Appendix]{brylinski}, we have 
\[
\langle Z_{\fn}, T_m \bar{Z}_{\fn}^{\fa}\rangle_v \epsilon_v=\langle a(Y), T_m \bar{a}(Y^{\fa})\rangle^{\mathrm{Br}}_v, 
\]
where the local height on the left hand side is the one of Gillet--Soul\'e.
Brylinski moreover gives an explicit description of his local pairing under the assumption that $H^2(\underline{Y}_0(N)_{\mathbb{F}}, \cW_p^0(1))=0$, namely
\[
\langle a(Y), T_m \bar{a}(Y^{\fa})\rangle^{\mathrm{Br}}_v=(a(Y), T_m \bar{a}(Y^{\fa}))_v \log N(v).
\]
The proof that $H^2(\underline{Y}_0(N)_{\mathbb{F}}, \cW_p^0(1))=0$ can be found in \cite[Proof of Prop. 5.1]{pGZshnidman}. 
\end{remark}

We have seen that the calculation of $\langle Z_{\fn}, T_m \bar{Z}_{\fn}^{\fa}\rangle_v \epsilon_v$ boils down to calculating $(a(Y), T_m \bar{a}(Y^{\fa}))_v$. The latter has been calculated already in \cite{pGZshnidman}:
\begin{proposition}\label{prop:nonar1}
Let $m\geq 1$ such that $(m,N)=1$ and $r_{\cA}(m)=0$. Then
\begin{multline*}
(a(Y), T_m \bar{a}(Y^{\fa}))_v \\
=\frac{\vert S\vert}{2} m^{k-t-1} \sum_{n\geq 1} \sum_{g \in \Hom_{\Lambda \slash \pi^n}(\underline{P}_{\fn}^{\fa}, \underline{P}_{\fn})_m} \left( \bar{\epsilon} \left( X^{\otimes k-t-1}_{g\sqrt{D}g^{-1}} \otimes X^{\otimes 2t}_{\overline{g\phi_{\fa}}} \right), \epsilon \left( X_{\sqrt{D}}^{\otimes k-t-1} \otimes X_1^{\otimes 2t} \right) \right). 
\end{multline*}
\end{proposition}

\begin{proof}
This is \cite[Prop. 5.2]{pGZshnidman}, except for the factor of $\vert S\vert$ due to our normalizations. For future use, we summarize the proof. In order to do this, we switch to the notations used in \cite{pGZshnidman} and \cite{nekovar}.

We let $x=a(Y)$ and $\bar{x}^\fa=\bar{a}(Y^\fa)$. These Tate vectors sit above the points $y=P_{\fn}
$ and $y^\sigma=P_{\fn}^{\fa}$ in $X_0(N)$, where $\sigma\in \Gal(H/K)$ corresponds to $[\fa]\in \Cl_K$. We denote by $\underline{x}, \underline{x}^\fa, \underline{y}, \underline{y}^\sigma$ the extensions over $\Lambda$, and by $\underline{x}_s, (\underline{x}^\fa)_s, \underline{y}_s, (\underline{y}^\sigma)_s$ the corresponding special fibers. 

Let $q$ be the prime below $v$ and write $m=m_0 q^s$ with $(m_0, q)=1$. Following \cite[III \S 5]{GZ}, we have 
\[
T_{m_0} \underline{y}^\sigma = \sum_{g\in \Hom_{\Lambda}(\underline{y}^\sigma, \underline{y}^\sigma_g)_{m_0}} \underline{y}^\sigma_g,
\]
and 
\[
T_{q^s} \underline{y}^\sigma_g = \sum_j \underline{y}^\sigma_g(j).
\]
The latter sum is indexed by certain integers between 
$0$ and $s$, and the $\underline{y}^\sigma_g(j)$ are divisors over $\Lambda$ comprised of quasi-canonical liftings of level $p^j$ of their reductions (the notations are borrowed from \cite[III (5.2)]{GZ}). 

By definition of the Hecke operators, we have 
\[
T_m \overline{\underline{x}}^\fa = \sum_{g\in \Hom_{\Lambda}(\underline{y}^\sigma, \underline{y}^\sigma_g)_{m_0}} \sum_j m^{k-t-1} \overline{\underline{x}}_g^\fa(j),
\]
where $\overline{\underline{x}}^\fa(j)$ is a Tate vector sitting above $\underline{y}^\fa(j)$. In the case where $\underline{y}$ and $\underline{y}_g^{\sigma}(j)$ intersect, the special fiber of $\overline{\underline{x}}^\fa(j)$ is represented by \cite[p. 1150]{pGZshnidman}
\begin{equation}\label{specialfiber}
(\overline{\underline{x}}^\fa(j))_s=(\underline{y}_s, \bar{\epsilon}(X_{hg\sqrt{D}g^{-1}h^{-1}}^{\otimes k-t-1} \otimes X_{\overline{hg\phi_\fa}}^{\otimes 2t})),
\end{equation}
for any $h\in \Hom_{\Lambda/\pi^n}(\underline{y}_g^{\sigma}, \underline{y})_{q^s}$ with $n\geq 1$. Note that this special fiber does not depend on $j$.

We deduce that 
\begin{align*}
    \vert S\vert^{-1}(\underline{x}, & T_m \overline{\underline{x}}^{\fa})_v  =   \sum_{g\in \Hom_{\Lambda}(\underline{y}^\sigma, \underline{y}^\sigma_g)_{m_0}} \sum_j m^{k-t-1} (\underline{x}, \overline{\underline{x}}_g^\fa(j)) \\
&=  m^{k-t-1} \sum_{g\in \Hom_{\Lambda}(\underline{y}^\sigma, \underline{y}^\sigma_g)_{m_0}} \sum_j (\underline{y} \cdot \underline{y}_g^{\sigma}(j)) (\bar{\epsilon}(X_{hg\sqrt{D}g^{-1}h^{-1}}^{\otimes k-t-1} \otimes X_{\overline{hg\phi_\fa}}^{\otimes 2t}), \epsilon (X_{\sqrt{D}}^{k-t-1}\otimes X_1^{\otimes 2t})) \\
&=  m^{k-t-1} \sum_{g\in \Hom_{\Lambda}(\underline{y}^\sigma, \underline{y}^\sigma_g)_{m_0}} (\underline{y} \cdot T_{q^s} \underline{y}_g^{\sigma}) (\bar{\epsilon}(X_{hg\sqrt{D}g^{-1}h^{-1}}^{\otimes k-t-1} \otimes X_{\overline{hg\phi_\fa}}^{\otimes 2t}), \epsilon (X_{\sqrt{D}}^{k-t-1}\otimes X_1^{\otimes 2t})).
\end{align*}

By \cite[III (4.4)]{GZ}, we have 
\[
(\underline{y} \cdot T_{q^s} \underline{y}_g^{\sigma}) = \frac{1}{2} \sum_{n\geq 1} \sum_{h'\in \Hom_{\Lambda/\pi^n}(\underline{y}_g^{\sigma}, \underline{y})_{q^s}} 1.
\]
Injecting this into the above expression, we obtain
\begin{align*}
(\underline{x}, & T_m \overline{\underline{x}}^{\fa})_v  \\
&= 
\frac{\vert S\vert m^{k-t-1}}{2} \sum_{g\in \Hom_{\Lambda}(\underline{y}^\sigma, \underline{y}^\sigma_g)_{m_0}} \sum_{n\geq 1} \sum_{h\in \Hom_{\Lambda/\pi^n}(\underline{y}_g^{\sigma}, \underline{y})_{q^s}} (\bar{\epsilon}(X_{hg\sqrt{D}g^{-1}h^{-1}}^{\otimes k-t-1} \otimes X_{\overline{hg\phi_\fa}}^{\otimes 2t}), \epsilon (X_{\sqrt{D}}^{k-t-1}\otimes X_1^{\otimes 2t})) \\
&= \frac{\vert S\vert m^{k-t-1}}{2} \sum_{n\geq 1} \sum_{g\in \Hom_{\Lambda/\pi^n}(\underline{y}^{\sigma}, \underline{y})_{m}} (\bar{\epsilon}(X_{g\sqrt{D}g^{-1}}^{\otimes k-t-1} \otimes X_{\overline{g\phi_\fa}}^{\otimes 2t}), \epsilon (X_{\sqrt{D}}^{k-t-1}\otimes X_1^{\otimes 2t})),
\end{align*}
where the first equality holds because \eqref{specialfiber} is true for any $h$. 
\end{proof}

The right hand side of the formula of Proposition \ref{prop:nonar1} is $0$ when the prime $q$ splits in $K$ by \cite[III (7.1)]{GZ} (which shows that the double sum is empty). When $q$ is non-split, $\End_{\Lambda\slash \pi}(\underline{P}_{\fn}
)=R$ is an order in the quaternion algebra $B$ ramified at $\infty$ and $q$. In this case, the double sum above can be made more explicit. The reduction of endomorphisms induces an embedding $K\hookrightarrow B$, which yields a decomposition 
\[
B=B_+ \oplus B_- = K \oplus Kj,
\]
where $j$ is an element in the non-trivial coset of $N_{B^\times}(K^\times)\slash K^\times$. The reduced norm on $B$ is additive with respect to this decomposition. Every $b\in B$ will be written as $b=\alpha + b_- = \alpha + \beta j$, and we then have $N(b)= N(\alpha) + N(b_-) = N(\alpha) + N(\beta j)$.

\begin{proposition}\label{prop:nonar2}
Let $m\geq 1$ such that $(m,N)=1$ and $r_{\cA}(m)=0$. Then
\[
(a(Y), T_m \bar{a}(Y^{\fa}))_v= \vert S \vert m^{k-t-1} \frac{(4D)^{k-t-1}}{\binom{2k-2}{k-t-1}} \sum_{\substack{b=\alpha+b_- \in R\fa \slash \pm 1 \\ N(b)=mN(\fa) }} \bar{\alpha}^{2t} U(b_-) P_{k, t}\left( 1 - \frac{2N(b_-)}{N(b)} \right),
\]
where
\[
U(b_-)=
\begin{cases}
0 & \text{ if } q \text{ splits in } K \\
\frac{1}{2}(1+\ord_q(N(b_-))) & \text{ if } q \text{ is inert in } K \\
\ord_q(\vert D\vert N(b_-)) & \text{ if } q \text{ is ramified in } K,
\end{cases}
\]
and $P_{k,t}(x)$ is the Jacobi polynomial introduced in \eqref{Pkt}.
\end{proposition}

\begin{proof}
By the previous discussion, we may assume that $q$ is non-split in $K$.
According to \cite[III (7.3)]{GZ}, we then have
\[
\End_{W\slash \pi^n}(\underline{P}_{\fn})=\{ b=\alpha + b_- \in R \colon \vert D\vert N(b_-)\equiv 0 \mod{q(N(\fq))^{n-1}}\},
\]
and the association $g \mapsto g\phi_{\fa}=b=\alpha+b_-$ yields an identification between $\Hom_{W\slash \pi^n}(\underline{P}^{\fa}_{\fn}, \underline{P}_{\fn})$ and $\End_{W\slash \pi^n}(\underline{P}_{\fn})\cdot \fa$ in $B$ such that $\deg(g)=N(b)\slash N(\fa)$. Under this identification, we have \cite[Prop. 5.3]{pGZshnidman}\footnote{The function $H_{k-t-1, t}(x)$ in \cite{pGZshnidman} is equal to the function $P_{k,t}(x)$ defined in \eqref{Pkt}.}
\[
\left( \bar{\epsilon} \left( X^{\otimes k-t-1}_{g\sqrt{D}g^{-1}} \otimes X^{\otimes 2t}_{\overline{g\phi_{\fa}}} \right), \epsilon \left( X_{\sqrt{D}}^{\otimes k-t-1} \otimes X_1^{\otimes 2t} \right) \right)=\frac{(4D)^{k-t-1}}{\binom{2k-2}{k-t-1}} \bar{\alpha}^{2t}P_{k, t}\left( 1 - \frac{2N(b_-)}{N(b)} \right).
\]
Combining this with Proposition \ref{prop:nonar1} yields
\[
(a(Y), T_m \bar{a}(Y^{\fa}))_v=\frac{\vert S\vert}{2} m^{k-t-1} \frac{(4D)^{k-t-1}}{\binom{2k-2}{k-t-1}} \sum_{n\geq 1} \sum_{\substack{b=\alpha+b_- \in R\fa \\ \substack{N(b)=mN(\fa) \\ \vert D\vert N(b_-)\equiv 0 \mod{q(N(\fq))^{n-1}}}}} \bar{\alpha}^{2t} P_{k,t}\left( 1 - \frac{2N(b_-)}{N(b)} \right).
\]
The result follows after interchanging the order of summation in the double sum and noting that 
\[
\# \{ n\geq 1 \colon \vert D\vert N(b_-) \equiv 0 \mod q(N(\fq)^{n-1})\} = U(b_-).
\]
\end{proof}

Let $\langle \;, \; \rangle^{\GS}_{\mathrm{fin}}=\sum_{v\nmid \infty} \langle \;, \;\rangle_v \epsilon_v$ be the sum of the local heights on $X$ over all the finite places of $H$.

\begin{proposition}
Assume that $(m, N)=1$ and $r_{\cA}(m)=0$. Then
\[
\langle Z, T_m \bar{Z}_{\cA}\rangle_{\mathrm{fin}}^{\GS} = - u^2 \frac{\vert S\vert (4\vert D\vert m)^{k-t-1}}{D^t \binom{2k-2}{k-t-1}} \sum_{0<n< \frac{m\vert D\vert}{N}} \sigma_{\cA}(n) r_{\cA, \chi}(m\vert D\vert - nN) P_{k,t}\left( 1- \frac{2nN}{m\vert D\vert} \right).
\]
\end{proposition}

\begin{proof}
The only difference between the formula of Proposition \ref{prop:nonar2} and the one of \cite[II Prop. 4.15]{nekovar} is the fact that the sum is weighted by $\bar{\alpha}^{2t}$. The factor responsible for the appearance of $r_\cA(m\vert D\vert - nN)$ in \cite{nekovar} is thus replaced by 
\[
\sum_{\substack{\fc \subset \oh_K \\ \substack{[\fc]=\cA^{-1} \\ N(\fc)=m\vert D\vert - nN}}} \bar{\alpha}^{2t},
\]
where $\fc=(\alpha)\fd \fa^{-1}$ and $\alpha\in \fd^{-1} \fa$. But this expression is equal to 
\[
\sum_{\substack{\fc \subset \oh_K \\ \substack{[\fc]=\cA^{-1} \\ N(\fc)=m\vert D\vert - nN}}} \bar{\chi}(\fc \fa \fd^{-1})=
\frac{\chi(\bar{\fa})}{D^t} r_{\bar{\cA}, \bar{\chi}}(m\vert D\vert - nN)=\frac{\chi(\bar{\fa})}{D^t} r_{\cA, \chi}(m\vert D\vert - nN).
\]
\end{proof}

\subsection{The case $r_{\cA}(m)\neq 0$}

Given points $\underline{x}_1$ and $\underline{x}_2$ of $\underline{X}_0(N)$, reduction of homomorphisms induces injections 
\[
\Hom_{\Lambda/\pi^{n+1}}(\underline{x}_1, \underline{x}_2) \hookrightarrow \Hom_{\Lambda/\pi^{n}}(\underline{x}_1, \underline{x}_2), \qquad n\geq 1
\]
and we have 
\[
\Hom_{L^{\mathrm{ur}}}(x_1, x_2)=\Hom_W(\underline{x}_1, \underline{x}_2)=\bigcap_{n\geq 1} \Hom_{\Lambda/\pi^{n}}(\underline{x}_1, \underline{x}_2).
\]
We define 
\[
\Hom_{\Lambda/\pi^{n}}^{\mathrm{new}}(\underline{x}_1, \underline{x}_2):= \Hom_{\Lambda/\pi^{n}}(\underline{x}_1, \underline{x}_2) \setminus \Hom_W(\underline{x}_1, \underline{x}_2).
\]
The integer $2ur_{\cA}(m)$ is the multiplicity with which $P_{\fn}$ occurs in $T_m P_{\fn}^{\fa}$, i.e., 
\[
\vert \Hom_{\Lambda}(\underline{P}_{\fn}^{\fa}, \underline{P}_{\fn})_m \vert = 2ur_{\cA}(m).
\]

\begin{theorem}\label{nonarchht}
Assume that $(m, N)=1$. Then
\begin{align*}
\langle Z, T_m \bar{Z}_{\cA}\rangle_{\mathrm{fin}}^{\GS} &
= - u^2 \frac{\vert S\vert (4\vert D\vert m)^{k-t-1}}{D^t \binom{2k-2}{k-t-1}} \sum_{0<n< \frac{m\vert D\vert}{N}} \sigma_{\cA}(n) r_{\cA, \chi}(m\vert D\vert - nN) P_{k,t}\left( 1- \frac{2nN}{m\vert D\vert} \right) \\
& + \frac{\vert S\vert (4\vert D\vert m)^{k-t-1}}{\binom{2k-2}{k-t-1}} hu r_{\cA, \chi}(m) \log \frac{N}{m}.
\end{align*}
\end{theorem}

\begin{proof}
We use the same notations as in the proof of Proposition \ref{prop:nonar1}. As in the proof of Proposition \ref{prop:nonar1}, we have 
\begin{equation}\label{yes}
(\underline{x}, T_m \overline{\underline{x}}^{\fa})_v  =  \vert S\vert m^{k-t-1} \sum_{g\in \Hom_{\Lambda}(\underline{y}^\sigma, \underline{y}^\sigma_g)_{m_0}} (\underline{y} \cdot T_{q^s} \underline{y}_g^{\sigma}) (\bar{\epsilon}(X_{hg\sqrt{D}g^{-1}h^{-1}}^{\otimes k-t-1} \otimes X_{\overline{hg\phi_\fa}}^{\otimes 2t}), \epsilon (X_{\sqrt{D}}^{k-t-1}\otimes X_1^{\otimes 2t})),
\end{equation}
where $h$ is an arbitrary element of $\Hom_{\Lambda/\pi^n}(\underline{y}_g^{\sigma}, \underline{y})_{q^s}$ for some $n\geq 1$. Note that this element $h$ can be the reduction of an element of $\Hom_{\Lambda}(\underline{y}_g^{\sigma}, \underline{y})_{q^s}$ in the present case since $r_{\cA}(m)\neq 0$.

We begin by assuming that $v$ sits over a prime $q$ that is non-split in $K$. In particular, $q$ does not divide $N$.
By \cite[III (8.5)]{GZ}, we have 
\[
(\underline{y} \cdot T_{q^s} \underline{y}_g^{\sigma}) = \frac{1}{2} \sum_{n\geq 1} \sum_{h'\in \Hom^{\mathrm{new}}_{\Lambda/\pi^n}(\underline{y}_g^{\sigma}, \underline{y})_{q^s}} 1 + \frac{v_q}{2} \sum_{h'\in \Hom_\Lambda(\underline{y}_g^{\sigma}, \underline{y})_{q^s}} 1,
\]
with $v_q=s$ or $s/2$ depending on whether $q$ is ramified or inert in $K$.
Injecting this into the above expression, we obtain
\begin{align*}
(\underline{x}, T_m \overline{\underline{x}}^{\fa})_v & = 
\frac{\vert S \vert m^{k-t-1}}{2} \sum_{n\geq 1} \sum_{g\in \Hom^{\mathrm{new}}_{\Lambda/\pi^n}(\underline{y}^\sigma, \underline{y})_{m}} (\bar{\epsilon}(X_{g\sqrt{D}g^{-1}}^{\otimes k-t-1} \otimes X_{\overline{g\phi_\fa}}^{\otimes 2t}), \epsilon (X_{\sqrt{D}}^{k-t-1}\otimes X_1^{\otimes 2t})) \\
& \qquad + v_q \frac{\vert S \vert m^{k-t-1}}{2} \sum_{g\in \Hom_{\Lambda}(\underline{y}^{\sigma}, \underline{y})_{m}} (\bar{\epsilon}(X_{g\sqrt{D}g^{-1}}^{\otimes k-t-1} \otimes X_{\overline{g\phi_\fa}}^{\otimes 2t}), \epsilon (X_{\sqrt{D}}^{k-t-1}\otimes X_1^{\otimes 2t})),
\end{align*}
where the equality holds because \eqref{specialfiber} is true for arbitrary $h$. 

The quaternionic expression of Proposition \ref{prop:nonar2} continues to hold for the first term if we impose the additional constraint that $b_-\neq 0$ (always satisfied when $r_{\cA}(m)=0$):
\begin{multline*} 
\frac{\vert S\vert m^{k-t-1}}{2} \sum_{n\geq 1} \sum_{g\in \Hom^{\mathrm{new}}_{\Lambda/\pi^n}(\underline{y}^\sigma, \underline{y})_{m}} (\bar{\epsilon}(X_{g\sqrt{D}g^{-1}}^{\otimes k-t-1} \otimes X_{\overline{g\phi_\fa}}^{\otimes 2t}), \epsilon (X_{\sqrt{D}}^{k-t-1}\otimes X_1^{\otimes 2t})) \\
=\vert S\vert m^{k-t-1} \frac{(4D)^{k-t-1}}{\binom{2k-2}{k-t-1}} \sum_{\substack{\substack{b=\alpha+b_- \in R\fa \slash \pm 1 \\ N(b)=mN(\fa) } \\ b_-\neq 0}} \bar{\alpha}^{2t} U(b_-) P_{k,t}\left( 1 - \frac{2N(b_-)}{N(b)} \right).
\end{multline*}
It remains to handle the second term above, namely
\[
v_q \frac{\vert S\vert m^{k-t-1}}{2}\sum_{g\in \Hom_{\Lambda}(\underline{y}^{\sigma}, \underline{y})_{m}} (\bar{\epsilon}(X_{g\sqrt{D}g^{-1}}^{\otimes k-t-1} \otimes X_{\overline{g\phi_\fa}}^{\otimes 2t}), \epsilon (X_{\sqrt{D}}^{k-t-1}\otimes X_1^{\otimes 2t})).
\]
The result \cite[Prop. 5.3]{pGZshnidman} continues to hold in this case with $b_-=0$ since $g\phi_\fa = b = \alpha+b_- \in \End_\Lambda(\underline{y})=\oh_K$. Hence, we obtain
\[
(\bar{\epsilon}(X_{g\sqrt{D}g^{-1}}^{\otimes k-t-1} \otimes X_{\overline{g\phi_\fa}}^{\otimes 2t}), \epsilon (X_{\sqrt{D}}^{k-t-1}\otimes X_1^{\otimes 2t}))=\frac{(4D)^{k-t-1}}{\binom{2k-2}{k-t-1}} \bar{\alpha}^{2t} P_{k,t}(1).
\]
But $P_{k, t}(1)=1$, and we deduce that 
\begin{multline*}
v_q \frac{\vert S\vert m^{k-t-1}}{2} \sum_{g\in \Hom_{\Lambda}(\underline{y}^{\sigma}, \underline{y})_{m}} (\bar{\epsilon}(X_{g\sqrt{D}g^{-1}}^{\otimes k-t-1} \otimes X_{\overline{g\phi_\fa}}^{\otimes 2t}), \epsilon (X_{\sqrt{D}}^{k-t-1}\otimes X_1^{\otimes 2t})) \\
=v_q \frac{\vert S\vert (4mD)^{k-t-1}}{2\binom{2k-2}{k-t-1}}
\sum_{\substack{g\in \Hom_{\Lambda}(\underline{y}^{\sigma}, \underline{y})_{m} \\ g\phi_\fa = \alpha \in \End_{\Lambda}(\underline{y})}} \bar{\alpha}^{2t}.
\end{multline*}
Recall that $y=P_{\fn}$ is represented by $A=\C/\oh_K$, $y^\sigma=P_{\fn}^{\fa}$ is represented by $A^\fa=\C/\fa^{-1}$, and with these identifications the isogeny $\phi_\fa$ is the quotient isogeny $\C/\oh_K\lra \C/\fa^{-1}$. In particular, $g\phi_\fa\in \End(\C/\oh_K)$ is given by multiplication by some $\alpha\in\oh_K$ that must satisfy $\alpha a\in \oh_K$ for all $a\in \fa^{-1}$. In other words, we must have $\alpha \in \fa$. The norm of $\alpha$ is $\deg(g)\deg(\phi_\fa)=mN(\fa)$. We have shown that the assignment $g\mapsto g\phi_\fa$ gives a map
\[
\Hom_{\Lambda}(\underline{y}^{\sigma}, \underline{y})_{m} \lra \{ \alpha\in \fa \colon N_{K/\Q}(\alpha)=mN(\fa) \}.
\]
This map is injective and $\Hom_{\Lambda}(\underline{y}^{\sigma}, \underline{y})_{m}$ has size $2ur_{\cA}(m)$. On the other hand, the integral ideal $\alpha \fa^{-1}\in \cA^{-1}$ has norm $m$ and thus the right hand set has size $\vert \oh_K^\times \vert r_{\cA^{-1}}(m)=2ur_{\cA}(m)$. We deduce that the map is a bijection. We thus obtain 
\begin{align*}
v_q \frac{\vert S\vert m^{k-t-1}}{2} \sum_{g\in \Hom_{\Lambda}(\underline{y}^{\sigma}, \underline{y})_{m}} & (\bar{\epsilon}(X_{g\sqrt{D}g^{-1}}^{\otimes k-t-1} \otimes X_{\overline{g\phi_\fa}}^{\otimes 2t}), \epsilon (X_{\sqrt{D}}^{k-t-1}\otimes X_1^{\otimes 2t})) \\
& =v_q \frac{\vert S\vert (4mD)^{k-t-1}}{2\binom{2k-2}{k-t-1}}
\sum_{\substack{\alpha\in \fa \\ N_{K/\Q}(\alpha)= m N(\fa)}} \bar{\alpha}^{2t} \\
& =v_q \frac{\vert S\vert (4mD)^{k-t-1}}{2\binom{2k-2}{k-t-1}}
\sum_{\substack{\alpha\in \fa \\ N_{K/\Q}(\alpha)= m N(\fa)}} \bar{\chi}(\alpha) \\
& =v_q \frac{\vert S\vert (4mD)^{k-t-1}}{2\binom{2k-2}{k-t-1}}
\bar{\chi}(\fa^{-1})^{-1}\sum_{\substack{\alpha\in \fa \\ N_{K/\Q}(\alpha)= m N(\fa)}} \bar{\chi}(\alpha\fa^{-1}) \\
& =v_q \frac{\vert S\vert (4mD)^{k-t-1}}{\binom{2k-2}{k-t-1}}
\chi(\bar{\fa}) u \sum_{\substack{\fc\in \cA^{-1} \\ N(\fc)= m}} \bar{\chi}(\fc) \\
& =v_q \frac{\vert S\vert (4mD)^{k-t-1}}{\binom{2k-2}{k-t-1}}
\chi(\bar{\fa}) u r_{\bar{\cA}, \bar{\chi}}(m) \\
& =v_q \frac{\vert S\vert (4mD)^{k-t-1}}{\binom{2k-2}{k-t-1}}
\chi(\bar{\fa}) u r_{\cA, \chi}(m).
\end{align*}
To summarize, when $v$ sits over a non-split prime $q$, we have established that 
\begin{align*}
(\underline{x}, T_m \overline{\underline{x}}^{\fa})_v & = 
\vert S\vert m^{k-t-1} \frac{(4D)^{k-t-1}}{\binom{2k-2}{k-t-1}} \sum_{\substack{\substack{b=\alpha+b_- \in R\fa \slash \pm 1 \\ N(b)=mN(\fa) } \\ b_-\neq 0}} \bar{\alpha}^{2t} U(b_-) P_{k, t}\left( 1 - \frac{2N(b_-)}{N(b)} \right) \\
& \qquad + v_q \frac{\vert S\vert (4mD)^{k-t-1}}{\binom{2k-2}{k-t-1}}
\chi(\bar{\fa}) u r_{\cA, \chi}(m).
\end{align*}

Next, we deal with the case when $v$ sits above a prime $q$ that splits in $K$, say $q\oh_K=\fq\bar{\fq}$. Our starting point is \eqref{yes}. In this case, by \cite[III (8.5) \ 3) and III (8.6)]{GZ}, we have 
\[
(\underline{y}\cdot T_{q^s}\underline{y}^\sigma_g)=\frac{v_q}{2} \sum_{h'\in \Hom_{\Lambda}(\underline{y}_g^{\sigma}, \underline{y})_{q^s}} 1,
\]
where 
\[
v_q=
\begin{cases}
k_\fq & \text{ if } v\nmid N, \text{ where } k_\fq\geq 0 \text{ and } k_\fq + k_{\bar{\fq}}=s \\
0 & \text{ if } v\mid \fn \\
-\ord_q(N) & \text{ if } v\mid \bar{\fn}.
\end{cases}
\]
Injecting this expression into \eqref{yes} yields
\begin{multline*}
(\underline{x}, T_m \overline{\underline{x}}^{\fa})_v  = v_q \frac{\vert S\vert m^{k-t-1}}{2} \sum_{g\in \Hom_{\Lambda}(\underline{y}^\sigma, \underline{y})_{m}} (\bar{\epsilon}(X_{g\sqrt{D}g^{-1}}^{\otimes k-t-1} \otimes X_{\overline{g\phi_\fa}}^{\otimes 2t}), \epsilon (X_{\sqrt{D}}^{k-t-1}\otimes X_1^{\otimes 2t})) \\
=v_q \frac{\vert S\vert (4mD)^{k-t-1}}{\binom{2k-2}{k-t-1}}
\chi(\bar{\fa}) u r_{\cA, \chi}(m).
\end{multline*}
(This exact expression was calculated above.)
Putting everything together, we deduce that 
\begin{align*}
\langle Z, & T_m \bar{Z}_{\cA}\rangle_{\mathrm{fin}}^{\GS} \\ & = - u^2 \frac{\vert S\vert (4\vert D\vert m)^{k-t-1}}{D^t \binom{2k-2}{k-t-1}} \sum_{0<n< \frac{m\vert D\vert}{N}} \sigma_{\cA}(n) r_{\cA, \chi}(m\vert D\vert - nN) P_{k,t}\left( 1- \frac{2nN}{m\vert D\vert} \right) \\
& \qquad + (-1)^{k+t}  \frac{\vert S\vert (4mD)^{k-t-1}}{\binom{2k-2}{k-t-1}}
u r_{\cA, \chi}(m) \sum_{q} v_q \sum_{v\mid q}  \log N(v) \\ 
& = - u^2 \frac{\vert S\vert (4\vert D\vert m)^{k-t-1}}{D^t \binom{2k-2}{k-t-1}} \sum_{0<n< \frac{m\vert D\vert}{N}} \sigma_{\cA}(n) r_{\cA, \chi}(m\vert D\vert - nN) P_{k,t}\left( 1- \frac{2nN}{m\vert D\vert} \right) \\
& \qquad + (-1)^{k+t}  \frac{\vert S\vert (4mD)^{k-t-1}}{\binom{2k-2}{k-t-1}}
u r_{\cA, \chi}(m) \left( \sum_{q \text{ split}} h\ord_q(m/N) \log q + \sum_{q \text{ non-split}} h \ord_q(m) \log q  \right) \\ 
& = - u^2 \frac{\vert S\vert (4\vert D\vert m)^{k-t-1}}{D^t \binom{2k-2}{k-t-1}} \sum_{0<n< \frac{m\vert D\vert}{N}} \sigma_{\cA}(n) r_{\cA, \chi}(m\vert D\vert - nN) P_{k,t}\left( 1- \frac{2nN}{m\vert D\vert} \right) \\
& \qquad +  \frac{\vert S\vert (4m\vert D\vert)^{k-t-1}}{\binom{2k-2}{k-t-1}}
h u r_{\cA, \chi}(m) \log(N/m). 
\end{align*}
\end{proof}

\section{Final formula}\label{s:GZ}

We recall from \eqref{def:Za} and \eqref{def:Zbara} that 
\[
Z_{\cA}=\chi(\fa)^{-1} Z_{\fn}^{\fa}  \qquad \text{ and } \qquad \bar{Z}_{\cA}=\bar{\chi}(\fa)^{-1} \bar{Z}_{\fn}^{\fa}.
\] 
Define 
\[
\Delta:=\frac{1}{\sqrt{\deg(\pi_{M,N})}} (Z+\bar{Z}).
\]
If $\cA=[\fa]\in \Cl_K$, then define
\[
\Delta_{\cA, \chi}:=\frac{1}{\sqrt{\deg(\pi_{M,N})}}(Z_{\cA}+\bar{Z}_{\bar{\cA}}). 
\]
Recall from \eqref{def:delchi} the cycle
\[
\Delta_\chi := \sum_{\cA\in \Cl_K} \Delta_
{\cA, \chi} \in \CH^{k+t}(X)_{0, K(\sqrt{\deg(\pi_{M,N})})(\chi)}.
\]

Recall that $g_\cA=\sum_{n\geq 1} a_m(\cA)q^n\in S_{2k}^{\new}(\Gamma_0(N))$ represents the linear functional on $S_{2k}^{\new}(\Gamma_0(N))$ given by 
\[
f \longmapsto \frac{(2k-2)! \sqrt{\vert D\vert}D^t}{2^{4k-1}\pi^{2k}} L'_\cA(f,\chi,k+t).
\]

\begin{theorem}\label{THM}
Assume that $0<t \leq k-1$. For all $(m,N)=1$, we have 
\[
\langle \Delta_{\chi}, T_m \Delta_{\chi}\rangle^{\GS}= 2 h u^2 \frac{(4\vert D\vert)^{k-t-1}}{D^t \binom{2k-2}{k-t-1}} \sum_{\cA} a_m(\cA).
\]
\end{theorem}

\begin{proof}
By Theorem \ref{fourcoeff}, for any $(m,N)=1$, we have
\[
a_m(\cA)=a_m^{\mathrm{fin}}(\cA)+a_m^{\infty}(\cA),
\]
where
\begin{align}\label{afin}
a_m^{\mathrm{fin}}(\cA) & = m^{k-t-1} \left[  \frac{h}{u}D^t r_{\bar{\cA}, \chi}(m)\log \frac{N}{m} -\sum_{0<n\leq \frac{m\vert D\vert}{N}} \sigma'_{\cA}(n) r_{\bar{\cA}, \chi}(m\vert D\vert-nN) P_{k,t}\left( 1-\frac{2nN}{m\vert D\vert}\right) \right] \\ \label{ainf}
a_m^{\infty}(\cA) & = m^{k-t-1} \left[\frac{h}{u}D^t r_{\bar{\cA}, \chi}(m)\left( \frac{\Gamma'}{\Gamma}(k+t) + \frac{\Gamma'}{\Gamma}(k-t) + \log \left(\frac{\vert D\vert}{4\pi^2}\right) +2 \frac{L'}{L}(1, \epsilon_K) \right) \right. \\
& \qquad\qquad\qquad\qquad\qquad\qquad\qquad\qquad\qquad \left. - \sum_{n=1}^\infty \sigma_{\cA}(n) r_{\bar{\cA}, \chi}(m\vert D\vert + nN) Q_{k,t}\left( 1+\frac{2 n N}{m\vert D\vert} \right)\right]. \\
\end{align}

By Theorem \ref{nonarchht}, using also \eqref{eq: barZ nonarch}, we have 
\begin{equation}
\langle \bar{Z}, T_m Z_{\cA}\rangle_{\mathrm{fin}}^{\GS}=
\langle Z, T_m \bar{Z}_{\bar{\cA}}\rangle_{\mathrm{fin}}^{\GS} =u^2 \frac{\vert S\vert (4\vert D\vert)^{k-t-1}}{D^t \binom{2k-2}{k-t-1}} a_m^{\mathrm{fin}}(\cA). 
\end{equation}
By Theorem \ref{archht}, we have 
\begin{equation}\label{eq:needed}
\langle Z, T_m Z_{\cA} \rangle_{\infty}^{\GS}=\langle \bar{Z}, T_m \bar{Z}_{\bar{\cA}} \rangle_{\infty}^{\GS}=u^2 \frac{\vert S \vert (4\vert D\vert)^{k-t-1}}{D^t \binom{2k-2}{k-t-1}} a_m^{\infty}(\bar{\cA}).
\end{equation}
Recall that $\vert S\vert=\deg(\pi_{M,N})$.
We deduce that, for all $(m,N=1)$, we have
\begin{align*}
h^{-1}\langle \Delta_{\chi}, T_m \Delta_{\chi}\rangle^{\GS} & =\langle \Delta, T_m \sum_{\cA} \Delta_{\cA, \chi} \rangle^{\GS} \\
&= \langle Z+\bar{Z}, T_m\sum_{\cA} (Z_{\cA} + \bar{Z}_{\bar{\cA}}) \rangle^{\GS} \\
& =\sum_{\cA} (\langle Z, T_m \bar{Z}_{\bar{\cA}}\rangle_{\mathrm{fin}}^{\GS}+\langle \bar{Z}, T_m Z_{\cA}\rangle_{\mathrm{fin}}^{\GS} +\langle Z, T_m Z_{\cA}\rangle_{\infty}^{\GS}+ \langle \bar Z, T_m \bar{Z}_{\bar{\cA}}\rangle_{\infty}^{\GS}) \\
&= u^2 \frac{(4\vert D\vert)^{k-t-1}}{D^t \binom{2k-2}{k-t-1}} \sum_{\cA} (2a_m^{\mathrm{fin}}(\cA)+2a_m^{\infty}(\bar{\cA})) \\
&= 2u^2 \frac{(4\vert D\vert)^{k-t-1}}{D^t \binom{2k-2}{k-t-1}} \sum_{\cA} a_m(\cA).
\end{align*}
\end{proof}

Define $W$ to be the subspace of generalized Heegner cycles generated by
\[
\{ T_m \Delta_\chi \mid (m,N)=1 \}.
\]
Let $W'$ denote the quotient of $W$ modulo the null space with respect to the Gillet--Soulé height pairing on $W\times W$. 

\begin{proposition}\label{prop:heckecompat}
The module $W'$ is isomorphic as a Hecke module to a subquotient of $S_{2k}(\Gamma_0(N))$.
\end{proposition}

\begin{proof}
The argument is the same as \cite[Prop. 5.1.2]{zhang}, but requires Lemma \ref{lem:heckeq}, which we prove below.
\end{proof}

\begin{lemma}\label{lem:heckeq}
For $n$ and $m$ coprime to $N$, we have 
\[T_n T_m \Delta_\chi = \sum_{d\mid (n,m)} d^{2k-1} T_{\frac{nm}{d^2}}\Delta_\chi.\]
\end{lemma}
\begin{proof}
    We immediately reduce to the case where $m$ and $n$ are powers of a single prime $p$. Since the action of Hecke operators on $\Delta_\chi$ lies over the action of the Hecke operators on the modular curve, the desired formula follows from the case $k = 1$ and the formula $f^*\Delta_\chi = p^{2k-2}\Delta_\chi$, where $f = [p] \times 1 \colon W_{2k-2} \times A^{2t} \to W_{2k-2} \times A^{2t}$. To verify this, we consider separately, in the definition of the generalized Heegner cycle (see Section \ref{subsec:gen Heeg cycles}), the $k-t-1$ pure Kuga--Sato factors and the $2t$ ``mixed'' factors. Let $\eta$ be the projection onto $H^1(E_y)^{\otimes 2}$, so that, e.g., $\eta \Gamma_1 = \Gamma_1 - E_y \times \{0\} - \{0\} \times E_y$. On the pure Kuga--Sato factors, the pullback of $\eta\Gamma_{\sqrt{D}}$ along $[p] \times [p] \colon E_y \times E_y \to E_y \times E_y$ is $p^2\eta\Gamma_{\sqrt{D}}$ \cite[2.4.2]{zhang}. On the mixed factors the pullback of $\eta \Gamma_1$ along $[p] \times [1] \colon E_y \times A_F \to E_y \times A_F$ is $\eta \Gamma_p = p\eta \Gamma_1$. It follows that $f^*\epsilon Y = p^{2k-2t-2}p^{2t}\epsilon Y = p^{2k-2}\epsilon Y$, and hence $f^*\Delta_\chi = p^{2k-2}\Delta_\chi$, as claimed.     
    \end{proof}

Given an element $T\in W$, denote its image in $W'$ by $T'$. If $f$ is a normalized newform, we may consider the $f$-isotypic component of $\Delta'_{\chi}$, which we denote by $\Delta'_{\chi, f}$. 

\begin{theorem}\label{thm: main result}
Assume that $0<t\leq k-1$. For any normalized newform $f\in S_{2k}(\Gamma_0(N))$, we have
\[
L'(f,\chi,k+t)=\frac{4^{k+t}\pi^{2k} (f,f)}{(k-t-1)!(k+t-1)!  hu^2 \sqrt{\vert D\vert} \vert D\vert^{k-t-1}}  \langle \Delta'_{\chi, f}, \Delta'_{\chi, f}\rangle^{\GS}.
\]
\end{theorem}

\begin{proof}
We have 
\[
L'(f,\chi,k+t)=\sum_{\cA} L_{\cA}'(f,\chi,k+t)=\frac{2^{4k-1}\pi^{2k}}{(2k-2)! \sqrt{\vert D\vert} D^t} (f, \sum_{\cA} g_\cA).
\]
Extend $\{ f \}$ to a basis $\{ f_1=f, f_2, \ldots, f_d \}$ of $S_{2k}(\Gamma_0(N))$ consisting of the normalized newforms together with a basis of the space of oldforms. We may then write 
\[
\Delta'_{\chi}= \sum_{i=1}^d \Delta'_{\chi, f_i},
\]
such that $T_m \Delta'_{\chi, f_i}=a_m(f_i)\Delta'_{\chi, f_i}$ for all $(m,N)=1$. We then have 
\[
\langle \Delta_{\chi}, T_m \Delta_{\chi}\rangle^{\GS}= \langle \Delta'_{\chi}, T_m \Delta'_{\chi}\rangle^{\GS}=\sum_{i,j} \langle \Delta'_{\chi, f_i}, T_m \Delta'_{\chi, f_j}\rangle^{\GS}=\sum_{i,j} a_m(f_j) \langle \Delta'_{\chi, f_i}, \Delta'_{\chi, f_j}\rangle^{\GS}.
\]
Using Theorem \ref{THM}, we deduce that up to oldforms we have 
\[
\sum_{\cA} g_{\cA} = \frac{D^t \binom{2k-2}{k-t-1}}{2u^2 h(4\vert D\vert)^{k-t-1}} \sum_{i,j} \langle \Delta'_{\chi, f_i}, \Delta'_{\chi, f_j}\rangle^{\GS} f_j.
\]
The cycles $\Delta'_{\chi, f_j}$ have different eigenvalues with respect to the Hecke operators, which are self-dual with respect to the height pairing. Thus, $\langle \Delta'_{\chi, f_i}, \Delta'_{\chi, f_j}\rangle^{\GS}=0$ if $i\neq j$ and 
\[
L'(f,\chi,k+t)=\frac{2^{4k-1}\pi^{2k}}{(2k-2)! \sqrt{\vert D\vert} D^t} \frac{D^t \binom{2k-2}{k-t-1}}{2u^2 h(4\vert D\vert)^{k-t-1}} \overline{\langle \Delta'_{\chi, f}, \Delta'_{\chi, f}\rangle^{\GS}}.
\]
The result follows since $\langle \;, \;\rangle^{\GS}$ is Hermitian and in particular, any self-pairing is a real number.
\end{proof}

\section{Algebraicity}\label{s:algebra}

We prove generalizations of the algebraicity results of Gross--Zagier \cite[V (4.3)]{GZ} and Zhang \cite[Thm. 5.2.2]{zhang}. Define 
\begin{multline}\label{functionspe}
\mathcal{G}^m_{M, N, k, t, \chi}(z,z'):= \frac{\vert S\vert (4\vert D \vert)^{k-t-1}}{D^t \binom{2k-2}{k-t-1}}\bigg[ \sum_{\substack{\gamma\in R_N^m/\pm 1 \\ \gamma z'\neq z}} g_{k,t}(z, \gamma z') \alpha(\gamma, z, z')^{2t}  \\
  + u D^t r_{\cA_1^{-1}\cA_2, \chi}(m) \lim\limits_{w\to z} (g_{k,t}(z, w)- \log \vert 2\pi i \eta^4(z)(w-z)\vert_v) \bigg],
\end{multline}
where we recall that for $\gamma=\left( \begin{smallmatrix} a&b\\c&d\end{smallmatrix}\right)\in R_N^m$, we have $\alpha(\gamma, z, z')=c\bar{z}z'+d\bar{z}-az'-b$, and
\begin{equation*}
g_{k,t}(z,z'):=-Q_{k,t}\left(1+\frac{\vert z-z'\vert^2}{2yy'}\right),
\end{equation*}
with $Q_{k,t}$ the Jacobi function of the second type \eqref{Qkt1}. 

Let $\cA\in \Cl_K$ with corresponding Galois element $\sigma\in \Gal(H/K)$. Let $\tau_0=\frac{-1+\sqrt{D}}{2}\in \cH$ such that $\Gamma_0(N)\tau_0=P_{\fn}$, as in Section \ref{s:hhodg}. 
Define the quantity 
\[
\gamma^m_{M,N,k,t,\chi}(\mathcal{A}):=\sum_{\rho\in \Gal(H/K)} \mathcal{G}^m_{M, N, k, t, \chi}(\tau_0^{\rho},\tau_{0}^{\rho \sigma})=\sum_{\substack{\cA_1, \cA_2\in \Cl_K \\ \cA_1^{-1}\cA_2=\cA }} \mathcal{G}^m_{M, N, k, t, \chi}(\tau_{\fa_1},\tau_{\fa_2}).
\]
By Theorem \ref{main.} and \eqref{eq:needed}, we have 
\[
\langle Z, T_m Z_{\cA} \rangle^{\GS}_{\infty}=m^{k-t-1}\gamma^m_{M,N,k,t,\chi}(\mathcal{A})=u^2 \frac{\vert S \vert (4\vert D\vert)^{k-t-1}}{D^t \binom{2k-2}{k-t-1}} a_m^{\infty}(\bar{\cA}),
\]
where $a_m^{\infty}(\bar{\cA})$ is given by \eqref{ainf}. We deduce that 
\begin{multline}\label{ambar0}
a_{m}(\bar\cA)= \frac{D^t \binom{2k-2}{k-t-1}}{u^2\vert S \vert (4\vert D\vert)^{k-t-1}}m^{k-t-1}\gamma^m_{M,N,k,t,\chi}(\mathcal{A}) \\
+m^{k-t-1} \left[  \frac{h}{u}D^t r_{\cA, \chi}(m)\log \frac{N}{m} -\sum_{0<n\leq \frac{m\vert D\vert}{N}} \sigma'_{\bar{\cA}}(n) r_{\cA, \chi}(m\vert D\vert-nN) P_{k,t}\left( 1-\frac{2nN}{m\vert D\vert}\right) \right].
\end{multline}

Let $\lambda=\{\lambda_m\}_{\geq 1}$ be a relation for $S_{2k}(\Gamma_0(N))$, in the sense of \cite[p. 316]{GZ}, so that:
\begin{itemize}
    \item[i)] $\lambda_m\in \Z$, $\lambda_m=0$ for all but finitely many $m$;
    \item[ii)] $\sum_{m\geq 1} \lambda_m a_m=0$ for all $\sum_{m\geq 1} a_m q^m \in S_{2k}(\Gamma_0(N))$;
    \item[iii)] $\lambda_m=0$ for $(m,N)\neq 1$.
\end{itemize}
Define 
\[
\gamma_{M,N,k,t,\chi, \lambda}(\cA):=\sum_{m\geq 1} m^{k-t-1} \lambda_m \gamma_{M,N,k,t,\chi}^m(\cA)=\sum_{\rho\in \Gal(H/K)} \mathcal{G}_{M, N, k, t, \chi, \lambda}(\tau_0^{\rho},\tau_{0}^{\rho \sigma}).
\]
Then \eqref{ambar0} implies the equality:
\begin{theorem}\label{thm:algebraicity}
For all $0<t\leq k-1$, $\lambda$ a relation for $S_{2k}(\Gamma_0(N))$, and $\cA$ an ideal class of $K$, we have
\begin{multline}
\gamma_{M,N,k,t,\chi, \lambda}(\cA) =u^2 \frac{\vert S \vert (4\vert D\vert)^{k-t-1}}{D^t \binom{2k-2}{k-t-1}}\sum_{m\geq 1}m^{k-t-1} \lambda_m \bigg[ -\frac{h}{u}D^t r_{\cA, \chi}(m)\log \frac{N}{m} \\ + \sum_{0<n\leq \frac{m\vert D\vert}{N}} \sigma'_{\bar{\cA}}(n) r_{\cA, \chi}(m\vert D\vert-nN) P_{k,t}\left( 1-\frac{2nN}{m\vert D\vert}\right)
 \bigg].
\end{multline}
\end{theorem}

The quantity $\sigma'_{\bar{\cA}}(n)$ appearing in the sum on the right hand side is a non-negative even integral multiple of the logarithm of a prime (by the remark following \cite[IV (4.6)]{GZ}). Hence, Theorem \ref{thm:algebraicity} expresses $\gamma_{M,N,k,t,\chi, \lambda}(\cA)$ as a $\Q(\chi)$-rational linear combination of logarithms of primes, 
generalizing \cite[V (4.3)]{GZ} to the case $t>0$. 

We also obtain the following result in support of a generalization of the algebraicity conjecture of Gross--Zagier \cite[V (4.4)]{GZ}:

\begin{theorem}\label{thm:alg2}
    Let $0<t\leq k-1$, $\lambda$ be a relation for $S_{2k}(\Gamma_0(N))$, and $\sigma\in \Gal(H/K)$. 
    Assume that the height pairing on generalized Heegner cycles is non-degenerate. Then there exists a $K(\chi)$-rational number $c$ and an element $\delta\in H\otimes K(\chi)$ such that 
    \[
    \mathcal{G}_{M, N, k, t, \chi, \lambda}(\tau_0^{\rho},\tau_{0}^{\rho \sigma})=c\log \vert \delta^\rho \vert_v,
    \]
    for all $\rho \in \Gal(H/K)$, where $v \colon H\hookrightarrow \C$ is the embedding fixed in Section \ref{s:hhodg}.
\end{theorem}

\begin{proof}
    Let $\cA$ be the ideal class corresponding to $\sigma$.
    By Theorem \ref{fullheight} and \eqref{galinf}, we have
    \begin{align*}
    \mathcal{G}_{M, N, k, t, \chi, \lambda}(\tau_0^{\rho},\tau_{0}^{\rho \sigma}) & =\sum_{m\geq 1} m^{k-t-1} \lambda_m \mathcal{G}^m_{M, N, k, t, \chi}(\tau_0^{\rho},\tau_{0}^{\rho \sigma}) \\
    & =\sum_{m\geq 1} \lambda_m \left\langle  Z_{\cA_1}, T_m Z_{\cA_2} \right\rangle_v \\
    & =\sum_{m\geq 1} \lambda_m \left\langle  Z, T_m Z_{\cA} \right\rangle_w \\
    & =\left\langle \sum_{m\geq 1} \lambda_m T_m Z, Z_{\cA} \right\rangle_w,
    \end{align*}
    where $\cA_1$ is the ideal class corresponding to $\rho\in \Gal(H/K)$, $\cA_2=\cA_1\cA$, and $w$ is the infinite place corresponding to $\rho$. Thus, the individual terms $\mathcal{G}_{M, N, k, t, \chi, \lambda}(\tau_0^{\rho},\tau_{0}^{\rho \sigma})$ in the definition of $\gamma_{M,N,k,t,\chi, \lambda}(\cA)$ are the local height pairings at archimedean places of the cycles $\sum_{m\geq 1} \lambda_m T_m Z$ and $Z_{\cA}$. Using the non-degeneracy of the height pairing and Proposition \ref{prop:heckecompat}, the fact that $\lambda$ is a relation for $S_{2k}(\Gamma_0(N))$ implies that $\sum_{m\geq 1} \lambda_m T_m Z=0 \in \CH^{k+t}(X)_{0, K(\chi)}$. Thus, there exist subvarieties $W_i$ of $X$ and functions $f_i$ in $W_i$ defined over $H$ such that 
    \[
    \sum_{m\geq 1} \lambda_m T_m Z=c\sum_i \mathrm{div}(f_i \vert W_i),
    \]
    for some $c\in K(\chi)$. It follows that 
    \begin{align*}
    \sum_{m\geq 1} \lambda_m \left\langle  Z, T_m Z_{\cA} \right\rangle_w & =c \sum_i \log\vert f_i(Z_{\cA}\cdot W_i)\vert w \\
    &=c\log\bigg\vert \prod_i f_i(Z_{\cA}\cdot W_i)\bigg\vert_w \\
    &=c\log\bigg\vert \bigg(\prod_i f_i(Z_{\cA}\cdot W_i)\bigg)^{\rho}\bigg\vert_v.
    \end{align*}
\end{proof}

Theorem \ref{thm:alg2} generalizes \cite[Thm. 5.2.2]{zhang} to the case $t>0$.

\section{Corollary}\label{s:coro}

\begin{proof}[Proof of Corollary \ref{cor: combine with Yara}]
    Fix a prime $\ell \nmid 2(2k-1)!N\varphi(N)$ split in $K$.
    Let $\chi':=\chi \cdot \mathrm{Nm}^{-t}$, a Hecke character of infinity type $(t,-t)$.
    Consider the $\ell$-adic $\Gal_K$-representation $V_{f,\chi',\ell}=V_{f,\ell}(k)\otimes \chi_\ell'=V_{f,\ell}(k)\otimes \chi_\ell(t)$.
    The Rankin--Selberg $L$-function $L(f,\chi,s)$ has global root number $-1$. 
    
    Let $F/\Q_\ell$ be a finite extension containing the Fourier coefficients of $F$ and the values of $\chi'$. Consider the generalized Kuga--Sato variety $X':=W_{2k-2}(\Gamma_1(N))\times A^{2k-2}$, where $W_{2k-2}(\Gamma_1(N))$ is the Kuga--Sato $(2k-1)$-fold over the modular curve $X_1(N)$ with $\Gamma_1(N)$-level structure, and $A$ is a CM elliptic curve as in \cite[\S 4.1]{CastellaHsieh}. It contains the generalized Heegner cycle
    \[
    \Upsilon:=\epsilon_W\otimes \epsilon_A((\Gamma_{\id}^T)^{2k-2})\in \CH^{2k-1}(X'_{H})_{0},
    \]
    where the projectors $\epsilon_W$ and $\epsilon_A$ are the ones defined in \cite[(2.1.2), (1.4.4)]{bdp1}. Define a class $z_{f,\chi'}\in H_f^1(K, V_{f,\chi', \ell})$ by 
    \[
    z_{f,\chi'}=\mathrm{cor}_{H/K}(z_{f,\chi', 1}),
    \]
    where $z_{f,\chi', 1}\in H^1(H, V_{f,\chi', \ell})$ is the $\chi_\ell'$-component of the image of $\AJ_{X'}^{2k-1}(\Upsilon)$ in 
    \[
    H^1(H, V_{f,\ell}(k)\otimes \sym^{2k-2}H^1_{\et}(\bar{A}, \Q_\ell(1))).
    \] 
    By \cite[Thm. B]{CastellaHsieh}, we have 
    \begin{equation}\label{ThmB}
        z_{f,\chi'}\neq 0 \implies H^1_f(K, V_{f,\chi',\ell})=F\cdot z_{f,\chi'}.
    \end{equation}
    Since we assume the injectivity of $\AJ^{k+t}_X$, to prove Corollary \ref{cor: combine with Yara} it suffices to prove that
    \[
    \ord_{s=1} L(f,\chi,s)=1 \implies z_{f,\chi'}\neq 0.
    \]
    
    Let $\tilde{X}:=W_{2k-2}(\Gamma_1(N))\times A^{2t}$ and define the generalized Heegner cycles 
    \[
    \Upsilon^t_{\cA}=\epsilon_W\otimes \epsilon_A((\Gamma_{\sqrt{D}})^{k-t-1}\times (\Gamma_{\phi_{\fa}}^T)^{2t})\in \CH^{k+t}(\tilde{X}_H)_{0}.
    \] 
    Define  
    \[
    z^t_{f,\chi'}=\mathrm{cor}_{H/K}(z^t_{f,\chi',1})\in H^1(K, V_{f,\chi',\ell}),
    \]
    where $z^t_{f,\chi', 1}\in H^1(H, V_{f,\chi', \ell})$ is the $\chi_\ell'$-component of the image of of $\AJ_{\tilde{X}}^{k+t}(\Upsilon^t_{[\oh_K]})$ in 
    \[
    H^1(H, V_{f,\ell}(k)\otimes \sym^{2t}H^1_{\et}(\bar{A}, \Q_\ell(1))).
    \] 
    By \cite[Prop. 4.1.1]{bdp3}, there exists an algebraic correspondence mapping $\Upsilon$ to an integral multiple of $\Upsilon^t_{[\oh]_K}$. Thus, in order to prove that $z_{f,\chi'}\neq 0$, it suffices to prove that $z^t_{f,\chi'}\neq 0$. 

    Recall the generalized Kuga--Sato variety $X=W_{2k-2}\times A^{2t}$ studied in this paper, along with the generalized Heegner cycles $Z_{\cA}, \bar{Z}_{\cA}\in \CH^{k+t}(X_H)_{0, K(\chi)}$. 
    Define $\tilde{z}^t_{f,\chi'}\in H^1(K, V_{f, \chi',\ell})$ to be the $\chi'$-component of
    \[
    \mathrm{cor}_{H/K}(\AJ_{X,f}^{k+t}(Z))\in H^1(K, V_{f,\ell}(k)\otimes \kappa_{2t} H_{\et}^{2t}(\bar{A}, \Q_\ell(t))),
    \]
    where $\AJ^{k+t}_{X,f}$ denotes the composition 
    \[
    \CH^{k+t}(X)_0\overset{\AJ^{k+t}_X}{\lra} H^1(H, H_{\et}^{2k-1}(\bar{X}, \Q_\ell(k))\otimes \kappa_{2t} H_{\et}^{2t}(\bar{A}, \Q_\ell(t))) \overset{\epsilon_f}{\lra} H^1(K, V_{f,\ell}(k)\otimes \kappa_{2t} H_{\et}^{2t}(\bar{A}, \Q_\ell(t))).
    \]
    
    Note that $\kappa_{2t}H^{2t}(\bar{A}, \Q_\ell(t))$ is a quotient of $\epsilon_A H^{2t}(\bar{A}, \Q_\ell(t))$. In order to prove that $z^t_{f,\chi'}\neq 0$, it therefore suffices to prove that $\tilde{z}^t_{f,\chi'}\neq 0$ (the different level structures have no effect since we are projecting the cohomology classes to the cohomology of the representation $V_{f,\ell}$, which can be cut out from the cohomology of $W_{2k-2}$ and/or $W_{2k-2}(\Gamma_1(N))$). 
    
    We have 
    \[
    \tilde{z}^t_{f,\chi'}=\sum_{\cA\in \Cl_K} \AJ_{X,f}^{k+t}(Z_{\cA})=\AJ_{X,f}^{k+t}\bigg( \sum_{\cA\in \Cl_K} Z_{\cA}\bigg).
    \]
    Let $\tau\in\Gal(H/,\Q)$ be a lift of a generator of $\Gal(K/\Q)$. Let $(-1)^k \nu_f$ be the sign of the functional equation of $L(f,s)$. By \cite[Lem. 4.8]{pGZshnidman}, we have 
    \[
    \tau(\tilde{z}^t_{f,\chi'})=(-1)^{k-t-1}\nu_f \chi'(\fn) \AJ_{X,f}^{k+t}\bigg( \sum_{\cA\in \Cl_K} \bar{Z}_{\cA}\bigg).
    \]
    In view of \eqref{def:delchi}, we thus have 
    \begin{equation}\label{zchi}
    \tilde{z}^t_{f,\chi'}+(-1)^{k-t-1}\nu_f \chi'(\bar{\fn})\tau(\tilde{z}^t_{f,\chi'})=\sqrt{\deg(\pi_{M,N})} \AJ^{k+t}_X(\Delta_{\chi, f}).
    \end{equation}
    
    Using Theorem \ref{thm:intro}, we see that
    \begin{align*}
        \ord_{s=k+t} L(f,\chi,s)=1 & \overset{\eqref{thm:intro}}{\implies} \langle\Delta_{\chi,f}, \Delta_{\chi,f}\rangle^{\GS}\neq 0 \\
        & \implies \Delta_{\chi, f}\neq 0 \\
        & \overset{\AJ \text{ inj.}}{\implies} \AJ_X^{k+t}(\Delta_{\chi, f})\neq 0 \\
        & \overset{\eqref{zchi}}{\implies} \tilde{z}^t_{f,\chi'} \neq 0 \\
        & \implies z_{f,\chi'} \neq 0 \\
        & \overset{\eqref{ThmB}}{\implies} \dim_F H^1_f(K, V_{f,\chi', \ell})=1 \\
        & \overset{\AJ \text{ inj.}}{\implies} \dim_{K(\chi)} \epsilon_{f,\chi} \CH^{k+t}(X)_{0}=1.
    \end{align*} 
\end{proof}

\section*{Appendix: Conventions for Heegner points}

Following \cite{grossHP}, any Heegner point on $X_0(N)$ for $K$ can be described as $(\cA, \fn):=(\C/\fa \lra \C/\fa\fn^{-1})$ for any integral ideal $\fa$ such that $[\fa]=\cA$. The identification of $Y_0(N)$ with $\Gamma_0(N)\backslash\cH$ is obtained as follows. Given a point $\tau\in \cH$, assign to it $(\C/\langle 1, \tau \rangle, 1/N + \langle 1, \tau \rangle)$. Conversely, given a cyclic $N$-isogeny $E\lra E'$, write $E=\C/L$ and $E'=\C/L'$. Up to homothety, we may assume $L\subset L'$ and $\C/L \lra \C/L'$ is the identity map on covering spaces. We can then choose an oriented basis $L=\langle \omega_2, \omega_1\rangle$ such that $\tau=\omega_1/\omega_2\in \cH$ and $L'=\langle \omega_2/N, \omega_1\rangle$. The point $(E\lra E')$ of $Y_0(N)$ then corresponds to $\Gamma_0(N)\tau$. The point corresponding to the Heegner point $(\cA, \fn)$ will be denoted by $z_{\cA, \fn}\in \cH$. Since $z_{\cA, \fn}\in K$ and $Nz_{\cA, \fn}\in K$, the point satisfies a quadratic equation $az_{\cA, \fn}^2 + bz_{\cA, \fn} +c = 0$ with $\gcd(a,b,c)=1$, $D=b^2-4ac$, and $N\mid a$. We have $\bar{z}_{\cA,\fn}=z_{\bar{\cA}, \bar{\fn}}$. If $\sigma_\cB\in \Gal(H/K)$ corresponds to $\cB\in \Cl_K$ under the Artin map, then $z_{\cA, \fn}^{\sigma_\cB}=z_{\cA\cB^{-1}, \fn}$. Letting $p\mid N$, $p\oh_K=\fp\bar{\fp}$, $\fn=\fp^k \fm$ and $\fn'=\bar{\fp}^k \fm$, we have $w_p(z_{\cA, \fn})=z_{\cA [\bar{p}^k], \fn'}$. These are the conventions used in \cite[p. 663]{nekovar}.

These conventions differ slightly from the ones in \cite{GZ}. We begin by recalling that $\oh_K=\langle 1, \frac{-1+\sqrt{D}}{2} \rangle$. The ideal $\fn$ is determined by some $\beta\in \Z/2N\Z$ such that $\beta^2\equiv D\pmod{4N}$ and $\fn=\langle N, \frac{\beta+\sqrt{D}}{2} \rangle$. We have $N(\fn)=N$ and $\fn^{-1}=\bar{\fn}N^{-1}=\langle 1, \frac{-\beta+\sqrt{D}}{2N} \rangle$. Consider a quadratic equation of the form
\begin{equation}\label{quad}
AX^2+BX+C=0, \quad A>0, \quad B^2-4AC=D, \quad N\mid A, \quad B\equiv \beta \pmod{2N},
\end{equation}
and consider the integral ideal $\fa=\langle A, \frac{B+\sqrt{D}}{2} \rangle$. Observe that $N(\fa)=A$, $\fa^{-1}=\bar{\fa}A^{-1}=\langle 1, \frac{-B+\sqrt{D}}{2A} \rangle$, and $\fa\fn^{-1}=\langle \frac{A}{N}, \frac{B+\sqrt{D}}{2} \rangle$. Gross and Zagier define $\tau_{\cA, \fn}=(-B+\sqrt{D})/2A\in \cH$ to be the solution to \eqref{quad} that lies in $\cH$. Then $\fa^{-1}=\langle 1, \tau_{\cA, \fn} \rangle$. Observe moreover that 
\[
\fa^{-1}\bar{\fn}^{-1}=\fa^{-1}\fn N^{-1}=(\fa\fn^{-1})^{-1}N^{-1}=\overline{\fa\fn^{-1}}NA^{-1}N^{-1}=\overline{\fa\fn^{-1}}A^{-1}=\left\langle \frac{1}{N}, \tau_{\cA, \fn} \right\rangle.
\]
We deduce that 
\[
\Gamma_0(N)\tau_{\cA, \fn}=\Gamma_0(N)z_{\bar{\cA}, \bar{\fn}}=\Gamma_0(N)\bar{z}_{\cA, \fn}.
\]
In particular, we obtain
\begin{equation}\label{galact}
\tau_{\cA, \fn}^{\sigma_{\cB}}=z_{\bar{\cA}, \bar{\fn}}^{\sigma_{\cB}}=z_{\bar{\cA}\cB^{-1}, \bar{\fn}}=\tau_{\cA\cB, \fn}.
\end{equation}
This is coherent with the statement \cite[p. 235]{GZ} that ``$\Gal(H/K)\simeq \Cl_K$ acts by multiplication on $\cA$''. Letting $p\mid N$, $p\oh_K=\fp\bar{\fp}$, $\fn=\fp^k \fm$ and $\fn'=\bar{\fp}^k \fm$, we have 
\begin{equation}\label{ALact}
w_p(\tau_{\cA, \fn})=w_p(z_{\bar{\cA}, \bar{\fn}})=z_{\bar{\cA}[\fp^k], \fp^k\bar{\fm}}=\tau_{\cA [\bar{\fp}^k], \bar{\fp}^k \fm}=\tau_{\cA [\bar{\fp}^k], \fn'}.
\end{equation}
Note that the cycles $Z_{\cA}$ and $\bar{Z}_{\cA}$ sit above the Heegner point 
\[
P_{\fn}^{\fa}=(\C/\fa^{-1} \lra \C/\fa^{-1}\fn^{-1})=z_{\bar{\cA}, \fn}=\tau_{\cA, \bar{\fn}}=(A^\fa \lra A^\fa/A^\fa[\fn]).
\]

Equation \eqref{ALact} is consistent with the claim \cite[p. 243]{GZ} about Atkin--Lehner, but inconsistent with the claim in \cite[p. 236 i)]{GZ}. It is moreover claimed in \cite[top of p. 243]{GZ} that $\tau_{\cA, \fn}^{\sigma_{\cB}}=\tau_{\cA\cB^{-1}, \fn}$. This contradicts \eqref{galact} and is the reason that all sums in the archimedean calculation of \cite{GZ} are over ideal classes $\cA_1$ and $\cA_2$ satisfying $\cA_1\cA_2^{-1}=\cA$, when the sums should in fact be taken over ideal classes $\cA_1$ and $\cA_2$ satisfying $\cA_1^{-1}\cA_2=\cA$. This does not affect the results of \cite{GZ} because there is no infinite order Hecke character to deal with, but it does matter in our setting. In the same vein, it seems that \cite[2nd equality bottom p. 142]{zhang} should contain $\cA^{-1}$ instead of $\cA$. Once again, this does not affect the results of \cite{zhang}, since there is no infinite order Hecke character. 


\end{document}